\documentclass[twoside,a4paper,11pt]{amsart}

\usepackage[symbol]{footmisc}
\usepackage{epigraph}
\usepackage{geometry}
\usepackage{amsmath}
\usepackage{amssymb}
\usepackage{amsthm}
\usepackage{mathrsfs}
\usepackage{mathtools}
\usepackage{esint}
\usepackage{graphicx}
\usepackage{color}
\usepackage{hyperref}
\usepackage{cleveref}
\usepackage{extarrows}
\usepackage{enumitem}
\usepackage[alphabetic]{amsrefs}
\usepackage[english]{babel}

\renewcommand{\H}{\mathbb{H}}

\newcommand{\N}{\mathbb{N}}

\newcommand{\R}{\mathbb{R}}
\renewcommand{\S}{\mathbb{S}}

\newcommand{\cL}{\mathcal{L}}

\newcommand{\cS}{\mathcal{S}}
\newcommand{\cU}{\mathcal U}

\newcommand{\cR}{\mathcal{R}}
\newcommand{\cT}{\mathcal{T}}
\newcommand{\cV}{\mathcal{V}}

\newcommand{\sE}{\mathscr{E}}

\newcommand{\sL}{\mathscr{L}}
\newcommand{\sM}{\mathscr{M}}
\newcommand{\sT}{\mathscr{T}}

\newcommand{\bE}{\mathbf{E}}

\newcommand{\Imm}{\mbox{\rm Im}\,}

\renewcommand{\exp}{\mbox{\rm exp}\;\!}
\newcommand{\Lip}{\mbox{\rm Lip}}

\newcommand{\bcup}{\bigcup}

\newcommand{\tens}{\otimes}

\newcommand{\sm}{\setminus}
\newcommand{\lgl}{\langle}
\newcommand{\rgl}{\rangle}
\newcommand{\pa}{\partial}
\newcommand{\con}{\subset}
\newcommand{\na}{\nabla}

\newcommand{\id}{{\rm id}} 

\newcommand{\ep}{\varepsilon} 
\newcommand{\ga}{\gamma}
\newcommand{\be}{\beta}
\newcommand{\al}{\alpha}
\newcommand{\de}{\delta}

\newcommand{\ze}{\zeta}

\newcommand{\ro}{\rho}
\newcommand{\si}{\sigma}
\newcommand{\te}{\theta}
\newcommand{\De}{\Delta}
\newcommand{\Ga}{\Gamma}
\newcommand{\La}{\Lambda}

\newcommand{\Om}{\Omega}
\newcommand{\vp}{\varphi}

\newcommand{\var}{\mbox{\rm Var}}

\DeclarePairedDelimiter\scal{\langle}{\rangle}

\definecolor{darkgreen}{rgb}{0.0, 0.56, 0.0}
\definecolor{firebrick}{rgb}{0.7, 0.13, 0.13}

\theoremstyle{plain}
\newtheorem{thm}{Theorem}[section] 

\theoremstyle{plain}

\theoremstyle{plain}
\newtheorem{prop}[thm]{Proposition}

\theoremstyle{plain}
\newtheorem{lemma}[thm]{Lemma}

\theoremstyle{plain}
\newtheorem{cor}[thm]{Corollary}

\theoremstyle{definition}
\newtheorem{defn}[thm]{Definition} 

\theoremstyle{definition}
\newtheorem{remark}[thm]{Remark}

\theoremstyle{definition}

\title[Convergence of elastic flows of curves into manifolds]{Convergence of elastic flows of curves into manifolds}
\author{Marco Pozzetta}
\address{Dipartimento di Matematica e Applicazioni, Universit\`a di Napoli Federico II, Via Cintia, Monte S. Angelo, 80126 Napoli, Italy.}
\email{marco.pozzetta@unina.it}
\date{\today}

\makeatletter
\newtheorem*{rep@theorem}{\rep@title}
\newcommand{\newreptheorem}[2]{%
	\newenvironment{rep#1}[1]{%
		\def\rep@title{#2 \ref{##1}}%
		\begin{rep@theorem}}%
		{\end{rep@theorem}}}
\makeatother
\newreptheorem{theorem}{Theorem}

\begin{document}

\begin{abstract}
	For a given $p\in[2,+\infty)$, we define the $p$-elastic energy $\sE$ of a closed curve $\ga:\S^1\to M$ immersed in a complete Riemannian manifold $(M,g)$ as the sum of the length of the curve and the $L^p$--norm of its curvature (with respect to the length measure). We are interested in the convergence of the $(L^p,L^{p'})$--gradient flow of these energies to critical points. By means of parabolic estimates, it is usually possible to prove sub-convergence of the flow, that is, convergence to critical points up to reparametrizations and, more importantly, up to isometry of the ambient. Assuming that the flow sub-converges, we are interested in proving the smooth convergence of the flow, that is, the existence of the full limit of the evolving flow.\\
	We first give an overview of the general strategy one can apply for proving such a statement. The crucial step is the application of a \L ojasiewicz-Simon gradient inequality, of which we present a versatile version. Then we apply such strategy to the flow of $\sE$ of curves into manifolds, proving the desired improvement of sub-convergence to full smooth convergence of the flow to critical points. As corollaries, we obtain the smooth convergence of the flow for $p=2$ in the Euclidean space $\R^n$, in the hyperbolic plane $\H^2$, and in the two-dimensional sphere $\S^2$. In particular, the result implies that such flow in $\R^n$ or $\H^2$ remains in a bounded region of the space for any time.
\end{abstract}

\maketitle


\noindent{\small\textbf{MSC Codes (2020):} 53E40, 35R01, 46N20.}


\medskip

\section{Introduction}

In this work we consider regular curves $\ga:\S^1\to M$ in complete Riemannian manifolds $(M^m,g)$, where $m\ge 2$ is the dimension of $M$ and will be usually omitted. The reader not used to Differential Geometry can safely assume that $M$ is the Euclidean space at this level. Fix $p\in[2,+\infty)$. If a curve $\ga$ is sufficiently regular, say of class $W^{2,p}$, we can define its $p$-elastic energy by setting
\[
\sE(\ga) \coloneqq \int_{\S^1} 1 + \frac{1}{p}|k|^p\,ds,
\]
where $k$ is the curvature of $\ga$ and $ds=|\ga'|dx$ is the length measure. More precisely, if $\tau = |\ga'|^{-1}\pa_x\ga$ is the unit tangent vector of $\ga$ and $D$ is the connection on $M$, then $k=D_\tau \tau$. In case $(M,g)$ is the Euclidean space $\R^m$, then one classically recovers $k=\pa_s^2 \ga$, where $\pa_s \coloneqq |\ga'|^{-1}\pa_x$ is the arclength derivative.

In the last years a considerable interest has been devoted towards this kind of energies. In this paper we are interested in studying some variational aspects of $\sE$ and in particular we will investigate some properties of a gradient flow of this energy. In order to explain this concept to the non-expert reader, let us assume for the moment that $M=\R^m$ with the Euclidean metric and everything is smooth. We define the first variation functional $\delta \sE$ at a given curve $\ga$ to be the operator
\[
\de\sE[\vp] \coloneqq \frac{d}{d\ep}\bigg|_{\ep=0} \sE(\ga+\ep\vp),
\]
where $\vp$ is a vector field along $\ga$. Explicit calculations usually lead to expressions like
\[
\de\sE[\vp] = \scal{\scal{V(\ga),\vp}},
\]
where $\scal{\scal{\cdot,\cdot}}$ here is some duality defining the action of the vector $V(\ga)$ on $\vp$. Then the $\scal{\scal{\cdot,\cdot}}$-gradient flow of $\sE$ is a function $\ga:[0,T)\times \S^1\to \R^m$ solving the evolution equation $\pa_t\ga(t,x) = - V(\ga(t,x))$. Moreover an initial datum $\ga(0,\cdot)=\ga_0(\cdot)$ is given, and we understand that the equation is satisfied in the classical sense. It is clear that different representations of the variation $\de\sE$ define different driving velocities $\pa_t\ga$, and thus different gradient flows of $\sE$.
In this paper we study the gradient flow defined by the $(L^p,L^{p'})$-duality $\scal{\scal{\cdot,\cdot}}= \scal{\cdot,\cdot}_{L^{p'},L^p}$. We shall see in \Cref{sec:Manifolds} how to explicitly calculate and define such gradient flows when $M$ is an arbitrary Riemannian manifold.

For a given gradient flow, a number of questions can be investigated, starting from the existence and uniqueness of a solution once an initial datum is given. In the case of geometric evolution equations, as the energy functional and the velocity of the flow are independent of the parametrization of the curves, uniqueness is always understood up to reparametrization (see \Cref{rem:Riparametrizzazioni} for additional comments). Short time existence and uniqueness results have been studied in the literature, mainly in the case $p=2$, starting from \cite{PoldenThesis} and \cite{HuPo99}. However, these evolution equations can be seen as parabolic evolution equations in the unknown given by the parametrization of the curve, and thus we will refer to general results like the one in \cite{MaMa12} for short time existence and uniqueness results in the framework of smooth curves.

Our study concerns the long time behavior of the solution of the gradient flow. One can hope that the solution $\ga(t,\cdot)$ admits a limit, for example in some $C^k$-topology, as $t\to T^-$, where $T\in(0,+\infty]$ is the maximal time of existence of the solution. In such a case the limit should be a curve $\ga_\infty$ which is a critical point for the energy $\sE$, as the flow stops at time $T$. The description of the long time behavior is quite often not an easy task, especially in case of evolution equations where high order space-derivatives appear. Nevertheless, it is often possible to prove very strong estimates on the solution that are uniform in time, and this is done by means of parabolic techniques. Let us say that $M=\R^m$ and $p=2$, then it is known that these bounds lead to the conclusion that the flow sub-converges, that is, there are a sequence of times $t_n\nearrow T^-$ and a sequence of points $p_n\in\R^m$ such that the sequence $\ga(t_n,\cdot)-p_n$ converges to a critical point curve $\ga_\infty$ in $C^k$ for any $k$, up to reparametrization. This has been studied in \cite{PoldenThesis} and then a complete proof is given in \cite{DzKuSc02}; in \cite{DaSp17} and \cite{DaLaLiPoSp18} the same conclusion is proved for the flow taking place in the hyperbolic plane $\H^2$ and in the unit $2$-sphere $\S^2$ respectively.

However, the sole sub-convergence cannot tell anything about the full limit as $t\to T^-$, and actually it does not prevent from the possibility that two different sequences $\ga(t_n,\cdot),\ga(\tau_n,\cdot)$ with $t_n,\tau_n\to T^-$ converge to different critical points of $\sE$, always up to reparametrization and, more importantly, isometry of the ambient. The sub-convergence does not imply that the flow remains in a compact region for any time either. In this paper we formalize and apply a method firstly appeared in \cite{ChFaSc09} for promoting the sub-convergence of a flow to the existence of the full limit as $t\to T^-$. In \cite{ChFaSc09} the authors apply these techniques to the Willmore flow of closed surfaces. The key ingredient to run the argument is a \L ojasiewicz-Simon gradient inequality for the energy functional under consideration. Such an inequality estimates the difference in energy between a chosen critical point and points sufficiently close to it in terms of some norm of the first variation functional of the energy (\Cref{cor:LojaFunctionalAnalytic}). As the norm of the first variation functional coincides with the norm of the velocity of the gradient flow, by its very definition, this furnishes an additional inequality that eventually can imply the full convergence of the flow. This method has been successfully applied in \cite{DaPoSp16} for proving the full convergence of the elastic flow of open curves subject to clamped boundary conditions (see also \cite{RuSp20}); we will borrow important notations from \cite{DaPoSp16}. We remark that the functional analytic tool used here and in those works, namely the \L ojasiewicz-Simon-type inequality, is ultimately based on the important results contained in \cite{Ch03}. The idea of using these inequalities for proving convergence of solutions to parabolic equations goes back to the seminal paper of Simon \cite{Si83}, that contributed to add his name to the inequality. The first appearances of the \L ojasiewicz-Simon gradient inequality are contained in \cite{Loja63, Loja84}.

Let us conclude by mentioning some related results in this area. Very recently many contributions have been given to the theory of gradient flows of networks, both in the context of elastic flows and of the curve shortening flow. Roughly speaking, a network is given by a suitable union of open immersed curves joined at their endpoints, possibly prescribing the angles that such curves must define at their junctions. Results about short and long time behavior of these flows are contained in \cite{DaLiPo19, DaPo14, GaMePl19a, GaMePl19b, MaNoPl17, MaNoPl19}. It would be interesting to apply our methods also to unsolved problems in the context of these flows, as well as for high order flows of higher dimensional manifolds like in \cite{Ma02}. We finally mention that different ideas appeared in the literature for proving the full convergence of a flow; we recall for example \cite{NoOk17} and \cite{MuSp20}, that are based either on a priori hypotheses on the critical points of the energy functional or on a known classification of such critical points.

\medskip
\subsection*{Main results and comments}

Let $p\ge 2$ be fixed. We will show in \Cref{sec:Manifolds} that we can define the gradient flow of $\sE$ with respect to the $(L^p,L^{p'})$-duality by the following evolution equation.

For a given smooth curve $\ga_0:\S^1\to M$, we say that $\ga:[0,T)\times \S^1\to M$ is the solution of the gradient flow of $\sE$ with datum $\ga_0$ if it classically satisfies the equation
\begin{equation}\label{eq:DefnFlowPGenericoCOPY}
\begin{cases}
\pa_t \ga = -\left( \na^2 |k|^{p-2} k + \frac{1}{p'} |k|^p k -k + R(|k|^{p-2} k,\tau)\tau \right) & \mbox{ on } [0,T)\times \S^1,\\
\ga(0,\cdot)=\ga_0(\cdot)   & \mbox{ on } \S^1,
\end{cases}
\end{equation}
where we understand that $|k|^{p-2}\equiv1$ in case $p=2$. In the above equation $\tau(t,x), k(t,x)$ respectively are the tangent vector and the curvature at the point $x$ of the curve $\ga(t,\cdot)$ at time $t$, $R$ is the Riemann tensor of $(M,g)$ (see \Cref{sec:Manifolds}), and $\na$ is the normal connection along $\ga(t,\cdot)$, that is
\[
\na \phi \coloneqq D_\tau  \phi - g(D_\tau \phi, \tau)\tau,
\]
for any vector field $\phi:[0,T)\times \S^1\to TM$ such that $\phi (t,x) \in T_{\ga(t,x)}M$. Observe that in case $p=2$ and $M=\R^m$ is the Euclidean space, then the flow reduces to the classical evolution
\begin{equation*}
\begin{cases}
\pa_t \ga = -\left( \na^2  k + \frac{1}{2} |k|^2 k -k  \right) & \mbox{ on } [0,T)\times \S^1,\\
\ga(0,\cdot)=\ga_0(\cdot)   & \mbox{ on } \S^1,
\end{cases}
\end{equation*}
and $\na$ is just the composition of the normal projection along $\ga$ with the arclength derivative:
\[
\na \phi = \pa_s \phi  - \scal{\pa_s \phi, \tau} \tau,
\]
for $\phi$ as above. Without loss of generality, by Nash Theorem, we will always assume that $(M,g)$ is smoothly isometrically embedded in the Euclidean space $\R^n$, for some $n$ sufficiently large. In this way it is meaningful to say that a sequence of curves $\{\ga_l\}_{l\in\N}$ converges in $C^k$ to a curve $\ga$. We mention that the evolution equation \eqref{eq:DefnFlowPGenericoCOPY} is called $L^2$-gradient flow of $\sE$ by several authors independently of $p$, instead of $(L^p,L^{p'})$-gradient flow.

The following theorem is our main result, which is \Cref{thm:ConvergenceManifoldsMAIN} in the following.

\begin{thm}\label{thm:MainIntro}
	Suppose that $(M,g)$ is an analytic complete Riemannian manifold endowed with an analytic metric tensor $g$. Let $p\ge2$ and suppose that $\ga:[0,+\infty)\times \S^1\to M$ is a smooth solution of \eqref{eq:DefnFlowPGenericoCOPY}. Suppose that there exist a sequence of isometries $I_n:M\to M$, a sequence of times $t_n\nearrow+\infty$, and a smooth critical point $\ga_\infty:\S^1\to M$ of $\sE$ such that
	\[
	I_n\circ\ga(t_n,\cdot) -  \ga_\infty (\cdot) \xrightarrow[n\to\infty]{} 0 \qquad \mbox{ in } C^{m}(\S^1),
	\]
	for any $m\in\N$, up to reparametrization. If $p>2$ assume also that  $|k_{\ga_\infty}(x)|\neq0$ for any $x$.
	
	Then the flow $\ga(t,\cdot)$ converges in $C^m(\S^1)$ to a critical point as $t\to+\infty$, for any $m$ and up to reparametrization.
\end{thm}

We remark that for high order geometric evolution equations as the ones considered here, no maximum principles hold, and then it is not possible to conclude that the flow always stays in a bounded region of the space by means of comparison arguments. This qualitative information however follows once it is known that the flow does converge, for example as a consequence of a result like \Cref{thm:MainIntro}.

\medskip

In the first part of this paper we outline the structure of the general strategy leading to the proof of a result like \Cref{thm:MainIntro}. To this scope we take as a model the elastic flow for $p=2$ in the Euclidean space, which is completely studied in \cite{MaPo20}, whose aim is presenting the crucial general steps of the method in the most simple way (see also \cite{MaPlPo20} for a survey on elastic flows).
We believe that these methods could be applied for proving convergence of high order flows out of their sub-convergence from a unified point of view. That is why in \Cref{sec:Rn} we state and prove the main abstract tool, which is the \L ojasiewicz-Simon gradient inequality, in a purely functional analytic setting, without reference to the elastic flows of curves. As in \cite{ChFaSc09} and \cite{DaPoSp16}, such inequality follows from the results of \cite{Ch03}, but we rephrase it in ready-to-use version that can be applied to different gradient flows; this is the content of \Cref{cor:LojaFunctionalAnalytic}.

In the second part of the work we carry out the above mentioned strategy in the case of the flow in \eqref{eq:DefnFlowPGenericoCOPY}, and then we prove \Cref{thm:MainIntro}. This will lead us to prove a number of properties on the first and second variations of the energy, and useful parabolic estimates for the $p$-elastic flow that might be interesting in themselves.

In this paper, we shall only study the flow of the $p$-elastic energy for exponent $p\ge 2$; the case $p \in [1,2)$ would cause a number of technical problems, even starting from the very definition of the evolution equation, and thus the understanding of the behavior of the flow for such exponents $p$ remains an open problem.

\medskip

Let us conclude by stating some consequences of \Cref{thm:MainIntro}. As we already mentioned, sub-convergence of the flow for $p=2$ has been proved in the literature in some ambient spaces; thus we can apply \Cref{thm:MainIntro} to get the following consequence.

\begin{cor}\label{cor:CorollaryIntro}
	Let $p=2$ and suppose that $\ga:[0,+\infty)\times \S^1\to M$ is a smooth solution of \eqref{eq:DefnFlowPGenericoCOPY}. Assume that $M$ is either the Euclidean space $\R^m$, the hyperbolic plane $\H^2$, or the standard $2$-sphere $\S^2$.\\	
	Then $\ga$ smoothly converges as $t\to+\infty$ to a critical point $\ga_\infty$ of $\sE$ up to reparametrization. In particular, the flow stays in a compact set of $M$ for any time.
\end{cor}

The proof of \Cref{cor:CorollaryIntro} follows from \Cref{thm:MainIntro} by the fact that sub-convergence of the flow has been proved in \cite{DzKuSc02} if $M=\R^m$, in \cite{DaSp17} if $M=\H^2$, and in \cite{DaLaLiPoSp18} if $M$ is a $2$-sphere.

\medskip

\Cref{thm:MainIntro} is clearly applicable in the special case where the isometries $I_n$ appearing in the statement are the identity on $M$ for any $n$. This is precisely the case in which one already knows that the flow remains in a compact subset of $M$. Such a hypothesis is automatically satisfied if the ambient manifold $M$ is compact. Therefore we can state the following.

\begin{cor}\label{cor:CorollaryIntro2}
	Suppose that $(M,g)$ is an analytic compact Riemannian manifold endowed with an analytic metric tensor $g$. Let $p=2$ and suppose that $\ga:[0,+\infty)\times \S^1\to M$ is a smooth solution of \eqref{eq:DefnFlowPGenericoCOPY}. Suppose that $\|\ga(t,\cdot)\|_{C^m(\S^1)} \le C(m)<+\infty$ for any $t\ge 0$.	
	Then the flow $\ga(t,\cdot)$ converges in $C^m(\S^1)$ to a critical point as $t\to+\infty$, for any $m$ and up to reparametrization.
\end{cor}

\Cref{cor:CorollaryIntro2} follows from the fact that, since $M$ is compact, uniform bounds in $C^m$ for any $m$ guarantee the existence of a sequence of times $t_n\nearrow+\infty$ and of a critical point $\ga_\infty$ such that $\ga(t_n,\cdot) \to \ga_\infty$ in $C^{m}(\S^1)$ for any $m$ as $n\to+\infty$ up to reparametrization. Hence we can apply \Cref{thm:MainIntro} and \Cref{cor:CorollaryIntro2} follows.

Let us remark that under the assumptions that $M$ is an analytic compact manifold with an analytic metric $g$ and $p=2$, the uniform bounds in $C^m$ in the hypotheses of \Cref{cor:CorollaryIntro2} are likely to be true in general. Indeed, one should be able to derive the usual parabolic estimates in the same fashion of \cite{DzKuSc02}, thus getting the desired uniform bounds.

\medskip

Let us conclude with a few comments.

\begin{remark}
	\Cref{thm:MainIntro} remains true if one considers the analogously defined flow of the energy $\int \lambda + \tfrac1p |k|^p$ for any $\lambda>0$.
	We believe that the statement of \Cref{cor:CorollaryIntro} continues to hold true if $M$ is any hyperbolic space $\H^m$ or a sphere $\S^m$ with $m\ge2$, but the sub-convergence of the flow has not been proved explicitly in the literature in these ambients, up to the knowledge of the author. More generally, it is likely that \Cref{cor:CorollaryIntro} remains true whenever $(M,g)$ is a homogeneous manifold, that is a Riemannian manifold such that the group of isometries acts transitively on $M$; indeed, in such a case, one should be able to prove sub-convergence of the flow for $p=2$ exactly as in \cite{DzKuSc02}.
	
	We remark that the hypothesis of $(M^m,g)$ being of bounded geometry is not sufficient to imply that the solution of the flow converges. Indeed, in Appendix C we construct a simple example of a solution to the flow of the elastic energy with $p=2$ in an analytic complete surface in $\R^3$ that does not (sub)converge.
\end{remark}

\begin{remark}
	We remark that, even if \Cref{cor:CorollaryIntro} implies that the solution of the elastic flow for $p=2$ in $\R^m$ stays in a compact region, this result does not tell anything about the size of the compact set containing the flow. We believe it is a nice open question to quantify, if possible, the size of such compact set depending on the given initial datum $\ga_0$. We also mention that a related problem which is still open, up to the author's knowledge, is to prove or disprove the Huisken's conjecture stating that for the flow \eqref{eq:DefnFlowPGenericoCOPY} for $p=2$ in $\R^2$, if the datum $\ga_0$ does not intersect a closed halfplane, then the solution $\ga(t,\cdot)$ is never completely contained in such halfplane.
\end{remark}

\begin{remark}
	As mentioned above, the use of a \L ojeasiewicz-Simon inequality for proving the convergence of a gradient flow is today quite understood. A rather clear account of the method applied to the general setting of gradient flows in Hilbert or metric spaces is contained in the recent \cite{ChMi18} and \cite{HaMa19}, respectively. It is also clear from \cite[Theorem 3.27]{HaMa19} that the validity of a \L ojeasiewicz-Simon inequality implies a precise rate of convergence to equilibrium, which is, at worst, polynomial in time, depending on the exponent $\theta$ appearing in the established inequality, cf. \Cref{cor:LojaFunctionalAnalytic}. The way we shall derive the \L ojeasiewicz-Simon inequality does not provide an explicit value for the exponent $\theta$, which is just known to belong to the interval $(0,\tfrac12]$. However, exploiting the argument in the proof of the aforementioned \cite[Theorem 3.27]{HaMa19}, one expects that also for the geometric gradient flows studied in this paper the convergence to critical points is realized at a rate at least polynomial in time, i.e., that the $L^2$ distance between a suitable parametrization of the flow and the limit critical point decays at least polynomially in time.
\end{remark}

\medskip
\subsection*{Organization} In \Cref{sec:Rn} we outline the general strategy one can apply for proving that sub-convergence of a flow can be promoted to full convergence, and we prove the \L ojasiewicz-Simon gradient inequality we shall use (\Cref{cor:LojaFunctionalAnalytic}). In \Cref{sec:Manifolds} and \Cref{sec:ConvergenceFinal} we prove in detail that the sub-convergence of the flow of the $p$-elastic energy of curves in manifolds imply the smooth convergence of the flow, completing the proof of \Cref{thm:MainIntro}. In Appendix A we collect the computations providing a general expression for the second variation of the $p$-elastic energy. Appendix B is devoted to the proof of some parabolic estimates about the $p$-elastic flow that are needed in \Cref{sec:ConvergenceFinal}. Appendix C contains an example of solution to the elastic flow with exponent $p=2$ which does not converge.

\medskip
\subsection*{Acknowledgments} I am grateful to Carlo Mantegazza for his interest in this work and many useful conversations and suggestions. I also thank Matteo Novaga for having suggested me to study this problem.

\bigskip

\section{Elastic flow in the Euclidean space:\\ outline of the proof and functional analytic methods}\label{sec:Rn}

This section is devoted to the presentation of the general method for improving the sub-convergence of a flow to its full convergence. We consider here regular curves $\ga:\S^1\to \R^n$ and the classical elastic energy with exponent $p=2$, that is
\[
\sE(\ga) = \int_{\S^1} 1+\frac12 |k|^2 \, ds,
\]
if $\ga \in C^2$, where $k$ is the curvature vector of $\ga$ and $ds=|\ga'(x)|dx$ denotes integration with respect to the arclength. We will call $\tau=|\ga'(x)|^{-1}\ga'(x)$ the unit tangent vector of $\ga$ and we will denote by $\pa_s = |\ga'(x)|^{-1}\pa_x$ the differentiation with respect to the arclength. Recall therefore that $k=\pa_s^2 \ga=\pa_s\tau$. Finally, by analogy with the study we will carry out for regular curves into manifolds, if $\vp:\S^1\to\R^n$ is a differentiable vector field we define $\na\vp = \pa_s \vp - \scal{\pa_s \vp, \tau}\tau$, that is the normal projection of the arclength derivative of $\vp$. In case of risk of confusion, a subscript $\ga$ will be added to a geometric object understanding it refers to the curve $\ga$. Moreover, we will denote with the symbol $\ga^\top$ (resp. $\ga^\perp$) the projection onto the tangent space (resp. normal space) of $\ga$, i.e., if $\vp:\S^1\to\R^n$ is any field, then $\ga^\top\vp (x) = \scal{\vp(x),\tau(x)}\tau(x)$ (resp. $\ga^\perp \vp(x) = \vp(x) - \ga^\top\vp(x)$).
Observe that $\sE$ is a geometric functional, in the sense that the energy of a curve is independent of its parametrization; as we shall see, this fact will have several consequences.

\medskip

In the following we just want to collect the most crucial ingredients, focusing on the proof of an abstract \L ojasiewicz-Simon inequality. A posteriori, this part will be a particular case of the theory developed in \Cref{sec:Manifolds} and \Cref{sec:ConvergenceFinal}. As the study of these gradient flows into Riemannian manifolds (\Cref{sec:Manifolds} and \Cref{sec:ConvergenceFinal}) is more involved, for the convenience of the reader we preferred to first present the significant steps of the proof here in the case of curves in $\R^n$ and exponent $p=2$. In the case of the gradient flow of the $2$-elastic energy in $\R^n$ the strategy can actually be simplified and we refer to \cite{MaPo20} for such a case.

\bigskip

\subsection{First and second variations}
The strategy starts from a careful study of the properties of the first and second variations of $\sE$. To this aim we need to define precisely the Banach spaces of vector fields $\vp:\S^1\to\R^n$ along a curve $\ga$ defining variations of the given curve.

\begin{defn}
	Let $\ga:\S^1\to\R^n$ be a regular curve of class $H^4$. For $k \in \N$ we define
	\[
	H(\ga)^{k,\perp} \coloneqq \left\{  \vp\in W^{k,2}(\S^1,\R^n) \,\,:\,\,  \scal{\tau(x),\vp(x)}=0 \,\,\,\mbox{a.e. }x \right\},
	\]
	where we understand that $W^{0,2}(\S^1,\R^n)=L^2(\S^1,\R^n)$. Also we denote $H(\ga)^{0,\perp}$ by $L^2(\ga)^\perp$.
\end{defn}

\begin{remark}
	If $k\ge 2$ and $\ga:\S^1\to\R^n$ is a regular curve of class $H^4$, there exists $\ro>0$ such that $\ga+\vp$ is still a regular curve for any $\vp \in B_\ro(0)\con H^k(\S^1,\R^n)$ and for any $\vp \in B_\ro(0)\con H^k(\ga)^{k,\perp}$. In the following we will always assume that $\ro=\ro(\ga)$ is such that variations $\ga+\vp$ are regular curves for any $\vp$ as before.
\end{remark}

Adopting the notation of \cite{DaPoSp16}, it is worth to introduce the following notation.

\begin{defn}
	Let $\ga:\S^1\to\R^n$ be a regular curve of class $H^4$. For suitable $\ro>0$ we define
	\[
	\begin{split}
	E:B_\ro(0)\con H(\ga)^{4,\perp}\to\R \qquad\qquad E(\vp)&\coloneqq\sE(\ga+\vp),\\
	\bE:B_\ro(0)\con H^{4}(\S^1,\R^n)\to\R \qquad\quad \bE(\vp)&\coloneqq\sE(\ga+\vp).
	\end{split}
	\] 
\end{defn}

In this way we can classically see first and second variations of $\sE$ as elements of dual spaces; the reason for distinguishing between normal or arbitrary fields along a curve will also be clear soon. If $\ga$ is fixed, we have
\[
\de \bE : B_\ro(0) \con H^4(\S^1,\R^n) \to (H^4(\S^1,\R^n) )^\star 
\qquad\qquad
\de \bE(\vp)[\psi] = \frac{d}{d\ep}\bigg|_0 \sE(\ga+\vp+\ep\psi),
\]
and similarly
\[
\de E : B_\ro(0) \con H(\ga)^{4,\perp} \to (H(\ga)^{4,\perp} )^\star 
\qquad\qquad
\de E(\vp)[\psi] = \frac{d}{d\ep}\bigg|_0 \sE(\ga+\vp+\ep\psi).
\]
We refer to \Cref{prop:FirstVariation0} and \Cref{prop:VarBoldE} for the general computation of the first variation functionals. In the case we are considering, for $\vp,\psi \in B_\ro(0)\con  H^4(\S^1,\R^n)$ one obtains
\begin{equation}\label{eq:RefFirstVar}
\de \bE (\vp) [\psi]= \int_{\S^1} \left\lgl \na^2_{\ga+\vp} k_{\ga+\vp} + \frac12 |k_{\ga+\vp}|^2k_{\ga+\vp} - k_{\ga+\vp}, \psi\right\rgl\, ds_{\ga+\vp}
\eqqcolon \left\lgl \na_{L^2(ds_{\ga+\vp})} \bE (\vp) , \psi \right\rgl_{L^2(ds_{\ga+\vp})},
\end{equation}
and the very same formula holds for $E$ and $\vp,\psi \in H(\ga)^{4,\perp}$. In particular we write
\[
\na_{L^2(dx)} \bE(\vp) = |\ga'+\vp'|\na_{L^2(ds_{\ga+\vp})} \bE (\vp)  = |\ga'+\vp'| \left( \na^2_{\ga+\vp} k_{\ga+\vp} + \frac12 |k_{\ga+\vp}|^2k_{\ga+\vp} - k_{\ga+\vp} \right),
\]
and
\[
\na_{L^2(dx)} E (\vp) = \ga^\perp \na_{L^2(dx)} \bE (\vp).
\]
Setting $\vp=0$ we see that $\na_{L^2(ds_{\ga})} \bE (0)$ is normal along $\ga$ and then
\[
\de \bE(0)[\psi] = \de E (0)[\ga^\perp\psi] = \de \bE(0)[\ga^\perp\psi],
\]
that is the variation of $\sE$ at $\ga$ only depends on normal vector fields along $\ga$. This is ultimately due to the geometric nature of the functional $\sE$ and highlights the fact that $\sE$ is degenerate with respect to variations defined by tangential fields along $\ga$. This is the true reason why one introduces the distinction between normal fields along $\ga$ and general fields.

As we will be interested in invertibility properties of the variations of $\sE$, we will only need to evaluate the second variation of $\sE$ along normal fields, ruling out the tangential degeneracy of the functional. Therefore we define the operator $\cL=\de^2 E(0):H(\ga)^{4,\perp}\to (H(\ga)^{4,\perp})^\star$ by
\[
\cL(\vp)[\psi] = \frac{d}{d\ep}\bigg|_0 \frac{d}{d\eta}\bigg|_0 \sE\left( \ga +\eta\vp + \ep\psi \right) \qquad\quad\forall\,\vp,\psi\in H(\ga)^{4,\perp}.
\]
Observe that $\cL$ is symmetric, that is $\cL(\vp)[\psi]=\cL(\psi)[\vp]$ for any $\vp,\psi\in H(\ga)^{4,\perp}$.

\medskip

Now the first key observation is the fact that for a fixed regular curve $\ga$ of class $H^4$ and suitable $\vp$, the operators $\de E (\vp)$ and $\de \bE(\vp)$ actually belong to $(L^2(\ga)^\perp)^\star$ and $(L^2(\S^1,\R^n))^\star$ respectively, as they are represented by the $L^2$ fields $\na_{L^2(dx)} E(\vp)$ and $\na_{L^2(dx)} \bE(\vp)$ respectively. Moreover, the same holds for the second variation functional $\cL$, and more precisely we can state the following. For the general theory of compact and Fredholm operators we refer to \cite[Section 19.1]{HormanderIII}.

\begin{lemma}\label{lem:FredholmRn}
 	Let $\ga:\S^1\to\R^n$ be a smooth regular curve and $\vp\in H(\ga)^{4,\perp}$. The operator $\cL(\vp)$ is an element of $(L^2(\ga)^{\perp})^\star$ represented by the pairing
	\begin{equation*}
	\cL(\vp)[\psi] = \left\lgl |\ga'| \, \left( \na^4 \vp  + \Om(\vp) \right) , \psi \right\rgl_{L^2(dx)} 
	\qquad\qquad\forall\,\psi \in H(\ga)^{4,\perp} , 
	\end{equation*}
	where $\Om:H(\ga)^{4,\perp}\to L^2(\ga)^{\perp}$ is a compact operator.
	
	Moreover the operator $\na^4:H(\ga)^{4,\perp}\to L^2(\ga)^{\perp}$ is Fredholm of index zero, and then so is the operator
	\[
	\cL:H(\ga)^{4,\perp}\to (L^2(\ga)^{\perp})^\star.
	\]
\end{lemma}

\begin{proof}
	The calculation of $\cL$ for curves in $\R^n$ can be easily carried out explicitly, and we refer to \Cref{prop:SecondVariation} for the computation in the general case of curves in manifolds (see also \cite{DaPoSp16} for the case of $\R^n$).
	The complete claim will follow from \Cref{lem:FredholmnessP2}.	
\end{proof}

The second classical ingredient needed for obtaining the \L ojasiewicz--Simon gradient inequality is the analiticity of the energy functional and of its first variation. In our case, we have that for a fixed smooth regular curve $\ga:\S^1\to\R^n$ and suitable $\ro>0$ the maps
\[
\begin{split}
	E:B_\ro(0) \to \R \qquad\qquad\qquad\quad E(\vp)&=\sE(\ga+\vp),\\
	\de E: B_\ro(0) \to (L^2(\ga)^{\perp})^\star \qquad\qquad \de E(\vp)&=\na_{L^2(dx)} E(\vp),
\end{split}
\]
are analytic. We refer to \cite[Lemma 3.4]{DaPoSp16} for a detailed proof of this fact. Keeping in mind the framework we just described, and the main properties we found on first and second variaitons, we can now present the \L ojasiewicz-Simon inequality.

\bigskip

\subsection{An abstract \L ojasiewicz--Simon gradient inequality}

In this subsection we prove a general statement collecting some conditions under which a \L ojasiewicz-Simon inequality holds for a given energy functional. This result does not depend on whether we are considering curves in $\R^n$ or in manifolds, actually it is stated at a purely functional analytic level for an abstract energy functional, and we will use it both in the study of the flow in $\R^n$ and into manifolds.

We need to recall the functional analytic setting of \cite{Ch03}.
We assume that $V$ is a Banach space, $U\con V$ is open, and $E:U\to\R$ is a map of class $C^2$. We denote by $\sM:U\to V^\star$ be the Fréchet first derivative, and $\sL:U\to L(V;V^\star)$ the Fréchet second derivative. We also assume that $0\in U$. Let us denote
\[
\cL \coloneqq \sL(0) \in L(V;V^\star), \qquad\qquad V_0\coloneqq\ker \cL\con V.
\]

We recall that a closed subspace $S\con V$ is said to be complemented if there exists a continuous projection $P:V\to V$ such that $\Imm P = S$. A continuous projection is a linear continuous map $P:V\to V$ such that $P\circ P = P$. In such a case, we denote by $P^\star:V^\star\to V^\star$ the adjoint projection.

\begin{prop}[{\cite[Corollary 3.11]{Ch03}}]\label{cor:Chill}
	In the above notation, assume that $E$ is analytic and $0$ is a critical point of $E$, i.e., $\sM(0)=0$. Assume that $V_0$ is finite dimensional, and therefore complemented with a projection map $P$. Moreover there exists a Banach space $W\hookrightarrow V^\star$ such that:
	\begin{enumerate}[label=(\roman*)]
		\item $\sM: U \to W$ is $W$--valued and analytic; \label{IT.1}		
		\item $P^\star(W)\con W$; \label{IT.2}
		\item $\cL(V) = \ker P^\star \cap W$. \label{IT.3}
	\end{enumerate}
	Then there exist $C,\ro>0$ and $\te\in(0,\frac12]$ such that
	\begin{equation*}
	|E(\psi)-E(\vp)|^{1-\te}\le C \| \sM(\psi) \|_{W},
	\end{equation*}
	for any $\psi \in B_\ro(\vp)$.
\end{prop}

\Cref{cor:Chill} is exactly \cite[Corollary 3.11]{Ch03} with $X=V$ and $Y=W$ therein. Indeed one can check that the hypotheses of \cite[Corollary 3.11]{Ch03}, that include Hypotheses 3.2 and 3.4 in \cite{Ch03}, reduce to the assumptions considered here in \Cref{cor:Chill}.
Applying \Cref{cor:Chill} we can prove the following.

\begin{cor}\label{cor:LojaFunctionalAnalytic}
	Let $E:B_{\ro_0}(0)\con V \to \R$ be an analytic map, where $V$ is a Banach space and $0$ is a critical point of $E$. Suppose that $W \equiv  Z^\star \hookrightarrow V^\star$ is a Banach space with $V\hookrightarrow Z$, and that $\sM: B_\ro(0) \to W$ is $W$-valued and analytic. Suppose also that $\cL\coloneqq\sL(0) \in L(V,W)$ and $\cL: V\to W$ is Fredholm of index zero.
	
	Then the hypotheses of \Cref{cor:Chill} are satisfied. In particular there exist $C$, $\ro>0$ and $\te\in(0,\frac12]$ such that
	\begin{equation*}
	|E(\psi)-E(0)|^{1-\te}\le C \| \sM(\psi) \|_{W},
	\end{equation*}
	for any $\psi \in B_\ro(0)$.
\end{cor}

\begin{proof}
	By hypothesis $V_0 \coloneqq \ker \cL $ is finite dimensional, and thus it is closed and complemented with a projection $P:V\to V$ such that $\Imm P = V_0$. Moreover \ref{IT.1} of \Cref{cor:Chill} is satisfied by assumption.
	
	We can write that $V=V_0 \oplus V_1$ where $V_1=\ker P$. If $P^\star:V^\star\to V^\star$ is the adjoint projection, we see that also $V^\star= V_0^\star \oplus V_1^\star$ and 
	\[
	V_0^\star = \Imm P^\star, \qquad V_1^\star = \ker P^\star.
	\]	
	Let us introduce the canonical isometric injection $J_0:Z\to Z^{\star\star}$. Let us call $J:V\to (Z)^{\star\star}$ the restriction of $J_0$ to $V$. Hence
	\[
	J:V\to J(V) \con Z^{\star\star}.
	\]
	We claim that $\cL:V\to W$ satisfies that if
	\[
	\cL^\star: W^\star \to V^\star
	\]
	is the adjoint of $\cL$, then
	\begin{equation}\label{eq:QuasiSelfAdjoint}
	\cL^\star \circ J=\cL .
	\end{equation}
	Indeed, using that $\cL$ is symmetric because it is a second Fréchet derivative, for any $\vp,\psi\in V$ and $F\coloneqq J(\psi) \in J(V)\con Z^{\star\star}$ we find
	\[
	(\cL^\star \circ J )(\psi) [\vp] = \cL^\star(F)[\vp]=F(\cL\vp)
	= (J(\psi))(\cL\vp) = (\cL\vp)[\psi] =\cL(\psi)[\vp].
	\]
	
As a general consequence of the fact that $\cL$ is Fredholm of index zero, we have that
\[
{\rm dim} \ker \cL= {\rm dim} \ker \cL^\star,
\]
where ${\rm dim}(\cdot)$ denotes the dimension of a finite dimensional space. Indeed, index zero means that ${\rm dim}\ker \cL = \dim  {\rm coker} \cL$, where we split $W$ as
\[
W= \Imm \cL\oplus {\rm coker} \cL,
\]
and ${\rm coker} \cL$ is finite dimensional.
Therefore $W^\star = (\Imm \cL)^\star  \oplus ({\rm coker} \cL)^\star$. And since $\ker \cL^\star = (\Imm\cL)^\perp = ({\rm coker} \cL)^\star$, we conclude that $\dim \ker \cL^\star = \dim ({\rm coker} \cL)^\star = \dim {\rm coker} \cL = \dim \ker \cL$.

We claim that
\begin{equation}\label{eq:3}
	J(\Imm P) = \ker \cL^\star \cap J(V).
\end{equation}
Indeed by \eqref{eq:QuasiSelfAdjoint} we see that
\[
\ker \cL = \ker (\cL^\star\circ J) = J^{-1}(\ker\cL^\star).
\]
Applying $J$ on both sides we get $J(\Imm P) =\ker\cL^\star\cap J(V)$, that is \eqref{eq:3}. Since $\Imm P = \ker \cL$ and $J$ is injective, we have $\dim\ker\cL={\rm dim} (J(\Imm P)) = {\rm dim} (\ker \cL^\star \cap J(V))$. Since ${\rm dim} \ker \cL = {\rm dim}\ker \cL^\star$, it follows that $\ker \cL^\star \cap J(V)= \ker \cL^\star$, and then
\[
J(\Imm P) = \ker \cL^\star.
\]
Therefore, recalling that $V^{\star\star}\hookrightarrow W^\star$ and that $W\hookrightarrow V^\star$, we get
\begin{equation}
\begin{split}
	(\ker\cL^\star)^\perp &= \left\{ w \in W \,\,:\,\, \scal{f,w}_{W^\star,W}=0 \,\,\,\forall\, f \in J(\Imm P) \right\} \\
	&= \left\{ w \in W \,\,:\,\, \scal{J(v),w}_{W^\star,W}=0 \,\,\,\forall\, v \in \Imm P \right\} \\
	&= \left\{ w \in W \,\,:\,\, \scal{w,v}_{V^\star,V}=0 \,\,\,\forall\, v \in \Imm P \right\} \\
	&= (\Imm P)^\perp \cap W.
\end{split}
\end{equation}
Finally, as $\Imm\cL$ is closed, this implies
\[
\begin{split}
\Imm\cL&= (\ker \cL^\star)^\perp = (\Imm P)^\perp \cap W 
= \left\{ 
f \in V^\star \,\,\,:\,\,\, \scal{f,Pv}_{V^\star,V}=0 \,\,\,\forall\,v \in V
 \right\} \cap W \\
 &= \ker P^\star \cap W,
\end{split}
\]
and then \ref{IT.3} of \Cref{cor:Chill} is verified.

	
	We are just left with proving \ref{IT.2}, that is, $P^\star(Z^\star)\con Z^\star$.
	Observe that if we check that $P^\star(Z^\star\cap V_0^\star) \con Z^\star \cap V_0^\star$, then we are done, indeed we would get
	\[
	P^\star(Z^\star)= P^\star ( Z^\star \cap V_0^\star \oplus Z^\star \cap V_1^\star  ) = P^\star (Z^\star\cap V_0^\star) \con Z^\star \cap V_0^\star \con Z^\star.
	\]
	Now if $f_0\in Z^\star \cap V_0^\star$, writing any $\vp\in V$ as $\vp=\vp_0\oplus \vp_1 \in V_0\oplus V_1$, we get
	\[
	P^\star(f_0)[\vp]=f_0(P\vp)=f_0(\vp_0) = f_0(\vp_0)+f_0(\vp_1) = f_0(\vp),
	\]
	indeed $f_0(\vp_1)=(P^\star f_0)(\vp_1)=f_0(P\vp_1)=f_0(0)=0$. Hence we proved that $P^\star f_0=f_0$ for any $f_0\in Z^\star \cap V_0^\star$, and thus got that $P^\star(Z^\star\cap V_0^\star) \con Z^\star \cap V_0^\star$.	
\end{proof}

Let us mention that an equivalent result has been proved independently in the recent \cite{Ru20}.

\bigskip

\subsection{Convergence of the elastic flow in the Euclidean space}

If $\ga:\S^1\to\R^n$ is a smooth critical point of $\sE$, the analysis on the first and second variations, together with \Cref{lem:FredholmRn}, implies that we can apply \Cref{cor:LojaFunctionalAnalytic} with the spaces $V=H(\ga)^{4,\perp}$ and $Z=L^2(\ga)^\perp$. This gives that for some $\ro>0$ it holds the \L ojasiewicz--Simon inequality
\[
\left| \sE(\ga+\vp) - \sE(\ga)  \right|^{1-\te} \le C \|\de E (\vp) \|_{(L^2(\ga)^\perp)^\star} \le C \|\na_{L^2(dx)} E (\vp) \|_{L^2(dx)},
\]
for any $\vp \in B_\ro(0) \con H(\ga)^{4,\perp}$. Now, using the geometric nature of $\sE$, the above inequality can be easily generalized to fields in $H^{4}(\S^1,\R^n)$ with suitably small norm. More precisely, one has that for some $\si>0$ for any $\psi \in B_\si(0)\con H^{4}(\S^1,\R^n)$ there is $\vp \in B_\ro(0)\con H(\ga)^{4,\perp}$ such that the curves $\ga+\psi$ and $\ga+\vp$ coincide up to reparametrization (see \Cref{lem:ReparametrizationManifold}). As $|\na_{L^2(dx)} E (\vp)|\le |\na_{L^2(dx)} \bE (\vp)|$ for any $\vp \in  B_\ro(0)\con H(\ga)^{4,\perp}$ and both $\sE$ and $\na_{L^2(dx)}\bE(\psi)$ are invariant under reparametrization, we find that
\begin{equation}\label{eq:LojaRn}
\left| \sE(\ga+\psi) - \sE(\ga)  \right|^{1-\te} \le  C \|\na_{L^2(dx)} \bE (\psi) \|_{L^2(dx)},
\end{equation}
for any $\psi \in B_\si(0)\con H^{4}(\S^1,\R^n)$.

\medskip

Following the ideas of \cite{Si83}, \cite{ChFaSc09}, and \cite{DaPoSp16}, we can now see how to use \eqref{eq:LojaRn} in order to derive the convergence of the gradient flow of $\sE$. Let us recall that by gradient flow of $\sE$ we mean here the evolution equation
\begin{equation}\label{eq:DefFlowRn}
\begin{cases}
\pa_t \ga = -\na^2 k - \frac12 |k|^2k + k,\\
\ga(0,\cdot)=\ga_0(\cdot),
\end{cases}
\end{equation}
for a given smooth curve $\ga_0:\S^1\to\R^n$, where one looks for a smooth solution $\ga:[0,T)\times \S^1\to\R^n$. In this context, short time existence and sub-convergence of the flow as $t\to+\infty$ have been proved, and more precisely we can state the following.

\begin{thm}[\cite{PoldenThesis}, \cite{HuPo99}, {\cite[Theorem 3.2]{DzKuSc02}}]\label{thm:PoldenDzKuSc}
	For a given smooth curve $\ga_0:\S^1\to\R^n$, a global solution $\ga:[0,+\infty)\times \S^1\to\R^n $ to the flow defined in \eqref{eq:DefFlowRn} exists and it is unique. Moreover there exist a sequence of times $t_j\to+\infty$ and a sequence of points $p_j\in\R^n$ such that the immersions
	\[
	\ga(t_j,\cdot) - p_j,
	\]
	converge in $C^m$ to a critical point $\ga_\infty$ of $\sE$, up to reparametrization, for any $m\in \N$.
\end{thm}

%

We can now illustrate the argument that leads to the convergence as $t\to+\infty$ of the solution of this gradient flow. Here we sketch the proof we will use for the flow of curves in manifolds in \Cref{sec:ConvergenceFinal}; however, as already mentioned, in this case where the ambient is the Euclidean space, the proof can be simplified and we refer the reader to \cite{MaPo20} for such a proof.

Let $\ga_0$ be fixed, and let $\ga,\ga_\infty,t_j,p_j$ be given by \Cref{thm:PoldenDzKuSc}. Without loss of generality we assume that $\ga_\infty$ is parametrized with constant speed.
Fix $m\ge8$ and let $\ep\in(0,1)$ to be chosen. By \Cref{thm:PoldenDzKuSc} there exists $p_{j_0}$ such that
\[
\|\bar \ga(t_{j_0},\cdot)- p_{j_0} - \ga_\infty (\cdot) \|_{C^{m}(\S^1)} \le \ep,
\]
where $\bar\ga(t,\cdot)$ is the constant speed reparametrization of $\ga(t,\cdot)$.
We want to show that if $\ep$ is sufficiently small, then actually $\bar{\ga}$ smoothly converges.
We rename $\ga_0(\cdot) = \bar{ \ga}(t_{j_0},\cdot)$. By short time existence and uniqueness results (\Cref{thm:LocalExistencePGenerico}, \cite{PoldenThesis}) there exists a solution $\tilde\ga:[0,+\infty)\times \S^1\to \R^n$ of
\[
\begin{cases}
\pa_t \tilde\ga = -\na^2 k_{\tilde\ga} - \frac12 |k_{\tilde\ga}|^2k_{\tilde\ga} + k_{\tilde\ga},\\
\tilde\ga(0,\cdot)=\ga_0(\cdot).
\end{cases}
\]
We denote by $\hat\ga$ the constant speed reparametrization of $\tilde\ga$.
For $\ep$ sufficiently small we can write $\hat\ga$ as a variation of $\ga_\infty$. More precisely, there is some maximal $T'\in(0,+\infty]$ such that for any $t\in[0,T')$ there exists $\psi_t\in B_\si(0)\con H^4(\S^1,\R^n)$ such that $\hat\ga(t,\cdot)= \ga_\infty(\cdot) + \psi_t(\cdot)$, where $\si$ is as in \eqref{eq:LojaRn}.

Suppose by contradiction that $T'<+\infty$. Suitable parabolic estimates give that for any $t<T'$, knowing that the flow $\hat \ga$ remains close in $W^{4,p}$ to the fixed $\ga_\infty$, the norm $\|k_{\hat \ga}\|_{W^{l,2}}$ is bounded by a term depending on $\|k_{\ga_0}\|_{W^{l,2}}$, on $l$, and on $\|k_{\ga_\infty}\|_{W^{2,2}}$ for any $l\in\N$. This fact is technical but rather classical in the theory of parabolic geometric equations and the details are carried out in \Cref{prop:StimaParabolica}. In particular, as $\ga_0$ is close to $\ga_\infty$ in $C^m$, these parabolic estimates applied for $l=m-2$ together with Sobolev embeddings eventually imply that
\[
\sup_{[0,T')} \|\hat \ga(t,\cdot) - \ga_0(\cdot)\|_{C^{m-3}(\S^1)} \le C(\ga_\infty),
\]
and the key fact here is that the constant on the right deos not depend on $\ep$.
By triangular inequality we then deduce
\begin{equation}\label{eq:BoundEstimateRn}
\sup_{[0,T')} \|\hat \ga(t,\cdot) - \ga_\infty(\cdot)\|_{C^{m-3}(\S^1)} \le C(\ga_\infty).
\end{equation}
Now we consider the evolution of
\[
H(t)\coloneqq (\sE(\hat\ga(t,\cdot)) - \sE(\ga_\infty) )^\te,
\]
where $\te$ is the \L ojasiewicz-Simon exponent of \eqref{eq:LojaRn}.
%
%
Using \eqref{eq:LojaRn} it is immediate to estimate that
\begin{equation}\label{eq:DifferentialInequalityRn}
	-\frac{d}{dt}H(t) \ge C(\ga_\infty) \|\pa^\perp_t\hat\ga\|_{L^2(dx)},
\end{equation}
where $\pa^\perp_t\hat\ga$ is just the projection of the velocity $\pa_t \hat \ga$ onto the normal space of $\hat \ga$. Using \eqref{eq:DifferentialInequalityRn} and the fact that $\|\pa^\perp_t\hat\ga\|_{L^2(ds_{\hat\ga})} = \|\pa_t\tilde\ga\|_{L^2(ds_{\tilde\ga})}$, one can show that the parametrization of $\tilde\ga$ does not degenerate, that is, the speed $|\pa_x\tilde\ga(t,x)|$ is bounded away from zero uniformly in time, and it is actually close to the speed of $\hat\ga$.

Therefore one shows that
\[
\begin{split}	
\|\hat\ga(t,\cdot)-\ga_\infty\|_{L^2(dx)} 
&\le C(\ga_\infty)\ep^\te,
\end{split}
\]
for $t\in[0,T')$. Suitable interpolation inequalities (see \Cref{rem:Interpolation}) together with \eqref{eq:BoundEstimateRn} imply that
\begin{equation*}
\begin{split}	
\|\hat\ga(t,\cdot)-\ga_\infty\|_{W^{4,2}}
&\le C  \|\hat\ga(t,\cdot)-\ga_\infty\|_{C^5}^{\alpha} \|\hat\ga(t,\cdot)-\ga_\infty\|_{L^2(dx)}^{1-\alpha} \\
& \le  C \|\ga_0-\ga_\infty\|^{\te(1-\alpha)}_{C^2(\S^1)} \le C \ep^{\te(1-\alpha)}
\end{split}
\end{equation*}
for $t\in[0,T')$ and some $\alpha\in(0,1)$. Hence if $\ep$ is sufficiently small this implies that $\|\psi_t\|_{W^{4,2}}\le \tfrac12\si$ for any $t\in[0,T')$, contradicting the maximality of $T'$.

Hence we have that for any $t\in[0,+\infty)$ the flow $\hat\ga(t,\cdot)$ can be written as $\ga_\infty + \psi_t$ for some uniformly bounded fields $\psi_t$; in particular the evolution $\hat\ga(t,x)$ stays in a compact set for any $t$. Once boundedness in space is achieved, the above estimates eventually imply that $\hat\ga$ smoothly converges to a translation of $\ga_\infty$, and then the same holds for the original flow $\ga$.

\medskip

As a result of this argument, or as a particular case of \Cref{thm:ConvergenceManifoldsMAIN}, we can state the following.

\begin{thm}\label{thm:ConvergenceRn}
	For a given smooth curve $\ga_0:\S^1\to\R^n$, a global solution $\ga:[0,+\infty)\times \S^1\to\R^n $ to the flow defined in \eqref{eq:DefFlowRn} exists, it is unique, and it smoothly converges as $t\to+\infty$ to a critical point $\ga_\infty$ of $\sE$ up to reparametrization. In particular, the flow stays in a compact set of $\R^n$ for any time.
\end{thm}

The variational approach leading to the above theorem and the abstract tool \Cref{cor:LojaFunctionalAnalytic} suggest that we can try to extend the result to the gradient flow of elastic functionals of curves immersed into Riemannian manifolds. The rest of the paper is, in fact, devoted to prove rigorously that, under suitable hypotheses, the sub-convergence of the gradient flow of the $p$-elastic energy can be improved to full convergence of the flow also on Riemannian manifolds. This will fill the gaps in the heuristic proof presented above in the case of the flow in the Euclidean space for $p=2$.

\bigskip

\section{Elastic flows into manifolds:\\ first and second variations and \L ojasiewicz-Simon inequalities}\label{sec:Manifolds}

In this section we analyze the first and second variation of the $p$-elastic energy of curves in manifolds and we establish the \L ojasiewicz-Simon gradient inequality for such energies. Let us start with a few definitions.

In the following $(M^m,g)$ will be a fixed complete Riemannian manifold of dimension $m\ge 2$. By Nash Theorem we can assume without loss of generality that $(M^m,g)\hookrightarrow \R^n$ isometrically and that a smooth curve into $M$ is a smooth regular curve $\ga:\S^1\to\R^n$ with $\ga(x)\in M$ for any $x$. The exponential map of $M$ will be denoted by $\exp:TM\to M$.

Having identified $M$ with a subset of $\R^n$, we will denote by $\scal{\cdot,\cdot}$ both the Euclidean product and the metric on $M$. If $V$ is a vector field in $\R^n$ and $x\in M$, by $M^\top V(x)$ (resp. $M^\perp V(x)$) we denote tangent (resp. normal) projection of $V$ on the tangent space of $T_xM$ (resp. the normal space of $T_xM^\perp$). We denote by $\pa_v$ a directional derivative in $\R^n$, and by $D$ the Levi-Civita connection on $M$, so that
\[
(D_v X) (x)= M^\top(\pa_v X)(x),
\]
for tangent fields $v,X$ on $M$. For a smooth curve $\ga$ into $M$, we also define $\pa_s=|\ga'|^{-1}\pa_x$, $\tau=\pa_s\ga$, and $ds=|\ga'|dx$. A subscript $\ga$ will be added in case of confusion if more than a curve is considered.

The symbol $\na$ will denote the normal connection along a curve $\ga$ in $M$, that is
\[
(\na_v \phi)(x) = (M^\top-\ga^\top)(\pa_v \phi)(x)=(M^\top-\ga^\top M^\top)(\pa_v \phi)(x) = D_v\phi(x) -\scal{D_v\phi(x),\tau(x)}\tau(x),
\]
for any smooth field $\phi \in TM\cap(T\ga)^\perp$. Unless otherwise stated we will always denote
\[
 \ga^\perp\coloneqq M^\top-\ga^\top,
\]
that is $\ga^\perp$ is the normal projection along $\ga$ as a submanifold of $M$.
We will also write $\na \coloneqq\na_{\tau}$, in analogy with the notation used for curves in $\R^n$.

\begin{remark}
	If $\vp,\psi\in C^1(\S^1,\R^n)$ are fields such that $\vp,\psi\in TM\cap (T\ga)^\perp$ for a given $\ga\in W^{4,p}_{imm}(\S^1,M)$, then
	\[
	\int_{\S^1} \scal{\na\vp,\psi}\,ds =\int \scal{\ga^\perp M^\top|\ga'|^{-1}\pa_x\vp,\psi}|\ga'|\,dx = - \int_{\S^1}  \scal{\vp,\na \psi}\,ds,
	\]
	that is, integration by parts holds for normal fields with respect to the normal connection $\na$ and the arclength measure $ds$.
\end{remark}

The curvature vector of a curve $\ga$ into $M$ is
\[
k= D_{\tau}\tau \,\,\,\in \,T\ga^\perp\con TM.
\]

We adopt the following convention on the Riemann tensor $R$ of $M$. If $X,Y,Z,W$ are tangent fields on $M$ then
\[
R(X,Y,Z,W)=\scal{R(Z,W)Y,X},
\]
where $R(Z,W)Y=D_ZD_W Y - D_WD_Z Y -D_{[Z,W]}Y$. We will sometimes use basic facts in Riemannian geometry, for which we refer to \cite{DoCarmo}.

\begin{remark}
	The curvature $k$ of $\ga$ into $M\con\R^n$ is the geodesic curvature of the curve on $M$. In particular we can also write that
	\[
	k = M^\top(\pa_s\tau )= M^\top(\pa^2_s \ga).
	\]
\end{remark}

Let us also define the Sobolev spaces
\[
W^{k,p}(\S^1,M)\coloneqq\left\{ \ga:\S^1\to\R^n \,\,\big|\,\, \ga\in W^{k,p}(\S^1,\R^n), \,\,\ga(x)\in M \mbox{ a.e. }x \right\},
\]
for $k\in\N$ with $k\ge 1$ and $p\in[1,+\infty)$. For $k\ge2$ and $p>1$ we denote by $W_{imm}^{k,p}(\S^1,M)$ the open subset of $W^{k,p}(\S^1,M)$ of immersions, that is the subset of functions $\ga$ such that $|\ga'|\ge c(\ga)>0$. The symbols introduced above for smooth curves will be analogously used for Sobolev curves sufficiently regular. Spaces $L^p(\S^1,M)$ are defined analogously.

For a fixed exponent $p\in(1,+\infty)$ and $\ga:\S^1\to M$ an immersion of class $W^{2,p}(\S^1,M)$ we define
\[
\sE(\ga) = \int_{\S^1} 1 + \frac1p |k|^p\,ds.
\]

Now we need to define the Banach spaces of vector fields along curves that we will use to produce variations of a given curve.

\begin{defn}
	If $\ga$ is a fixed immersion of class $C^1$, for $k\in \N$ we define
	\[
	T(\ga)^{k,p}\coloneqq\left\{ \vp \in W^{k,p}(\S^1,\R^n) \,\,:\,\, \vp(p)\in T_{\ga(p)}M \,\,\forall\, p\in \S^1 \right\},
	\]
	\[
	T(\ga)^{k,p,\perp}\coloneqq\left\{ \vp \in W^{k,p}(\S^1,\R^n) \,\,:\,\, \scal{\vp,\pa_x\ga}\equiv0, \,\,\vp(p)\in T_{\ga(p)}M \,\,\forall\, p\in \S^1 \right\},
	\]
	for any $k\ge1$,
	and
	\[
	T(\ga)^{p}=T(\ga)^{0,p}\coloneqq\left\{ \vp\in L^p(\S^1,\R^n) \,\,:\,\,  \vp(p)\in T_{\ga(p)}M \mbox{ a.e. on } \S^1 \right\},
	\]
	\[
	T(\ga)^{p,\perp}=T(\ga)^{0,p,\perp}\coloneqq\left\{ \vp\in L^p(\S^1,\R^n) \,\,:\,\, \scal{\vp,\pa_x\ga}=0, \,\, \vp(p)\in T_{\ga(p)}M \mbox{ a.e. on } \S^1 \right\}.
	\]
\end{defn}
When nothing is specified, $L^p$-spaces are equipped with the Lebesgue measure. If for a given curve $\ga$ we want to employ the induced length measure on $\S^1$ we will specify $L^p(ds_\ga)$.

The following lemma shows that the spaces $T(\ga)^{k,p}$ do not depend on the embedding of $M$ into $\R^n$.

\begin{lemma}
	Let $\ga$ be a fixed immersion of class $C^1$. Let $k\in \N$ and $\vp\in T(\ga)^{k,p}$. If there exists $\phi\in T(\ga)^{p}$ such that
	\[
	\int \scal{D^k_s\vp,D_s\psi}\,ds = - \int\scal{\phi,\psi}\,ds,
	\]
	for any $\psi :\S^1\to\R^n$ of class $C^\infty$ such that $\psi(x)\in T_{\ga(x)}M$, then $\vp \in T(\ga)^{k+1,p}$.
\end{lemma}

\begin{proof}
	Let us first prove by induction that for any $n\in \N$ with $n\ge 1$ if $\alpha \in T(\ga)^{n,p}$ then
	\begin{equation}\label{eq:DnPartialn}
	D^n_s\alpha = \pa_s^n \alpha + \Om_n(\al,...,\pa_s^{n-1}\alpha),
	\end{equation}
	where $\Om_n$ is smooth in its entries, it only depends on $M$, and $ \Om_n(\al,...,\pa_s^{n-1}\al) \in W^{1,p}$. In fact for $n=1$ we have
	\[
	D_s\alpha = \pa_s\alpha + \scal{\alpha,\pa_s N_j} N_j,
	\]
	where $\{N_j\}$ is a local orthonormal frame of $TM^\perp$, and summation over $j$ is understood.
	Since
	\[
	\pa_s M^\top - M^\top \pa_s = - \pa_s N_j \tens N_j - N_j \tens \pa_s N_j,
	\]
	for $n\ge1$ we get
	\[
	\begin{split}
	D_s^{n+1}\alpha & = M^\top\pa_s \left( \pa_s^n \alpha + \Om_n(\alpha,...,\pa_s^{n-1}\alpha) \right) \\
	&= \pa_s \left( \pa_s^n\alpha - \scal{\pa_s^n\alpha,N_j}N_j \right)
	+ \pa_s (M^\top\Om_n) + (\pa_s N_j \tens N_j + N_j \tens \pa_s N_j)(\pa_s^n\alpha +\Om_n) \\
	&= \pa_s^{n+1}\alpha +\pa_s (M^\perp \Om_n)
	+ \pa_s (M^\top\Om_n) + (\pa_s N_j \tens N_j + N_j \tens \pa_s N_j)(\pa_s^n\alpha +\Om_n) \\
	&= \pa_s^{n+1}\alpha +\pa_s  \Om_n
	+ (\pa_s N_j \tens N_j + N_j \tens \pa_s N_j)(\pa_s^n\alpha +\Om_n),
	\end{split}
	\]
	that proves \eqref{eq:DnPartialn}.
	
	For $\Psi:\S^1\to\R^n$ of class $C^\infty$ we write $\Psi=\psi + \Psi^\perp$ where $\psi \in TM$.
	We have
	\[
	\scal{D^k_s \vp, \pa_s\Psi} = \scal{D^k_s \vp, D_s\psi} + \scal{\Psi,N_j}\scal{D^k_s\vp, \pa_s N_j} 
	= \scal{D^k_s \vp, D_s\psi} - \scal{\Psi, B(\tau, D^k_s\vp)}
	,
	\]
	where $\{N_j\}$ is a local orthonormal frame of $TM^\perp$. Hence
	\[
	\int \scal{D^k_s \vp, \pa_s \Psi} \,ds = \int -\scal{\phi,\psi} - \scal{\Psi, B(\tau, D^k_s\vp)} \,ds = -\int \scal{\phi + B(\tau,D^k_s\vp) , \Psi}\,ds,
	\]
	which shows that $D_s^k \vp \in T(\ga)^{1,p}$. And therefore by \eqref{eq:DnPartialn} also $\pa_s^k \vp \in W^{1,p}$.	
\end{proof}

Let us state here another simple lemma about the regularity of the objects we will deal with.

\begin{lemma}\label{lem:CompositionSobolev}
	Let $g \in W^{k,p}((0,1),B_r(0))$ with $B_r(0) \con \R^N$. Let $f:B_{2r}(0)\to \R$ be a bounded function of class $C^k$ with bounded continuous derivatives up to order $k$. Then $f\circ g \in W^{k,p}(0,1)$ and the operator
	\[
	W^{k,p}((0,1),B_r(0)) \ni \,\, g \mapsto f\circ g  \,\,\in W^{k,p}(0,1)
	\]
	is of class $C^k$.
\end{lemma}

\begin{proof}
	Since $W^{k,p}((0,1),B_r(0)) \con C^{k-1,\alpha}((0,1),B_r(0))$ we see that $f\circ g \in W^{k-1,p}(0,1)$. Now for a function $\tilde{g} \in C^\infty((0,1),B_{\frac32 r}(0))$ the chain rule gives
	\[
	(f\circ\tilde{g})^{(k)} = (\na^k f)(\tilde{g})[\tilde{g}', \tilde{g}', ..., \tilde{g}']+
	P((\na^{k-1} f) (\tilde{g}), ... , (\na^2 f)(\tilde{g}), \tilde g', ...,  \tilde g^{(k-1)} ) +  \scal{(\na f)(\tilde{g}) , \tilde g^{(k)}},
	\]
	where $P$ is some polynomial.
	Considering a sequence $g_n \in C_c^\infty((0,1),B_{\frac32 r}(0))$ converging in $W^{k,p}$ to $g$ and thus also strongly in $C^{k-1}$ we see that
	\[
	(\na^k f)(g_n)[g_n', g_n', ..., g_n'] \to  (\na^k f)({g})[{g}', {g}', ..., {g}']
	\]
	uniformly, and
	\[
	\scal{(\na f)(g_n) ,  g_n^{(k)} } \to \scal{(\na f)({g}) ,  g^{(k)} }
	\]
	in $L^p$, and therefore $f\circ g \in  W^{k,p}(0,1)$.
	
	If now $g_h \in W^{k,p}((0,1),B_r(0))$ is a sequence converging to $g$ in $W^{k,p}$, and then in $C^{k-1}$, we have that $f\circ g_h \to f\circ g $ in $C^{k-1}$ and $(f\circ g_h)^{(k)}
	\to (f\circ g)^{(k)}$ in $L^p$ by the above formulas, and then $g\mapsto f\circ g$ is continuous between the corresponding Sobolev spaces. Since $f\in C^k$ with bounded derivatives, an analogous argument shows that $g\mapsto f\circ g$ is of class $C^k$.	
\end{proof}

\begin{cor}\label{cor:CompositionSobolevTangent}
	Let $\ga:\S^1\to M^m$ be a fixed regular curve of class $C^1$ and let $F:TM\to N^n$ a smooth map between manifolds. If $\vp \in T(\ga)^{k,p}$ then $F\circ \vp$ is of class $W^{k,p}$, in the sense that for any local chart $(U,\ze)$ on $N$ the map $\ze\circ F\circ \vp$ is of class $W^{k,p}$. Moreover for any local chart $(U,\ze)$ in $N$ the operator
	\[
	T(\ga)^{k,p} \ni\,\, \vp \mapsto \ze\circ F\circ \vp \,\,\in W^{k,p}((0,1),\R^n)
	\]
	is of class $C^\infty$.
\end{cor}

\begin{proof}
	For any local chart $(V,\xi)$ on $TM$, \Cref{lem:CompositionSobolev} implies that $\xi \circ \vp \in W^{k,p}(\S^1,\R^{2m})$. Since $\ze\circ F\circ \xi^{-1}$ is smooth and $\xi\circ \vp$ is bounded, we get that $\ze\circ F\circ \xi^{-1} \circ \xi\circ \vp$ is of class $W^{k,p}$. Smoothness of the operator $\vp\mapsto \ze\circ F\circ \vp$ follows by applying \Cref{lem:CompositionSobolev}.
\end{proof}

In \Cref{cor:CompositionSobolevTangent} a map fitting the hypothesis is the exponential map $\exp:TM\to M$. This leads to the following definition.

\begin{defn}\label{def:Variation}
	Let $\ga\in W_{imm}^{4,p}(\S^1,M)$. A map $\Phi:(-\ep_0,\ep_0)\times \S^1\to M$ is a variation of $\ga$ if
	\[
	\Phi(0,\cdot)=\ga(\cdot), \qquad \Phi(s,\cdot)\in W_{imm}^{4,p}(\S^1,M) \quad\forall\,s, \qquad
	\Phi(\cdot,x)\in W^{4,p}((-\ep_0,\ep_0),\R^n) \quad\forall\,x.
	\]
	In such a case we write that $\Phi\in \var(\ga)$ with variation field $\vp(x)\coloneqq\pa_\ep\Phi(0,x)$. If it also occurs that $\vp\in T\ga^\perp$, then we say that $\Phi$ is a normal variation and we write that $\Phi\in \var^\perp(\ga)$.
\end{defn}

Using the exponential map of $M$, we shall always use a typical construction of variations of a curve given a variation field. More precisely, suppose that $\vp \in T(\ga)^{4,p}$ for an immersed curve $\ga\in W^{4,p}(\S^1,M)$. We then define the variation
\begin{equation}
\Phi=\Phi(\ep,x):(-\ep_0,\ep_0)\times \S^1\to M \qquad\qquad \Phi(\ep,x)=\exp_{\ga(x)}(\ep\vp(x)),
\end{equation}
where $\exp_p:T_pM\to M$ is the exponential map of $M$. Since $\S^1$ is compact, the definition of $\Phi$ is well posed for $\ep_0$ small enough. It holds that
\[
\Phi(0,x)=\ga(x), \qquad \pa_\ep \Phi (0,x)= \vp(x).
\]
We also set $\ga_\ep(\cdot)=\Phi(\ep,\cdot)$. Finally for any $v\in T_pM$ we will denote by $\si_v:[0,l_v)\to M$ the geodesic in $M$ such that $\si(0)=p$ and $\si'(0)=v$. In this way we can write that
\[
\ga_\ep(x)=\Phi(\ep,x)=\si_{\ep\vp(x)}(1)=\si_{\vp(x)}(\ep).
\]
\Cref{cor:CompositionSobolevTangent} implies the following lemma.

\begin{lemma}\label{lem:Variation}
	Fix an immersed curve $\ga\in W^{4,p}(\S^1,M)$. Then there exist a radius $\ro(\ga)>0$ and $\ep_0(\ga)>1$ such that $\Phi(\ep,x)=\exp_{\ga(x)}(\ep\vp(x))$ is a variation of $\ga$ in the sense of \Cref{def:Variation} with variation field $\vp$ for any $|\ep|<\ep_0$ and any $\vp \in T(\ga)^{4,p}$ with $\|\vp\|_{W^{4,p}}\le \ro$.
\end{lemma}

As in \Cref{sec:Rn} we introduce the following functionals.

\begin{defn}
	Let $\ga\in W_{imm}^{4,p}(\S^1,M)$ be fixed and let $\ro$ be given by \Cref{lem:Variation}. For $\vp \in B_\ro(0)\con T(\ga)^{4,p}$ we define
	\[
	E: B_\ro(0)\con T(\ga)^{4,p,\perp} \to \R \qquad\qquad E(\vp)\coloneqq\sE(\Phi(1,\cdot)),
	\]
	and for $\vp \in B_\ro(0)\con T(\ga)^{4,p,\perp}$ we define
	\[
	\bE: B_\ro(0)\con T(\ga)^{4,p} \to \R \qquad\qquad \bE(\vp)\coloneqq\sE(\Phi(1,\cdot)),
	\]
	where $\Phi$ is the variation associated with the given field $\vp$.
\end{defn}

\bigskip

\subsection{First and second variations}

Let $\ga\in W_{imm}^{4,p}(\S^1,M)$ be fixed. We want to compute the variations of $E$ and $\bE$. For $\vp$ in the suitable domains of the two functionals we recall that the first variations are defined as
\[
\de\bE(\vp)\in (T(\ga)^{4,p})^\star \qquad\qquad \de\bE(\vp)[\psi] \coloneqq \frac{d}{ds}\bigg|_0 \bE(\vp+s\psi),
\]
\[
\de E(\vp)\in (T(\ga)^{4,p,\perp})^\star \qquad\quad \de E(\vp)[\psi]\coloneqq \frac{d}{ds}\bigg|_0 E(\vp+s\psi).
\]

Let us collect some computations first.

\begin{lemma}\label{lem:Variations}
	Let $\si\in W_{imm}^{4,p}(\S^1,M)$ and $\Phi\in\var(\si)$ with variation field $\vp$. Denote $\si_\ep(x)=\Phi(\ep,x)$. Then
	\begin{equation}\label{eq:Vards}
	\frac{d}{d\ep}\bigg|_0 ds_{\si_\ep} = \scal{\tau_\si, \pa_{s_\si}\vp} ds_{\si},
	\end{equation}
	
	\begin{equation}\label{eq:VarLungh}
	\frac{d}{d\ep}\bigg|_0 \int_{\S^1}\,ds_{\si_\ep} = -\int_{\S^1} \scal{\vp,k_\si}\,ds_\si, 
	\end{equation}
	
	\begin{equation}\label{eq:Vartau}
	\frac{d}{d\ep}\bigg|_0 \tau_{\si_\ep} = \pa_{s_\si}\vp -\scal{\pa_{s_\si}\vp,\tau_\si}\tau_\si,
	\end{equation}
	
	\begin{equation}\label{eq:VarK}
	(M^\top-\si^\top)\left(\frac{d}{d\ep}\bigg|_0 k_{\si_\ep} \right)= (M^\top-\si^\top)\pa_s((M^\top-\si^\top)\pa_s\vp) -\scal{\tau_\si,\pa_{s_\si}\vp}k_\si
	+ R(\vp^\perp,\tau_\si)\tau_\si,
	\end{equation}
	where $\vp^\perp=\si^\perp \vp$.
	
	\noindent If also $\Phi\in\var^\perp$, i.e. $\vp\in T\si^\perp$, then
	
	\begin{equation}\label{eq:Vardsperp}
	\frac{d}{d\ep}\bigg|_0 ds_{\si_\ep} = -\scal{k_\si,\vp}ds_\si,
	\end{equation}
	
	\begin{equation}\label{eq:Vartauperp}
	\frac{d}{d\ep}\bigg|_0 \tau_{\si_\ep} = \na \vp,
	\end{equation}
	
	\begin{equation}\label{eq:VarKPerp}
	(M^\top-\si^\top)\left(\frac{d}{d\ep}\bigg|_0 k_{\si_\ep}\right) = \na^2 \vp +\scal{\vp,k_\si}k_\si
	+ R(\vp,\tau_\si)\tau_\si.
	\end{equation}
\end{lemma}

\begin{proof}
	\Cref{eq:Vards} follows by a direct calculation. Then \eqref{eq:Vardsperp} follows by the fact that $\scal{\tau_{\si},\pa_x\vp}= -\scal{\pa_x\tau_{\si},\vp}$ for $\vp\in T\si^\perp$, and $\scal{\pa_{s_\si}^2\si,\vp}=\scal{M^\top(\pa^2_{s_\si} \si),\vp}$. Moreover for general variation field $\vp$ we have that
	\[
	\begin{split}
	\scal{\pa_{s_\si}\vp,\tau_\si} &= \left\lgl \pa_{s_\si} \left(\scal{\vp,\tau_\si}\tau_\si + \sum_2^m \scal{\vp,e_i}e_i\right) , \tau_\si \right\rgl = \pa_s(\scal{\vp,\tau_\si}) + \sum_2^m \scal{\vp,e_i}\scal{\pa_s(e_i\circ\si),\tau_\si} =\\
	&=\pa_s(\scal{\vp,\tau_\si}) + - \scal{\vp,k_\si},
	\end{split}
	\]
	where $\{\tau,e_2,...,e_m\}$ is a local orthonormal frame of $TM$. \Cref{eq:VarLungh} then follows by integration.
	
	\Cref{eq:Vartau} and \Cref{eq:Vartauperp} also follows by direct calculations and the definition of the normal connection $\na$.
	
	Now for $j=m+1,...,n$ let $N_j:U\to S^{n-1}$ be unit vector fields locally defined on a neighborhood of $\si(x)$ in $M$ such that $\{N_j\,:\,j=m+1,...,n\}$ is a local orthonormal frame of $(TM)^\perp$. Writing
	\[
	k_{\si_\ep}=D_{\tau_{\si_\ep}} \tau_{\si_\ep} = \pa_{s_{\si_\ep}} \tau_{\si_\ep} - \sum_j  \scal{ \pa_{s_{\si_\ep}} \tau_{\si_\ep}, N_j\circ \si_\ep}N_j\circ\si_\ep,
	\]
	we have that
	\[
	\begin{split}
	\frac{d}{d\ep}\bigg|_0 k_{\si_\ep} & = -\scal{\tau_\si,\pa_{s_\si}\vp}\pa_{s_\si}\tau_\si
	+\pa_{s_\si} \left(\pa_{s_\si}\vp -\scal{\pa_{s_\si}\vp,\tau_\si}\tau_\si\right)
	+\sum_j \scal{\tau_\si,\pa_{s_\si}\vp}\scal{\pa_{s_\si}\tau_\si,N_j\circ\si}N_j\circ\si +\\
	&\indent - \scal{\pa_{s_\si}\left(\pa_{s_\si}\vp -\scal{\pa_{s_\si}\vp,\tau_\si}\tau_\si\right), N_j\circ\si}N_j\circ\si
	- \scal{\pa_{s_\si}\tau_\si,\pa_\ep|_0(N_j\circ\si_\ep)}N_j\circ\si +\\
	&\indent - \scal{\pa_{s_\si}\tau_\si,N_j\circ\si}\pa_\ep|_0(N_j\circ\si_\ep).
	\end{split}
	\]
	Denoting by $S_{N_j}(v)\coloneqq-M^\top(\pa_v N_j)$ the shape operator of $M$ defined by $N_j$, we have that $M^\top[\pa_\ep|_0(N_j\circ\si_\ep)]=-S_{N_j}(\vp)$, and also
	\[
	\begin{split}
	(M^\top-\si^\top)&\pa_s((M^\top-\si^\top)\pa_s\vp)
	=\\
	&=  (M^\top-\si^\top)(\pa^2_s\vp) - \scal{\pa_s \vp, \tau_\si}k_\si - \sum_j \scal{\pa_s\vp,N_j}(M^\top-\si^\top)(\pa_s N_j)\\
	&=  (M^\top-\si^\top)(\pa^2_s\vp) - \scal{\pa_s \vp, \tau_\si}k_\si - \sum_j \scal{S_{N_j}(\tau_\si),\vp}\left( -S_{N_j}(\tau_\si) +\scal{N_j,\pa^2_s\si}\tau_\si \right) \\
	&= (M^\top-\si^\top)(\pa^2_s\vp) - \scal{\pa_s \vp, \tau_\si}k_\si - \sum_j \scal{B(\tau_\si,\vp),N_j}\left( \scal{B(\tau_\si,\tau_\si),N_j}\tau_\si - S_{N_j}(\tau_\si) \right),
	\end{split}
	\]
	where $B$ is the second fundamental form of $M$ in $\R^n$. Observe that if $\vp\in(T\si)^\perp$, then actually $(M^\top-\si^\top)\pa_s((M^\top-\si^\top)\pa_s\vp)=\na^2\vp$; for sake of readability, in this proof we will denote by $\na^2\vp$ the vector $(M^\top-\si^\top)\pa_s((M^\top-\si^\top)\pa_s\vp)$ for any $\vp\in TM$, that is, not only for normal fields along $\si$.
	
	We have
	\[
	\begin{split}
	M^\top\left(\frac{d}{d\ep}\bigg|_0 k_{\si_\ep} \right) 
	&= -\scal{\tau_\si,\pa_{s_\si}\vp}k_\si +M^\top(\pa^2_{s_\si}\vp) -\pa_{s_\si}\left(\scal{\pa_{s_\si}\vp,\tau_\si}\right) \tau_\si
	- \scal{\pa_{s_\si}\vp,\tau_\si} k_\si+\\
	&\indent 
	-\sum_j \scal{\pa_{s_\si}\tau_\si,N_j\circ\si} (-S_{N_j}(\vp)) \\
	&\,\,= -2\scal{\tau_\si,\pa_{s_\si}\vp}k_\si +(M^\top-\si^\top)(\pa^2_{s_\si}\vp)
	-\scal{\pa_{s_\si}\vp,\pa_{s_\si}\tau_\si}\tau_\si + \sum_j \scal{B(\tau_\si,\tau_\si),N_j} S_{N_j}(\vp) \\
	&\,\,= \na^2\vp - \scal{\tau_\si,\pa_{s_\si}\vp}k_\si - \scal{\pa_{s_\si}\vp,\pa^2_s\si}\tau_\si 
	+  \sum_j\scal{B(\tau_\si,\vp),N_j}\left( \scal{B(\tau_\si,\tau_\si),N_j}\tau_\si - S_{N_j}(\tau_\si) \right) +\\
	&\indent + \scal{B(\tau_\si,\tau_\si),N_j} S_{N_j}(\vp),
	\end{split}
	\]
	and
	\[
	\begin{split}
	\si^\top\left(\frac{d}{d\ep}\bigg|_0 k_{\si_\ep} \right) 
	&=\scal{\pa^2_{s_\si}\vp,\tau_\si}\tau_\si
	- \pa_{s_\si}(\scal{\pa_{s_\si}\vp,\tau_\si})\tau_\si
	-\sum_j  \scal{\pa_{s_\si}\tau_\si,N_j\circ\si}\,\si^\top(\pa_\ep|_0(N_j\circ\si_\ep)) =\\
	&=-\scal{\pa_{s_\si}\vp,\pa_{s_\si}\tau_\si}\tau_\si
	+ \sum_j \scal{B(\tau_\si,\tau_\si),N_j}\scal{B(\vp,\tau_\si),N_j}\tau_\si.
	\end{split}
	\]
	Hence
	\begin{equation}\label{eq:1}
	\begin{split}
	(M^\top-\si^\top)\left(\frac{d}{d\ep}\bigg|_0 k_{\si_\ep} \right)
	&= \na^2 \vp -\scal{\tau_\si,\pa_{s_\si}\vp}k_\si
	+\sum_j \scal{B(\tau_\si,\tau_\si),N_j} S_{N_j}(\vp) - \scal{B(\tau_\si,\vp),N_j}S_{N_j}(\tau_\si).
	\end{split}
	\end{equation}
	Understanding summation over repeated indices and letting $\{e_1,e_2,...,e_m\}$ be a local orthonormal frame of $TM$ along $\si$ with $e_1=\tau$ we have that
	\[
	\begin{split}
	\scal{&B(\tau_\si,\tau_\si),N_j} S_{N_j}(\vp) - \scal{B(\tau_\si,\vp),N_j}  S_{N_j}(\tau_\si) =\\
	&= \scal{B(\tau_\si,\tau_\si),N_j} 	\scal{S_{N_j}(\vp),e_i} e_i - \scal{B(\tau_\si,\vp),N_j}\scal{S_{N_j}(\tau_\si),e_i}e_i =\\
	&= \scal{B(\tau_\si,\tau_\si),N_j}\scal{B(\vp,e_i),N_j}e_i - \scal{B(\tau_\si,\vp),N_j}\scal{B(\tau_\si,e_i),N_j}e_i =\\
	&= \left(\scal{B(\tau_\si,\tau_\si),B(\vp,e_i)} - \scal{B(\tau_\si,\vp),B(\tau_\si,e_i)}\right)e_i =\\
	&= R(\tau_\si, e_i, \tau_\si, \vp)e_i=\\
	&= R(e_i,\tau_\si,\vp,\tau_\si)e_i=\\
	&=\scal{R(\vp,\tau_\si)\tau_\si,e_i}e_i,
	\end{split}
	\]
	where we used Gauss equation (\cite{DoCarmo}) and the symmetries of the Riemann tensor. Since $R(e_i,\tau_\si, \tau_\si, \tau_\si)=0$, the above calculation and \eqref{eq:1} imply \eqref{eq:VarK} and \eqref{eq:VarKPerp}.
\end{proof}

\begin{remark}\label{rem:Derivata|k|^p-2k}
	We remark that if $k\in[2,+\infty)$ and $\ga\in W^{4,p}_{imm}(\S^1;M)$, then
	\[
	\pa_s(|k|^{p-2}k)= (p-2)|k|^{p-4}\scal{k,\pa_s k} k + |k|^{p-2}\pa_s k \quad\in L^\infty,
	\]
	in the sense of weak derivatives. To see this fact one can take a sequence of smooth curves $\ga_n$ such that $\ga_n\to \ga$ strongly in $W^{4,p}$ and hence in $C^{3,\alpha}$ for an $\alpha\in(0,1)$ and argue by approximation.
\end{remark}

\begin{prop}\label{prop:FirstVariation0}
	Let $p\in[2,+\infty)$. Let $\ga\in W^{4,p}(\S^1,M)$ be a regular curve. For any $\psi \in B_\ro(0)\con T(\ga)^{4,p,\perp}$ it holds that
	\begin{equation*}
	\de E(0)[\psi]= \int -\scal{\na (|k|^{p-2}k),\na \psi}
	+ \left\lgl \frac{1}{p'}|k|^pk - k + R(|k|^{p-2}k,\tau)\tau , \psi \right\rgl\,ds.
	\end{equation*}
	For any $\psi \in B_\ro(0)\con T(\ga)^{4,p}$ it holds that
	\begin{equation*}
	\de \bE(0)[\psi]= \de\bE(0)[\psi^\perp]= \de E(0)[\psi^\perp],
	\end{equation*}
	where $\psi^\perp\coloneqq (\id-\ga^\top)\psi$.
\end{prop}

\begin{proof}
	Let us consider $\psi \in B_\ro(0)\con T(\ga)^{4,p}$. Using \Cref{lem:Variations} with $\ga,\psi$ in place of $\si,\vp$, if $\ga_\ep(\cdot)=\Phi(\ep,\cdot)$ is the variation of $\ga$, computations show that\\
	\[
	\begin{split}
	\frac{d}{d\ep}&\bigg|_0 \sE(\Phi(\ep,\cdot)) =\\
	&= \int -\scal{k,\psi} +\frac1p |k|^p \scal{\tau,\pa_s\psi}\,ds + \int |k|^{p-2}\scal{k,\pa_\ep|_0 k_{\ga_\ep}}\,ds  \\
	&= \int -\scal{k,\psi}  +\frac1p |k|^p \scal{\tau,\pa_s\psi}
	-|k|^p\scal{\tau,\pa_s\psi} +|k|^{p-2}R(k,\tau,\psi^\perp,\tau) + \scal{|k|^{p-2}k,\pa_s(M^\top-\ga^\top)\pa_s\psi}\,ds \\
	&= \int -\scal{k,\psi} -\frac{1}{p'}|k|^p\scal{\tau,\pa_s\psi} +|k|^{p-2} R(\psi^\perp,\tau,k,\tau) - \scal{\na (|k|^{p-2}k),\pa_s\psi}\,ds.
	\end{split}
	\]
	Moreover $\pa_s(\psi^\perp)=\pa_s(\psi-\scal{\psi,\tau}\tau)= \pa_s\psi -\pa_s(\scal{\psi,\tau})\tau -\scal{\psi,\tau}\pa_s\tau$ and then
	\[
	\begin{split}
	\int - \scal{&\na (|k|^{p-2}k),\pa_s\psi}\,ds =\\
	&= \int - \scal{\na (|k|^{p-2}k),\pa_s(\psi^\perp)} - \scal{\na (|k|^{p-2}k),\scal{\psi,\tau}k} \,ds \\
	&= \int -\scal{\na (|k|^{p-2}k),\na(\psi^\perp)} \,ds
	+\int |k|^p\scal{\pa_s\psi,\tau} + |k|^p\scal{\psi,k} +\scal{\psi,\tau} |k|^{p-2}\scal{k,\na k}\,ds.
	\end{split}
	\]
	Using that $ -\frac{1}{p'}|k|^p\scal{\tau,\pa_s\psi}+ |k|^p\scal{\pa_s\psi,\tau} +\scal{\psi,\tau} |k|^{p-2}\scal{k,\na k}=\pa_s\left( \frac1p|k|^p\scal{\tau,\psi} \right)-\frac1p|k|^p\scal{k,\psi}$, we conclude that
	\begin{equation}\label{eq:GeneralVariation}
	\frac{d}{d\ep}\bigg|_0 \sE(\Phi(\ep,\cdot))
	=	\int -\scal{\na (|k|^{p-2}k),\na(\psi^\perp)} + \frac{1}{p'} |k|^p\scal{k, \psi^\perp} -\scal{k,\psi^\perp} +|k|^{p-2}R(\psi^\perp,\tau,k,\tau)\,ds.
	\end{equation}
\end{proof}

\begin{cor}\label{prop:VarBoldE}
	Let $p\in[2,+\infty)$. Let $\ga:\S^1\to M$ be a fixed smooth regular curve.
	For any $\vp,\psi \in B_\ro(0)\con T(\ga)^{4,p}$ it holds that
	\begin{equation*}
	\begin{split}
	\de \bE(\vp)[\psi] &=
	-\left\lgl
	\na_{{\ga_\vp}} |k_{\ga_\vp}|^{p-2}k_{\ga_\vp} , \na_{{\ga_\vp}} (\ga_\vp^\perp) \cT(\psi)
	\right\rgl_{L^{p'}(ds_{\ga_\vp}),L^p(ds_{\ga_\vp})}
	+\\ &\indent + 
	\left\lgl 
	\cT^{\star}\left(
	\frac{1}{p'} |k_{\ga_\vp}|^pk_{\ga_\vp} - k_{\ga_\vp} 
	+ R (|k_{\ga_\vp}|^{p-2}k_{\ga_\vp},\tau_{\ga_\vp})\tau_{\ga_\vp}
	\right), \psi
	\right\rgl_{L^{p'}(ds_{\ga_\vp}),L^p(ds_{\ga_\vp})}
	\end{split}
	\end{equation*}
	where $\ga_\vp(\cdot)=\Phi(1,\cdot)$, $\Phi$ is the variation of $\ga$ given by $\vp$, and $\cT:T_{\ga(x)}M\to T_{\Phi(1,x)}M$ is the function $\cT(\psi)=d[\exp_{\ga(x)}]_{\vp}(\psi)$ and $\cT^\star$ is its adjoint.
\end{cor}

\begin{proof}
	Let us denote by $\Phi_V^\alpha(t,x)=\si_V(t)$ the variation of a curve $\alpha$ with respect to a field $V(x)$ along $\alpha$. We need to consider the curve $\Phi_{\vp+\ep\psi}^\ga(1,\cdot)=\exp_{\ga(x)}(\vp+\ep\psi)$. We have that
	\[
	\frac{d}{d\ep}\bigg|_0 \Phi_{\vp+\ep\psi}^\ga(1,\cdot)= d[\exp_{\ga(x)}]_{\vp}(\psi).
	\]
	For any $x$ denote by $\cT:T_{\ga(x)}M\to T_{\Phi^\ga_\vp(1,x)}M$ the function
	\[
	\cT(\psi)=d[\exp_{\ga(x)}]_{\vp}(\psi).
	\]
	By chain rule we have that
	\[
	\frac{d}{d\ep}\bigg|_0 \sE(\Phi_{\vp+\ep\psi}^\ga(1,\cdot))= \frac{d}{d\ep}\bigg|_0 \sE\left( \Phi^{\Phi^\ga_\vp(1,\cdot)}_{\cT(\psi)}(\ep,\cdot) \right).
	\]
	Roughly speaking, differentiation of the variation of $\ga$ with respect to the field $\vp+\ep\psi$ is equivalent to differentiation at the varied curve $\ga_\vp\coloneqq\Phi_\vp^\ga(1,\cdot)$ with respect to the field $\cT(\psi)$.
	
	Therefore if we consider $\vp,\psi \in B_\ro(0)\con T(\ga)^{4,p}$, \Cref{eq:GeneralVariation} implies that
	\[
	\begin{split}
	\frac{d}{d\ep}\bigg|_0 \sE(\Phi_{\vp+\ep\psi}^\ga(1,\cdot))
	&= \int -\scal{\na_{{\ga_\vp}} |k_{\ga_\vp}|^{p-2}k_{\ga_\vp},\na_{{\ga_\vp}} (\ga_\vp^\perp\cT(\psi)) } +\\
	&\indent + \frac{1}{p'} |k_{\ga_\vp}|^p\scal{k_{\ga_\vp},\ga_\vp^\perp\cT(\psi)} +\\
	&\indent 
	-\scal{k_{\ga_\vp}, \ga_\vp^\perp\cT(\psi)}
	+ |k_{\ga_\vp}|^{p-2}R(\ga_\vp^\perp\cT(\psi), \tau_{\ga_\vp}, k_{\ga_\vp}, \tau_{\ga_\vp} )\,ds_{\ga_\vp}.
	\end{split}
	\]
\end{proof}

\medskip

Now we want to calculate the second variation of $\sE$. More precisely, as in the case of \Cref{sec:Rn}, for a smooth immersion $\ga:\S^1\to M$ we consider normal fields $\vp,\psi \in T(\ga)^{4,p,\perp}$ along $\ga$ and we compute
\[
\cL\coloneqq\de^2E(0): T(\ga)^{4,p,\perp} \to (T(\ga)^{4,p,\perp})^\star,
\]
that is
\[
\cL(\vp)[\psi] = \frac{d}{d\ep}\bigg|_0 \frac{d}{d\eta}\bigg|_0 E(\ep\vp+\eta\psi) 
= \frac{d}{d\ep}\bigg|_0 \frac{d}{d\eta}\bigg|_0 \sE(\Phi^\ga_{\ep\vp+\eta\psi}(1,\cdot)) ,
\]
where $\Phi^\ga_{\ep\vp+\eta\psi}(1,\cdot)$ is the variation of $\ga$ via the field $\ep\vp+\eta\psi$.
Observe that $\cL(\vp)[\psi]=\cL(\psi)[\vp]$.

We need a technical tool first.

\begin{lemma}
	Let $\ga:\S^1\to M$ be a fixed smooth immersion and let $\ro$ be given by \Cref{lem:Variation}. Let $\ga_\ep(\cdot)=\Phi(\ep,\cdot)$ with $\Phi$ the variation of $\ga$ with variation field $\vp\in T(\ga)^{4,p,\perp}$.
	
	If $V(\ep) \in T(\ga_\ep)^{1,p}$ is a field along $\ga_\ep$ differentiable with respect to $\ep$ with $\pa_\ep V(\ep)\in T(\ga_\ep)^{1,p}$, we have that
	\begin{equation}\label{eq:V1}
	\frac{d}{d\ep}\bigg|_0 \pa_{s_{\ga_\ep}} V(\ep) = \scal{k_\ga,\vp}\pa_{s_\ga}V(0) + \pa_{s_\ga}\frac{d}{d\ep}\bigg|_0 V(\ep).
	\end{equation}
	
	If also $V(\ep)\in TM\cap(T\ga_\ep)^\perp$ for any $\ep$, then
	\begin{equation}\label{eq:V2}
	\begin{split}
	(M^\top-\ga^\top)\frac{d}{d\ep}\bigg|_0  (\na_{{\ga_\ep}}V(\ep)) 
	&= (M^\top-\ga^\top)\pa_{s_\ga}((M^\top-\ga^\top)\frac{d}{d\ep}\bigg|_0 V(\ep)) +\\
	&\indent	- \scal{V(0),\na_{{\ga}}\vp}k_\ga 
	+ \scal{k_\ga,\vp}\na_{{\ga}}V(0) 
	+ \scal{V(0),k_\ga}\na_{{\ga}}\vp +\\
	&\indent+ (\id-\ga^\top)R(\vp,\tau_\ga)V(0).
	\end{split}
	\end{equation}
\end{lemma}

\begin{proof}
	\Cref{eq:V1} follows by a direct calculation. In order to derive \eqref{eq:V2} let $\{N_j\}$ be a local orthonormal frame of $(TM)^\perp$. Understanding summation over $j$ it holds that
	\begin{equation}\label{eq:2}
	\begin{split}
	(M^\top-\ga^\top)\frac{d}{d\ep}\bigg|_0  (\na_{{\ga_\ep}}V(\ep)) 
	&= \scal{k_\ga,\vp}\na_{{\ga}}V(0)
	+ (M^\top-\ga^\top)\pa_{s_\ga} \frac{d}{d\ep}\bigg|_0 V(\ep) 
	-\scal{\pa_{s_\ga}V(0),\tau_\ga}\na_{{\ga}}\vp +\\
	&\qquad- \scal{\pa_{s_\ga}V(0),N_j} (M^\top-\ga^\top)\frac{d}{d\ep}\bigg|_0 N_j\circ \ga_\ep = \\
	&= (M^\top-\ga^\top)\pa_{s_\ga} \frac{d}{d\ep}\bigg|_0 V(\ep)
	+ \scal{k_\ga,\vp}\na_{{\ga}}V(0) 
	+ \scal{V(0),k_\ga}\na_{{\ga}}\vp +\\
	&\qquad+ \scal{V(0),S_{N_j}(\tau_\ga)}(S_{N_j}(\vp)-\scal{S_{N_j}(\vp),\tau_\ga}\tau_\ga).
	\end{split}
	\end{equation}
	Moreover for any field $W\in T(\ga)^{1,p}$ we have that
	\[
	(M^\top-\ga^\top)\pa_{s_\ga}((M^\top-\ga^\top)W) = (M^\top-\ga^\top)\pa_{s_\ga}W -\scal{W,\tau_\ga}k_\ga +\scal{W,N_j}(S_{N_j}(\tau_\ga) -\scal{S_{N_j}(\tau_\ga),\tau_\ga}\tau_\ga).
	\]
	Using $W= \frac{d}{d\ep}\big|_0 V(\ep)$ we deduce
	\[
	\begin{split}
	(M^\top-\ga^\top)\pa_{s_\ga} \frac{d}{d\ep}\bigg|_0 V(\ep) & = (M^\top-\ga^\top)\pa_{s_\ga}((M^\top-\ga^\top)\frac{d}{d\ep}\bigg|_0 V(\ep)) 
	+ \left\lgl \frac{d}{d\ep}\bigg|_0 V(\ep),\tau_\ga\right\rgl k_\ga +\\
	&\qquad-\left\lgl \frac{d}{d\ep}\bigg|_0 V(\ep), N_j \right\rgl (S_{N_j}(\tau_\ga) -\scal{S_{N_j}(\tau_\ga),\tau_\ga}\tau_\ga) =\\
	&= (M^\top-\ga^\top)\pa_{s_\ga}((M^\top-\ga^\top)\frac{d}{d\ep}\bigg|_0 V(\ep)) 
	- \scal{V(0),\na_{{\ga}}\vp}k_\ga +\\
	&\qquad- \scal{V(0),S_{N_j}(\vp)}(S_{N_j}(\tau_\ga) -\scal{S_{N_j}(\tau_\ga),\tau_\ga}\tau_\ga).
	\end{split}
	\]
	Moreover
	\[
	\begin{split}
	\scal{V(0),S_{N_j}(\tau_\ga)}&(S_{N_j}(\vp)-\scal{S_{N_j}(\vp),\tau_\ga}\tau_\ga) - \scal{V(0),S_{N_j}(\vp)}(S_{N_j}(\tau_\ga) -\scal{S_{N_j}(\tau_\ga),\tau_\ga}\tau_\ga) \\
	&= \scal{B(V(0),\tau_\ga),N_j}S_{N_j}(\vp)
	- \scal{B(V(0),\vp),N_j} S_{N_j}(\tau_\ga) 
	+\\ &\indent + \left(
	\scal{B(V(0),\vp),N_j}\scal{B(\tau_\ga,\tau_\ga),N_j} 
	- \scal{B(V(0),\tau_\ga),N_j}\scal{B(\vp,\tau_\ga),N_j}
	\right)\tau_\ga \\
	&= \left(
	\scal{B(V(0),\tau_\ga),B(\vp,e_i)} - \scal{B(V(0),\vp),B(\tau_\ga,e_i)}
	\right) e_i +\\
	&\indent + \left(
	\scal{B(V(0),\vp),B(\tau_\ga,\tau_\ga)} - \scal{B(V(0),\tau_\ga),B(\vp,\tau_\ga)}
	\right) \tau_\ga \\
	&= R(V(0),e_i,\tau_\ga,\vp)e_i + R(V(0),\tau_\ga,\vp,\tau_\ga)\tau_\ga \\
	&= -R(\tau_\ga,\vp)V(0) + \scal{R(\tau_\ga,\vp)V(0),\tau_\ga}\tau_\ga\\
	&= (\id-\ga^\top)R(\vp,\tau_\ga)V(0),
	\end{split}
	\]
	where $\{e_1,...,e_m\}=\{\tau_\ga,e_2,...,e_m\}$ is a local orthonormal frame of $TM$ and we used Gauss equation, and summation over repeated indices was understood. Inserting the previous identities in \eqref{eq:2} yields \eqref{eq:V2}.
\end{proof}

In the following proposition we calculate the second variation $\cL(\vp)[\psi]$ of $\sE$ with respect to normal variation fields $\vp,\psi$ along the given curve $\ga$. In the statement we isolate an integral depending on second order derivatives is $\psi$, a second integral depending at most on first order derivatives in $\psi$, and a third integral in which the first variation of the energy appears. The complete calculation is explicit in \Cref{eq:SecVarTotale}, and we shall also use such complete expression \eqref{eq:SecVarTotale} later on.

\begin{prop}\label{prop:SecondVariation}
	Let $p\in[2,+\infty)$. Let $\ga:\S^1\to M$ be a fixed smooth immersion and let $\ro$ be given by \Cref{lem:Variation}.
	For any $\vp,\psi \in B_\ro(0)\con T(\ga)^{4,p,\perp}$ it holds that
	\[
	\begin{split}
	\cL(\vp)[\psi]	
	&=
	\int
	\left\lgl 
	|k|^{p-2}\na^2\vp , \na^2\psi \right\rgl 
	+(p-2) |k|^{p-4}  \scal{k,\na^2\vp}\scal{k,\na^2\psi} +\\
	&\indent + \bigg\lgl |k|^{p-2} R(\vp,\tau)\tau 
	+ (p-2)|k|^{p-4}\scal{k, R(\vp,\tau)\tau}k
	+ |k|^{p-2}	\scal{k,\vp}k,\na^2\psi
	\bigg\rgl\,ds +\\
	&+\int A(\vp,\psi)\,ds
	-  \int \left(\left\lgl |k|^{p-2}k, \na^2 \psi \right\rgl 
	+ \left\lgl \frac{1}{p'}|k|^p k -k + R(|k|^{p-2}k,\tau)\tau , \psi   \right\rgl\right) \scal{k,\vp}\,ds,
	\end{split}
	\]
	where $A(\cdot,\cdot)$ is bilinear with $A(\vp,\psi)$ depending at most on first order derivatives in $\psi$, and, more precisely, the precise expression for $\cL(\vp)[\psi]$ is given by \eqref{eq:SecVarTotale}.
\end{prop}

The proof of \Cref{prop:SecondVariation} is contained in Appendix A.

\bigskip

\subsection{Critical points}

As we are interested in the properties of the variations evaluated at a critical point of $\sE$, we now consider the variations at such a curve. Let $p\in[2,+\infty)$, $\ga\in W^{4,p}_{imm}(\S^1,M)$ be fixed and let $\ro$ be given by \Cref{lem:Variation}. Recall that by \Cref{prop:FirstVariation0} the curve $\ga$ is a critical point if and only if for any $\psi \in B_\ro(0)\con T(\ga)^{4,p,\perp}$ it holds that
\begin{equation*}
\int -\scal{\na (|k|^{p-2}k),\na\psi}
+ \left\lgl \frac{1}{p'}|k|^pk - k + R(|k|^{p-2}k,\tau)\tau , \psi \right\rgl\,ds
=0.
\end{equation*}

\begin{lemma}\label{lem:RegularityCritPoint}
	Let $p\in[2,+\infty)$. Let $\ga\in W^{4,p}_{imm}(\S^1,M)$ be a critical point. Then $|k|^{p-2}k \in C^{3,\alpha}(\S^1,M)$ for some $\alpha\in(0,1)$, and either $\ga$ is a smooth geodesic and $k\equiv0$ or the set $\{x\,\,:\,\,k(x)=0\}$ is finite.
\end{lemma}

\begin{proof}
	Let $V=|k|^{p-2}k$. By \Cref{rem:Derivata|k|^p-2k} we know that $V\in T(\ga)^{1,\infty,\perp}$. Moreover $\scal{D_s V,\tau}=-\scal{V,k}$ and $\na V= D_s V + \scal{V,k}\tau$, and then $V$ solves
	\[
	\int \scal{D_s V,D_s\psi}\,ds = \int \left\lgl \frac{1}{p'} |k|^pk - k + |k|^pk +R(V,\tau)\tau, \psi  \right\rgl \,ds,
	\]
	for any $\psi \in T(\ga)^{1,p,\perp}$. In particular the weak derivative $D_s(D_sV)$ exists in $L^\infty$, and $V\in T(\ga)^{2,\infty,\perp}\con C^{1,\alpha}$ for $\alpha\in(0,1)$ that may change from line to line. By assumption also $k\in C^{1,\alpha}$ and thus, from the previous equation, $D^2_s V \in C^{1,\alpha}$, that implies $V\in C^{3,\alpha}$.
	
	Now if at some $x_0$ it holds that $V(x_0)=0$ and $D_s(V)(x_0)=0$, since $V$ now solves
	\[
	D^2_sV = - \frac{1}{p'} |k|^pk + k - |k|^pk -R(V,\tau)\tau
	\]
	pointwise in the classical sense, by existence and uniqueness we would get that $V(x)=0$ and thus $k(x)=0$ in a neighborhood of $x_0$. Iterating the argument this would imply that $k\equiv0$. It follows that if $k\not\equiv0$ then the set $\{x\,\,:\,\,k(x)=0\}$ has to be finite, for otherwise by compactness and since $V\in C^{3,\alpha}$ this would imply the existence of a point $x_0$ with $V(x_0)=0$ and $D_s(V)(x_0)=0$.
\end{proof}

In the following the symbol $\Omega:X\to Y$ will usually denote a compact linear operator between Banach spaces $X$ and $Y$. As compact perturbations do not affect the Fredholmness properties of an operator, from which the \L ojasiewicz-Simon inequality eventually follows, the symbol $\Omega$ will denote operators which may change from line to line, but we will always specify the spaces between which $\Omega$ acts compactly.

\begin{prop}\label{prop:VariationsCritPoints}
	Let $p\in[2,+\infty)$. Let $\ga\in W^{4,p}_{imm}(\S^1,M)$ be a critical point and let $\ro$ be given by \Cref{lem:Variation}. Then
	\begin{equation*}
	\de E(0)[\psi] = \left\lgl \na^2(|k|^{p-2}k) + \frac{1}{p'}|k|^pk - k + R(|k|^{p-2}k,\tau)\tau
	, \psi \right\rgl_{L^{p'}(ds),L^p(ds)},
	\end{equation*}
	for any $\psi\in B_\ro(0)\con T(\ga)^{4,p}$, and
	\begin{equation}\label{eq:SecondVariationCritPoint}
	\begin{split}
	\cL(\vp)[\psi] &= 
	\int 
	|k|^{p-2}\scal{\na^2\vp,\na^2\psi}  + (p-2)|k|^{p-4}\scal{k,\na^2\vp}\scal{k,\na^2\psi}
	+\\
	&\indent 	+ \left\lgl
	(p-2)|k|^{p-4}\scal{k,R(\vp,\tau)\tau}k
	+|k|^{p-2}R(\vp,\tau)\tau 
	, \na^2\psi
	\right\rgl 
	\,ds
	+\int
	\scal{\Om(\vp),\psi}
	\,ds ,
	\end{split} 
	\end{equation}
	for any $\vp,\psi\in B_\ro(0)\con T(\ga)^{4,p,\perp}$, where $\Om: T(\ga)^{4,p,\perp}\to T(\ga)^{p',\perp}$
	is compact.
\end{prop}

\begin{proof}
	We need to prove \eqref{eq:SecondVariationCritPoint}. By \Cref{lem:RegularityCritPoint} and by \eqref{eq:SecondVarTotale}, integration by parts yields
	\begin{equation*}
	\begin{split}
	\cL(\vp)[\psi] 
	&= 
	\int 
	|k|^{p-2}\scal{\na^2\vp,\na^2\psi}  + (p-2)|k|^{p-4}\scal{k,\na^2\vp}\scal{k,\na^2\psi} +\\
	&\indent 	+ \left\lgl
	(p-2)|k|^{p-4}\scal{k,R(\vp,\tau)\tau}k
	+|k|^{p-2}R(\vp,\tau)\tau 
	, \na^2\psi
	\right\rgl 
	\,ds +\\
	&+\int 
	\scal{(D_\vp\cR)(\psi,\tau,\tau),|k|^{p-2}k} 
	+ \bigg\lgl 
	R(\psi,\tau)\tau,
	|k|^{p-2}\na^2\vp +\\
	&\indent  +(p-2) |k_\ga|^{p-4}\scal{k_\ga,\na^2\vp}k_\ga
	+ |k_\ga|^{p-2} R(\vp,\tau_\ga)\tau_\ga 
	+ (p-2)|k_\ga|^{p-4}\scal{k_\ga, R(\vp,\tau_\ga)\tau_\ga}k_\ga
	\bigg\rgl
	\,ds +\\
	&+\int
	\scal{\Om(\vp),\psi}
	\,ds
	\end{split}
	\end{equation*}
	where by \Cref{lem:RegularityCritPoint} we could use that
	\[R(\na|k|^{p-2}k,\psi,\vp,\tau) = - R(\psi,\na|k|^{p-2}k,\vp,\tau), \]
	\[\int - R (\na\psi, |k|^{p-2}k, \vp,\tau ) = \int \scal{\psi, \na(R(\vp,\tau)|k|^{p-2}k)},\]
	\[\scal{R(\psi,\tau)\tau, \scal{\vp,k}|k|^{p-2}k} = \scal{\psi, R(\scal{\vp,k}|k|^{p-2}k,\tau)\tau}.\]
	Moreover $\Om: T(\ga)^{4,p,\perp}\to T(\ga)^{p',\perp}$
	is compact. For a given local reference frame $\{\pa_j\}$ in $M$ we can also write
	\[
	(D_\vp\cR)(\psi,\tau,\tau) = \vp^m\psi^i\tau^j\tau^k \left( \pa_m R^l_{ijk} + \Ga^\alpha_{mi}R^l_{\alpha jk}
	+ \Ga^\be_{mi}R^l_{i \be k}
	+ \Ga^\ga_{mi}R^l_{i j \ga}
	\right) \pa_l ,
	\]
	\[
	R(\psi,\tau)\tau = \psi^i\tau^j\tau^k R^l_{ijk},
	\]
	where $\vp=\vp^m\pa_m$, $\psi=\psi^i\pa_i$, $\tau=\tau^a\pa_a$, $R(\pa_i,\pa_j)\pa_k = R^l_{ijk}\pa_l$, and $\{\Ga^k_{ij}\}$ are the Christoffel symbols of $M$. This means that
	\[
	\begin{split}
	\cL(\vp)[\psi] &= 
	\int 
	|k|^{p-2}\scal{\na^2\vp,\na^2\psi}  + (p-2)|k|^{p-4}\scal{k,\na^2\vp}\scal{k,\na^2\psi}
	+\\
	&\indent 	+ \left\lgl
	(p-2)|k|^{p-4}\scal{k,R(\vp,\tau)\tau}k
	+|k|^{p-2}R(\vp,\tau)\tau 
	, \na^2\psi
	\right\rgl 
	\,ds
	+\int
	\scal{\Om(\vp),\psi}
	\,ds ,
	\end{split} 
	\]
	where $\Om: T(\ga)^{4,p,\perp}\to T(\ga)^{p',\perp}$
	is compact.
\end{proof}

We conclude with the following two observations about the regularity of critical points.

\begin{remark}\label{rem:RegularityP2}
	Let $p=2$. Then critical points are smooth up to reparametrization with respect to constant speed. Indeed \Cref{lem:RegularityCritPoint} implies that a constant speed critical point $\ga$ verifies that $k\in C^{3,\alpha}$, but then a bootstrap argument on the equation
	\[
	D_s^2 k = -\frac12 |k|^2k +k - |k|^2 - R(k,\tau)\tau,
	\]
	gives that $k\in C^\infty$ and thus $\ga\in C^\infty$.
\end{remark}

\begin{remark}\label{rem:RegularityP>2}
	Let $p>2$. If $\ga$ is a constant speed critical point of $\sE$ for some $p>2$ and if $k$ never vanishes, then $\ga$ is smooth. Indeed \Cref{lem:RegularityCritPoint} implies that a constant speed critical point $\ga$ verifies that $|k|^{p-2}k\in C^{3,\alpha}$, and thus the equation
	\[
	D_s^2 (|k|^{p-2}k) = -\frac{1}{p'} |k|^pk +k - |k|^p - R(|k|^{p-2}k,\tau)\tau,
	\]
	is classically satisfied. By a bootstrap argument we get that $|k|^{p-2}k$ is smooth, and then $|k|^p\in C^1$. Hence bootstrap on the equality
	\[
	k = D_s^2 (|k|^{p-2}k) + \frac{1}{p'} |k|^pk + |k|^p  +  R(|k|^{p-2}k,\tau)\tau,
	\]
	implies that $k$ is smooth, and then so is $\ga$.
\end{remark}

\bigskip

\subsection{Analysis of the second variations and \L ojasiewicz-Simon inequality}

In the following we study the properties of the variations of $\sE$, leading to the application of \Cref{cor:LojaFunctionalAnalytic} and then to the proof of the convergence of the gradient flows. We shall distinguish between the cases $p=2$ and $p>2$, indeed, as also \Cref{prop:VariationsCritPoints} suggests, we will see that the properties of the second variation at a curve $\ga$ depend on the zeros of the curvature of $\ga$ if $p>2$, while for $p=2$ the scenario is more regular.

For the convenience of the reader, let us start by recollecting the formulas for first and second variations at critical points under the form we will use them. We will always assume without loss of generality that critical points are parametrized with constant speed.

\begin{prop}\label{prop:VariationsP2}
	Let $p=2$. Let $\ga:\S^1\to M$ be a fixed smooth immersion and let $\ro$ be given by \Cref{lem:Variation}.
	
	For any $\vp,\psi \in B_\ro(0)\con T(\ga)^{4,2,\perp}$ it holds that
	\begin{equation}\label{eq:FirstVarP2}
	\begin{split}
	\de E(\vp)[\psi] &= 
	\left\lgl
	\ga^\perp \cT^\star\left(
	\na_{\ga_\vp}^2 k_{\ga_\vp} 
	+
	\frac{1}{2} |k_{\ga_\vp}|^2k_{\ga_\vp} - k_{\ga_\vp} 
	+ R (k_{\ga_\vp},\tau_{\ga_\vp})\tau_{\ga_\vp}
	\right)
	, \psi
	\right\rgl_{L^2(ds_{\ga_\vp}),L^2(ds_{\ga_\vp})}
	\end{split}
	\end{equation}
	where $\ga_\vp(\cdot)=\Phi(1,\cdot)$, $\Phi$ is the variation of $\ga$ given by $\vp$, and $\cT:T_{\ga(x)}M\to T_{\Phi(1,x)}M$ is the function $\cT(\psi)=d[\exp_{\ga(x)}]_{\vp}(\psi)$ and $\cT^\star$ is its adjoint.
	
	If $\ga$ is any critical point, then
	\begin{equation}\label{eq:SecVarP2}
	\begin{split}
	\cL(\vp)[\psi] &= 
	\int 
	\left\lgl \na^2\vp 
	, \na^2\psi
	\right\rgl 
	\,ds
	+\int
	\scal{\Om(\vp),\psi}
	\,ds ,
	\end{split} 
	\end{equation}
	for any $\vp,\psi\in B_\ro(0)\con T(\ga)^{4,2,\perp}$, where $\Om: T(\ga)^{4,2,\perp}\to T(\ga)^{2,\perp}$
	is compact.
%
\end{prop}

\begin{proof}
	\Cref{eq:FirstVarP2} and \Cref{eq:SecVarP2} immediately follow from \Cref{prop:VarBoldE} and \eqref{eq:SecondVariationCritPoint}.
%
\end{proof}

In case $p=2$, for a given immersion $\ga$, \Cref{prop:VariationsP2} implies that the operator $\de E(\vp) \in (T(\ga)^{4,2,\perp})^\star$ is represented by the function
\[
\na_{T(\ga)^2,T(\ga)^2} E (\vp)  = |\pa_x\ga_\vp| \ga^\perp \cT^\star\left(
\na_{\ga_\vp}^2 k_{\ga_\vp} 
+
\frac{1}{2} |k_{\ga_\vp}|^2k_{\ga_\vp} - k_{\ga_\vp} 
+ R (k_{\ga_\vp},\tau_{\ga_\vp})\tau_{\ga_\vp}
\right)
\qquad\in T(\ga)^{2,\perp},
\]
in the notation of \Cref{prop:VariationsP2}. In this way we can say that $\de E : T(\ga)^{4,2,\perp}\to (T(\ga)^{2,\perp})^\star$ via the paring
\[
\de E(\vp)[\psi] = \left\lgl \na_{T(\ga)^2,T(\ga)^2} E (\vp) , \psi \right\rgl_{L^2(dx),L^2(dx)}.
\]
Similarly we have that $\cL:T(\ga)^{4,2,\perp}\to  (T(\ga)^{2,\perp})^\star$ with
\[
\cL(\vp)[\psi] = \int \scal{\na^4 \vp, \psi} + \scal{\Om(\vp),\psi}
\,ds ,
\]
in the notation of \Cref{prop:VariationsP2}.

\medskip

Now we analogously consider $p>2$.

\begin{prop}\label{prop:VariationsP>2}
	Let $p>2$. Let $\ga:\S^1\to M$ be a fixed smooth immersion and $\ro$ be given by \Cref{lem:Variation}.
	
	For any $\vp,\psi \in B_\ro(0)\con T(\ga)^{4,p,\perp}$ it holds that
	\begin{equation*}
	\begin{split}
	\de E(\vp)[\psi] &= 
	-\left\lgl \na_{\ga_\vp} |k_{\ga_\vp}|^{p-2}k_{\ga_\vp} , \na_{\ga_\vp} T\ga_\vp^\perp \cT \psi  \right\rgl_{L^{p'}(ds_{\ga_\vp}),L^p(ds_{\ga_\vp})} +\\
	&\indent +	
	\left\lgl
	\ga^\perp \cT^\star\left(
	\frac{1}{p'} |k_{\ga_\vp}|^p k_{\ga_\vp} - k_{\ga_\vp} 
	+ R (|k_{\ga_\vp}|^{p-2}k_{\ga_\vp},\tau_{\ga_\vp})\tau_{\ga_\vp}
	\right)
	, \psi
	\right\rgl_{L^{p'}(ds_{\ga_\vp}),L^p(ds_{\ga_\vp})}
	\end{split}
	\end{equation*}
	where $\ga_\vp(\cdot)=\Phi(1,\cdot)$, $\Phi$ is the variation of $\ga$ given by $\vp$, and $\cT:T_{\ga(x)}M\to T_{\Phi(1,x)}M$ is the function $\cT(\psi)=d[\exp_{\ga(x)}]_{\vp}(\psi)$ and $\cT^\star$ is its adjoint.
	
	If $\ga$ is a critical point such that $|k_\ga(x)|\neq0$ for any $x$, then
	\begin{equation}\label{eq:FirstVarP>2}
	\begin{split}
	\de E(\vp)[\psi] &= 
	\left\lgl
	\ga^\perp \cT^\star\left(
	\na_{\ga_\vp}^2 |k_{\ga_\vp}|^{p-2}k_{\ga_\vp} 
	+
	\frac{1}{p'} |k_{\ga_\vp}|^p k_{\ga_\vp} - k_{\ga_\vp} 
	+ R (|k_{\ga_\vp}|^{p-2}k_{\ga_\vp},\tau_{\ga_\vp})\tau_{\ga_\vp}
	\right)
	, \psi
	\right\rgl_{L^{p'}(ds_{\ga_\vp}),L^p(ds_{\ga_\vp})}
	\end{split}
	\end{equation}
	and
	\begin{equation}\label{eq:SecVarP>2}
	\begin{split}
	\cL(\vp)[\psi] &= 
	\int 
	|k|^{p-2}\left\lgl \na^2\vp 
	, \na^2\psi
	\right\rgl 
	+ (p-2)|k|^{p-4}\scal{k,\na^2 \vp}\scal{k,\na^2 \psi}
	\,ds
	+\int
	\scal{\Om(\vp),\psi}
	\,ds ,
	\end{split} 
	\end{equation}
	for any $\vp,\psi\in B_\ro(0)\con T(\ga)^{4,p,\perp}$, where $\Om: T(\ga)^{4,p,\perp}\to T(\ga)^{p',\perp}$
	is compact.
	
	If $\ga$ is a geodesic then
	\begin{equation}\label{eq:SecVarP>2Geod}
	\cL(\vp)[\psi] = \int \scal{\na\vp,\na\psi} - R(\psi,\tau,\vp,\tau)\,ds,
	\end{equation}
	for any $\vp,\psi\in B_\ro(0)\con T(\ga)^{4,p,\perp}$.
\end{prop}

\begin{proof}
	The statements immediately follow from \Cref{prop:VarBoldE}, \eqref{eq:SecVarTotale}, and \eqref{eq:SecondVariationCritPoint}, together with \Cref{rem:RegularityP>2}.	
\end{proof}

It is clear from \Cref{prop:VariationsP>2} that whenever $k$ vanishes, the leading terms in the bilinear form defining $\cL$ disappear, and we cannot expect strong Fredholmness properties on $\cL$.

However, if $p>2$, for a given smooth critical point $\ga$ with $|k(x)|\neq0$ for any $x$, \Cref{prop:VariationsP>2} implies that the operator $\de E(\vp) \in (T(\ga)^{4,p,\perp})^\star$ is represented by the function
\[
\na_{T(\ga)^{p'},T(\ga)^p} E (\vp)\,\,\in T(\ga)^{p',\perp},
\]
\[
\na_{T(\ga)^{p'},T(\ga)^p} E (\vp)  = |\pa_x\ga_\vp| \ga^\perp \cT^\star\left(
\na_{\ga_\vp}^2 |k_{\ga_\vp}|^{p-2} k_{\ga_\vp} 
+
\frac{1}{p'} |k_{\ga_\vp}|^p k_{\ga_\vp} - k_{\ga_\vp} 
+ R (|k_{\ga_\vp}|^{p-2}k_{\ga_\vp},\tau_{\ga_\vp})\tau_{\ga_\vp}
\right)
,
\]
in the notation of \Cref{prop:VariationsP>2}. In this way we can say that $\de E : T(\ga)^{4,p,\perp}\to (T(\ga)^{p,\perp})^\star$ via the paring
\[
\de E(\vp)[\psi] = \left\lgl \na_{T(\ga)^{p'},T(\ga)^p} E (\vp) , \psi \right\rgl_{L^{p'}(dx),L^p(dx)}.
\]
Similarly we have that $\cL:T(\ga)^{4,p,\perp}\to  (T(\ga)^{p,\perp})^\star$ with
\[
\cL(\vp)[\psi] = \int \scal{\na^2 \left(|k|^{p-2}\na^2 \vp \right), \psi} 
+ (p-2)\left\lgl \na^2 \left(|k|^{p-4}\scal{k,\na^2\vp}k\right) , \psi   \right\rgl \,ds 
+ \int 
\scal{\Om(\vp),\psi}
\,ds ,
\]
in the notation of \Cref{prop:VariationsP>2}.

\medskip

With the above results we can now derive the desired Fredholmenss properties on the second variation functionals. Once again, we shall divide the cases $p=2$ and $p>2$, as also the technical part of the two proofs is different.

\begin{lemma}\label{lem:FredholmnessP2}
	Let $p=2$. Let $\ga:\S^1\to M$ be a smooth critical point and let $\ro>0$ be given by \Cref{lem:Variation}. Then the operator $\cL: T(\ga)^{4,2,\perp}\to  (T(\ga)^{2,\perp})^\star$ represented by the function
	\[
	\cL(\vp) = \na^4 \vp + \Omega(\vp) \qquad\in T(\ga)^{2,\perp},
	\]
	where $\Om:T(\ga)^{4,2,\perp}\to T(\ga)^{2,\perp}$ is compact, is Fredholm of index zero.	
\end{lemma}

\begin{proof}
	Since $\Om:T(\ga)^{4,2,\perp}\to T(\ga)^{2,\perp}$ is compact, it is equivalent to prove that
	\[
	\id + \na^4: T(\ga)^{4,2,\perp}(ds) \to T(\ga)^{2,\perp}(ds),
	\]
	is Fredholm of index zero. Indeed we claim that it is actually invertible. It is clearly injective, indeed if $\vp + \na^4\vp=0$, then multiplying by $\vp$ and integrating one has
	\[
	\int |\na^2\vp|^2 + |\vp|^2 \, ds =0,
	\]
	and then $\vp=0$. So we need to prove the surjectivity.
	
	Let $a:T(\ga)^{2,2,\perp}\times T(\ga)^{2,2,\perp}\to\R$ the continuous bilinear form defined by
	\[
	a(\vp,\psi)=\int_{\S^1} \lgl \na^2 \vp, \na^2 \psi\rgl + \lgl \vp,\psi\rgl\,ds,
	\]
	For $\vp\in T(\ga)^{2,2,\perp}$ it holds that
	\begin{equation}\label{eq:Nabla2D2}
	\na^2 \vp =  D_s^2 \vp - \lgl D_s^2\vp,\tau\rgl\tau  - \lgl D_s\vp,\tau\rgl k = D^2_s\vp + \left( 2\lgl D_s\vp,k\rgl + \lgl \vp, D_s k\rgl \right) \tau  - \lgl D_s\vp,\tau\rgl k,
	\end{equation}
	and
	\begin{equation}\label{eq:D2Partial2}
	D_s^2\vp = \pa_s^2\vp +\left( 2\scal{\pa_s \vp, \pa_s N_j} + \scal{\vp,\pa_s^2N_j} \right)N_j + \scal{\vp, S_{N_j}(\tau)}S_{N_j}(\tau) ,
	\end{equation}
	where $\{N_j\}$ is a local orthonormal frame of $TM^\perp$, and we understood sum over $j$.
	Therefore for $\vp\in T(\ga)^{2,2,\perp}$ we have that
	\[
	\begin{split}
	\int  \lgl D^2_s\vp,D^2_s\vp\rgl\,ds &= \int |\na^2\vp|^2 + |\lgl D_s\vp,\tau\rgl k|^2 + 2 \lgl D_s\vp,\tau\rgl \lgl \na^2\vp, k\rgl + |2\lgl D_s\vp,k\rgl + \lgl \vp, D_s k\rgl|^2 \,ds \\
	&\le \int \frac32 |\na^2\vp|^2 + C(\ga)\left(|\vp|^2 + |D_s\vp|^2\right)\,ds \\
	&\le \int \frac32 |\na^2\vp|^2 + C(\ga)\left(|\vp|^2 + |\pa_s\vp|^2\right)\,ds
	\end{split}
	\]
	and using also $\int |\pa_s \vp|^2\,ds = - \int \lgl \vp, \pa_s^2 \vp\rgl\,ds \le \tfrac12 \int \tfrac1\eta |\vp|^2 + \eta |\pa_s^2\vp|^2\,ds$ for any $\eta>0$ we conclude that
	\[
	\begin{split}
	\int |\pa_s^2\vp|^2 \,ds 
	&\le C(M,\ga) \int |\vp|^2 + |\pa_s\vp|^2 + |D_s^2\vp|^2 \,ds \\
	&\le C(M,\ga,\ep) \int |\vp|^2 + |\na^2 \vp|^2 \,ds + \ep \int |\pa_s^2\vp|^2 \,ds.
	\end{split}
	\]
	Hence we see that
	\[
	\int |\vp|^2 + |\pa_s\vp|^2 + |\pa_s^2\vp|^2\,ds \le C(\ga) a(\vp,\vp),
	\]
	that is, $a$ is coercive on $T(\ga)^{2,2,\perp}$.
	
	Now, if $X\in T(\ga)^{2,\perp}$ is fixed, we look at the energy functional $F:T(\ga)^{2,2,\perp}\to \R$ given by
	\[
	F(\vp) \coloneqq \int \frac12 |\na^2 \vp|^2 + \frac12 |\vp|^2 -\scal{X,\vp} \,ds.
	\]
	Since $-\scal{X,\vp} \ge -\tfrac14|\vp|^2 - |X|^2$, the coercivity of $a$ implies that $F$ has a minimizer $\vp\in T(\ga)^{2,2,\perp}$. Such minimizer $\vp$ satisfies the integral Euler-Lagrange equation
	\begin{equation}\label{eq:ELP2}
	a(\vp,\psi)= \int \scal{\na^2\vp,\na^2\psi} + \scal{\vp,\psi}\,ds = \int \scal{X,\psi},
	\end{equation}
	for any $\psi \in T(\ga)^{2,2,\perp}$. If we show that $\vp \in T(\ga)^{4,2,\perp}$, we will have proved that for any $X \in T(\ga)^{2,\perp}$ there exists $\vp \in T(\ga)^{4,2,\perp}$ such that $\na^4\vp+\vp=X$, and this will prove the required surjectivity. We are going to prove that $\vp \in T(\ga)^{3,2,\perp}$ first, and then $\vp \in T(\ga)^{4,2,\perp}$.
	
	Let $\Psi \in C^\infty(\S^1;\R^n)$ be any field. By \eqref{eq:Nabla2D2} and \eqref{eq:D2Partial2} we can write that
	\[
	\scal{\pa_s^2\vp, \pa_s \Psi} = \scal{\na^2 \vp, \pa_s \Psi} - \scal{A(\vp),\pa_s \Psi},
	\]
	where $A(\vp)$ is linear in $\vp$ and contains at most first order derivatives of $\vp$. Writing $\psi\coloneqq(M^\top-\ga^\top)\Psi$ we have that
	\[
	\begin{split}
	\scal{\pa_s^2\vp, \pa_s \Psi}
	&= \scal{\na^2\vp,\na\psi} + \scal{\Psi,N_j}\scal{\na^2\vp, \pa_s N_j} + \scal{\Psi,\tau} \scal{\na^2 \vp , k}  - \scal{A(\vp),\pa_s \Psi},
	\end{split}
	\]
	understanding summation over $j$, for a local orthonormal frame $\{N_j\}$ of $TM^\perp$. Let $\eta \in C^\infty_c(\S^1\sm\{0\})$ such that $\eta\ge0$ and $\int_{\S^1} \eta\,ds =1$. Define
	\[
	\Phi(x) = (M^\top-\ga^\top) \int_0^x \psi(t) - \eta(t) \psi_0\, ds(t),
	\]
	where $\psi_0\coloneqq\int_{\S^1} \psi\,ds $. By construction we see that $\Phi \in C^\infty(\S^1;\R^n)$ and $\Phi\in TM\cap T\ga^\perp$.
	
	Since for any differentiable $\ze:\S^1\to \R^n$ we have that
	\begin{equation}\label{eq:CommuteTPerpPartial}
	\begin{split}
		\pa_s((M^\top-\ga^\top) \ze) &= (M^\top-\ga^\top)\pa_s\ze -(\pa_sN_j \tens N_j + N_j \tens \pa_s N_j)\ze -(k\tens \tau + \tau\tens k)\ze \\
		&  \eqqcolon (M^\top-\ga^\top)\pa_s\ze + B(\ze),
	\end{split}
	\end{equation}
	we get that
	\[
	\na \Phi = \psi - \eta (M^\top-\ga^\top) \psi_0 + B\left(  \int_0^x \psi - \eta \psi_0\, ds \right),
	\]
	and thus
	\[
	\na^2 \Phi = \na \psi - \pa_s\eta  (M^\top-\ga^\top) \psi_0 + (M^\top-\ga^\top) \pa_s  \left(B\left(  \int_0^x \psi - \eta \psi_0\, ds \right)\right).
	\]
	Hence finally
	\[
	\begin{split}
	\left| \int \scal{\pa_s^2\vp,\pa_s\Psi} \right|
	&\le \left| \int \scal{\na^2\vp,\na\psi} \right| + \left| \int \scal{\pa_s(A(\vp)),\Psi} \right| + C(M,\ga,\|\vp\|_{W^{2,2}})\|\Psi\|_{L^2} \\
	&\le \left|\int \scal{\na^2\vp,\na^2\Phi}\right| 
	+ C(M,\ga,\|\vp\|_{W^{2,2}},\eta)\|\Psi\|_{L^2} \\
	&= \left|\int \scal{X,\Phi} - \scal{\vp,\Phi}\right| 
	+ C(M,\ga,\|\vp\|_{W^{2,2}},\eta)\|\Psi\|_{L^2} \\
	&\le C(M,\ga,\|\vp\|_{W^{2,2}},\eta,X)\|\Psi\|_{L^2},
	\end{split}
	\]
	that implies that $\vp \in T(\ga)^{3,2,\perp}$.
	Once again let $\Psi \in C^\infty(\S^1;\R^n)$ be any field. Using \eqref{eq:CommuteTPerpPartial} twice and writing $\psi=(M^\top-\ga^\top)\Psi$ as before, we have
	\[
	\begin{split}
	(M^\top-\ga^\top) \pa_s^2\Psi 
	&= \pa_s((M^\top-\ga^\top)\pa_s\Psi) + B(\pa_s\Psi) \\
	&= \pa_s( \na\psi + (M^\top-\ga^\top)\left[ \pa_s(\scal{\Psi,N_j}N_j) + \pa_s(\scal{\Psi,\tau}\tau) \right] ) + B(\pa_s\Psi) \\
	&= \pa_s\na\psi +\pa_s \left( \scal{\Psi,N_j}\pa_s N_j +\scal{\Psi,\tau}k \right) + B(\pa_s\Psi) .
	\end{split}
	\]
	Therefore
	\[
	\begin{split}
	\bigg| \int & \scal{\pa_s^3\vp,\pa_s\Psi}  \bigg| 
	= \left| \int \scal{\pa_s^2\vp,\pa_s^2\Psi}  \right| \\
	&\le \left| \int \scal{\na^2 \vp, \pa_s^2\Psi}\right| 
	+ \left|\int \scal{\pa_s^2 (A(\vp)), \Psi}\right| \\
	&\le \left| \int \scal{\na^2 \vp, \na^2\psi}\right| 
	+ \left| \int \scal{ \pa_s \na^2 \vp , \scal{\Psi,N_j}\pa_s N_j +\scal{\Psi,\tau}k   } \right| 
	+\left| \int \scal{\na^2\vp, B(\pa_s\Psi)} \right|
	+ \left|\int \scal{\pa_s^2 (A(\vp)), \Psi}\right| \\
	&\le \left| \int \scal{X,\psi} - \scal{\vp,\psi} \right| + C(M,\ga,\|\vp\|_{W^{3,2}}) \|\Psi\|_{L^2}
	+\left| \int \scal{\na^2\vp, B(\pa_s\Psi)} \right| \\
	&\le C(M,\ga,\|\vp\|_{W^{3,2}}) \|\Psi\|_{L^2},
	\end{split}
	\]
	where the last inequality follows integrating by parts using the definition of $B$, and this implies that $\vp \in T(\ga)^{4,2,\perp}$.
\end{proof}

\begin{lemma}\label{lem:FredholmnessP>2}
	Let $p>2$. Let $\ga:\S^1\to M$ be a smooth critical point with $|k(x)|\neq0$ for any $x$. There exists $\ro>0$ such that the operator $\cL: T(\ga)^{4,p,\perp}\to  (T(\ga)^{p,\perp})^\star$ represented by the function
	\[
	\cL(\vp) = \na^2 \left(|k|^{p-2}\na^2 \vp \right) 
	+(p-2) \na^2 \left(|k|^{p-4}\scal{k,\na^2\vp}k\right)
	+ \Omega(\vp) \qquad\in T(\ga)^{p',\perp},
	\]
	where $\Om:T(\ga)^{4,p,\perp}\to T(\ga)^{p',\perp}$ is compact, is Fredholm of index zero.	
\end{lemma}

\begin{proof}
	Since $\Om:T(\ga)^{4,p,\perp}\to T(\ga)^{p',\perp}$ is compact, it is equivalent to prove that the map
	\[
	T(\ga)^{4,p,\perp}(ds) \ni\quad  \vp \,\,\mapsto\,\, \sT(\vp)\coloneqq  \na^2 \left(|k|^{p-2}\na^2 \vp \right) 
	+(p-2) \na^2 \left(|k|^{p-4}\scal{k,\na^2\vp}k\right)  +\vp  \quad  \in T(\ga)^{p',\perp}(ds),
	\]
	is Fredholm of index zero. Indeed we claim that it is actually invertible. The operator $\sT$ is clearly injective, indeed if $\sT(\vp)=0$, then integration by parts on $\int \scal{\vp,\sT(\vp)}\,ds$ yields
	\[
	0= \int |k|^{p-2} |\na^2\vp|^2 + (p-2) |k|^{p-4}\scal{k,\na^2\vp}^2 + |\vp|^2 \, ds,
	\]
	and then $\vp=0$. Hence we are left to prove the surjectivity.
	
	This time we consider $a:T(\ga)^{2,p,\perp}\times T(\ga)^{2,p,\perp}\to\R$ to be the continuous bilinear form
	\[
	a(\vp,\psi)\coloneqq\int_{\S^1} |k|^{p-2}\lgl \na^2 \vp, \na^2 \psi\rgl 
	+(p-2) |k|^{p-4}\scal{k,\na^2\vp}\scal{k,\na^2\psi}
	+ \lgl \vp,\psi\rgl\,ds.
	\]
	By hypothesis there are constants $c_1,c_2$ depending on $\ga$ such that $c_1\le |k(x)|\le c_2$ for any $x$.
	
	From the proof of \Cref{lem:FredholmnessP2} we know that
	\[
	\int |\vp|^2 + |\pa_s\vp|^2 + |\pa_s^2\vp|^2\,ds \le C(M, \ga) \int |\na^2\vp|^2 + |\vp|^2\,ds,
	\]
	for any $\vp \in T(\ga)^{2,p,\perp}$. Therefore
	\[
	\int |\vp|^2 + |\pa_s\vp|^2 + |\pa_s^2\vp|^2\,ds \le C(M, \ga) a(\vp,\vp),
	\]
	for any $\vp \in T(\ga)^{2,p,\perp}$. It follows that if $X\in T(\ga)^{p',\perp}$ is fixed, the convex functional $F:T(\ga)^{2,p,\perp}\to\R$ defined by
	\[
	F(\vp) = \int \frac12 |k|^{p-2}|\na^2\vp|^2 +\frac{p-2}{2}|k|^{p-4}\scal{k,\na^2\vp}^2 +\frac12|\vp|^2 -\scal{X,\vp}\,ds
	\]
	has a unique minimizer $\vp\in T(\ga)^{2,p,\perp}$. Such minimizer $\vp$ satisfies
	\begin{equation*}
	a(\vp,\psi) = \int \scal{X,\Phi}\,ds,
	\end{equation*}
	for any $\Phi \in T(\ga)^{2,p,\perp}$.
	If we show that $\vp \in T(\ga)^{4,p,\perp}$, we will have proved that for any $X \in T(\ga)^{p',\perp}$ there exists $\vp \in T(\ga)^{4,p,\perp}$ such that $\sT(\vp)=X$, and this will prove the required surjectivity of $\sT$. We are going to show that $\vp \in T(\ga)^{3,p,\perp}$ first, and then $\vp \in T(\ga)^{4,p,\perp}$.
	
	Let $\Psi \in C^\infty(\S^1;\R^n)$ be any field and let $\psi\coloneqq(M^\top-\ga^\top)\Psi$. Let $\nu \in C^{\infty}(\S^1;\R^n)$ be the normal field along $\ga$ defined by
	\[
	\nu = \frac{|k|^{2-p}}{p-1} \left\lgl\psi,\frac{k}{|k|} \right\rgl \frac{k}{|k|}+ |k|^{2-p}\sum_{i=1}^{m-2} \scal{\psi,e_i} e_i,
	\]
	where $\{e_1,...,e_{m-2},\tfrac{k}{|k|}\}$ is a fixed orthonormal frame of the normal bundle $T\ga^\perp$ along $\ga$ in $M$. We remark that $\nu \in TM\cap T\ga^\perp$.
	Letting also $\eta \in C^\infty_c(\S^1\sm\{0\})$ such that $\eta\ge0$ and $\int_{\S^1} \eta\,ds =1$, we define
	\[
	\nu_0 = \int_{\S^1} \nu\,ds , \qquad\qquad \Phi(x) = (M^\top-\ga^\top)\int_0^x \nu(y) -\eta(y) \nu_0\, ds(y).
	\]
	In this way, in the notation of \eqref{eq:CommuteTPerpPartial} we have
	\[
	\begin{split}
	\na \Phi 
	&= (M^\top-\ga^\top) \left[ (M^\top-\ga^\top)(\nu-\eta \nu_0) +B\left(\int_0^x \nu(x) -\eta(x) \nu_0\right)   \right] \\
	&= \nu -\eta (M^\top-\ga^\top)\nu_0 
	-\left\lgl N_j,\int_0^x \nu -\eta\nu_0 \right\rgl (M^\top-\ga^\top)(\pa_s N_j)
	-\left\lgl \tau,\int_0^x \nu -\eta\nu_0 \right\rgl k.
	\end{split}
	\]
	By construction
	\[
	\psi = |k|^{p-2}\nu +(p-2) |k|^{p-2} \left\lgl \nu ,\frac{k}{|k|} \right\rgl \frac{k}{|k|},
	\]
	and then
	\[
	\begin{split}
	|k|^{p-2}&\na\Phi +(p-2) |k|^{p-2} \left\lgl \na\Phi ,\frac{k}{|k|} \right\rgl \frac{k}{|k|} = \\
	&= \psi 
	+ |k|^{p-2} \left[   -\eta (M^\top-\ga^\top)\nu_0 
	-\left\lgl N_j,\int_0^x \nu -\eta\nu_0 \right\rgl (M^\top-\ga^\top)(\pa_s N_j)
	-\left\lgl \tau,\int_0^x \nu -\eta\nu_0 \right\rgl k  \right] +\\
	&\indent + (p-2)|k|^{p-4} \bigg\lgl  -\eta (M^\top-\ga^\top)\nu_0 
	-\left\lgl N_j,\int_0^x \nu -\eta\nu_0 \right\rgl (M^\top-\ga^\top)(\pa_s N_j)
	+\\
	&\indent -\left\lgl \tau,\int_0^x \nu -\eta\nu_0 \right\rgl k , k  \bigg\rgl k .
	\end{split}
	\]
	Therefore, as in the proof of \Cref{lem:FredholmnessP2}, we estimate
	\[
	\begin{split}
	\left| \int \scal{\pa_s^2\vp,\pa_s\Psi} \right|
	&\le \left| \int \scal{\na^2\vp,\na\psi} \right| 
	+ C(M,\ga,\|\vp\|_{W^{2,2}},\eta)\|\Psi\|_{L^2}  \\
	&\le \left| \int |k|^{p-2}\scal{\na^2\vp,\na^2\Phi} +(p-2)|k|^{p-4}\scal{k,\na^2\vp}\scal{k,\na^2\Phi}  \right|
	+ C(M,\ga,\|\vp\|_{W^{2,2}},\eta)\|\Psi\|_{L^2} \\
	&= \left| \int \scal{X,\Phi} - \scal{\vp,\Phi} \right| + C(M,\ga,\|\vp\|_{W^{2,2}},\eta)\|\Psi\|_{L^2} \\
	& \le C(M,\ga,\|\vp\|_{W^{2,2}},\eta)\|\Psi\|_{L^2} ,
	\end{split}
	\]
	that implies $\vp \in T(\ga)^{3,p,\perp}$.
	
	Now the definition of $\nu$ implies that
	\[
	\scal{\na \nu, e_i} = |k|^{2-p} \scal{\na\psi, e_i} + A_i(\psi) ,
	\qquad
	\left\lgl \na\nu , \frac{k}{|k|} \right\rgl  = \frac{|k|^{2-p}}{p-1} \left\lgl \na\psi , \frac{k}{|k|} \right\rgl +A_k(\psi),
	\]
	where $A_i(\psi),A_k(\psi)$ depend on $\ga$ and $M$, and they depend linearly on $\psi$ and they are independent of the derivatives of $\psi$. Therefore we have
	\[
	\na\psi = |k|^{p-2} \na \nu +(p-2) |k|^{p-4} \scal{k,\na \nu}k + A_0(\psi),
	\]
	with $A_0$ having the same properties of $A_i,A_k$. Finally we can estimate
	\[
	\begin{split}
	\bigg| \int \scal{\pa_s^3\vp,\pa_s\Psi}  \bigg| 
	&= \left| \int \scal{\pa_s^2\vp,\pa_s^2\Psi}  \right| \\
	&\le \left| \int \scal{\na^2 \vp, \na^2\psi}\right| 
	+ C(M,\ga,\|\vp\|_{W^{3,2}}) \|\Psi\|_{L^2} \\
	&\le \left| \int \left\lgl
	\na^2\vp ,
	\na\left(  |k|^{p-2} \na \nu +(p-2) |k|^{p-4} \scal{k,\na \nu}k  \right)
	\right\rgl \right|
	+\left| \int \scal{\na^2\vp, \na(A_0(\psi))} \right|+\\
	&\indent 
	+ C(M,\ga,\|\vp\|_{W^{3,2}}) \|\Psi\|_{L^2} \\
	&\le \left|\int |k|^{p-2}\scal{\na^2\vp,\na^2\nu} +(p-2)|k|^{p-4}\scal{k,\na^2\vp}\scal{k,\na^2\nu} \right|
	+\left| \int \scal{\na^3\vp, A_0(\psi)} \right| +\\
	&\indent 
	+ C(M,\ga,\|\vp\|_{W^{3,2}}) \|\Psi\|_{L^2} \\
	&\le \left|\int \scal{X,\nu} -\scal{\vp,\nu}  \right|
	+ C(M,\ga,\|\vp\|_{W^{3,2}}) \|\Psi\|_{L^2} \\
	& \le C(M,\ga,\|\vp\|_{L^2}) \|\nu\|_{L^2}  + C(M,\ga,\|\vp\|_{W^{3,2}}) \|\Psi\|_{L^2} \\
	&\le  C(M,\ga,\|\vp\|_{W^{3,2}}) \|\Psi\|_{L^2},
	\end{split}
	\]
	and we have proved that $\vp \in T(\ga)^{4,p,\perp}$.
\end{proof}

A last fact needed for applying \Cref{cor:LojaFunctionalAnalytic} is the analyticity of the operators, as stated in the next lemma, for which we mainly refer to \cite{DaPoSp16}. Here the analyticity of the ambient comes into play.

\begin{lemma}\label{lem:Analyticity}
	Let $\ga:\S^1\to M$ be a smooth regular curve and let $\ro>0$ be given by \Cref{lem:Variation}. Let $p\ge2$. Suppose that $(M,g)$ is an analytic complete Riemannian manifold endowed with an analytic metric tensor $g$.
	\begin{enumerate}
		\item If $p=2$, then the maps
		\[
		E: B_\ro(0) \to \R, \qquad\qquad\de E: B_\ro(0) \to (T(\ga)^{2,\perp})^\star,
		\]
		are analytic.
		\item  If $p>2$ and $|k(x)|\neq0$ for any $x$, then the maps
		\[
		E: B_\ro(0) \to \R, \qquad\qquad\de E: B_\ro(0) \to (T(\ga)^{p,\perp})^\star,
		\]
		are analytic, up to decrease $\ro$.
	\end{enumerate}
\end{lemma}

\begin{proof}
	We adopt the notation used in \eqref{eq:FirstVarP2} and \eqref{eq:FirstVarP>2}.
	For a fixed $\vp(x) \in T_{\ga(x)}M$ we have that $\ga_\vp(x)=\exp_{\ga(x)}(\vp(x)) = \si_{\vp(x)}(1)$, where $\si_{\vp(x)}$ is the geodesic starting at $\ga(x)$ with initial velocity $\si_{\vp(x)}'(0)=\vp(x)$. As the manifold and the metric are assumed to be analytic, so are the connection $D$ and the Christoffel symbols $\Ga_{ij}^k$ on $M$. It follows that, as $\si_{\vp(x)}$ solves a semi-linear ordinary differential equation with analytic coefficients, it depends analytically on the initial data. In particular the exponential map is analytic and the dependence of $\ga_\vp(x)$ on $\vp(x) \in T_{\ga(x)}M$ is analytic.
	Also, since the exponential map on $M$ is analytic, so is its differential, and it follows that $d[\exp_{\ga(x)}]$ is analytic as a map defined on $T(\ga)^{4,p,\perp}$.
	
	Now that we know that the map $\vp\mapsto\ga_\vp$ is analytic, following the exhaustive proof of \cite[Lemma 3.4]{DaPoSp16}, one can check that the formulas for $E$ and $\de E$ in \eqref{eq:FirstVarP2} and \eqref{eq:FirstVarP>2} are sums of compositions of analytic functions of the parametrization $\ga_\vp$ (in case $p>2$ we use that $|k_{\ga_\vp}|$ never vanishes for $\ro$ small enough). 
\end{proof}

We now have all the ingredients for applying \Cref{cor:LojaFunctionalAnalytic}, thus getting the \L ojasiewicz-Simon gradient inequality for the functional $E$.

\begin{cor}\label{cor:LojaNormalePGenerico}
	Suppose that $(M,g)$ is an analytic complete Riemannian manifold endowed with an analytic metric tensor $g$. Let $p\ge2$. Let $\ga:\S^1\to M$ be a smooth critical point of $\sE$. In case $p>2$ assume that $|k(x)|\neq0$ for any $x$. There exist $C,\ro>0$ and $\te\in(0,\frac12]$ such that
	\begin{equation}\label{eq:LojaNormalePGenerico}
	|E(\vp)-E(0)|^{1-\te} \le C \|  
	\na_{T(\ga)^{p'},T(\ga)^p} E (\vp)
	\|_{L^{p'}(ds_{\ga_\vp})},
	\end{equation}
	for any $\vp\in B_\ro(0)\con T(\ga)^{4,p,\perp}$, where $\ga_\vp(\cdot)=\Phi(1,\cdot)$ and $\Phi$ is the variation of $\ga$ given by $\vp$.
\end{cor}

\begin{proof}
	Collecting the results of \Cref{prop:VariationsP2}, \Cref{prop:VariationsP>2}, \Cref{lem:FredholmnessP2}, \Cref{lem:FredholmnessP>2}, and \Cref{lem:Analyticity}, the statement follows from the direct application of \Cref{cor:LojaFunctionalAnalytic} taking $V=T(\ga)^{4,p,\perp}$, $Z=T(\ga)^{p,\perp}$, and $\ro_0>0$ depending on $\ga$ given by \Cref{lem:Variation}.
\end{proof}

As outlined in the strategy of \Cref{sec:Rn}, we can exploit the geometric invariance of the energy for extending the inequality \eqref{eq:LojaNormalePGenerico} to the functional $\bE$ and generic variations in $T(\ga)^{4,p}$. We need the following reparametrization result first.

\begin{lemma}\label{lem:ReparametrizationManifold}
	Let $\ga:\S^1\to M$ be a smooth immersion and $p\ge2$. Let $\ro_0(\ga)>0$ be given by \Cref{lem:Variation}. Then
	for any $\ro\in(0,\ro_0)$ there is $\si>0$ such that for any $\psi \in B_\si(0)\con T(\ga)^{4,p}$ there exists a diffeomorphism $L:\S^1\to \S^1$ of class $W^{4,p}$ and $\vp \in B_\ro(0)\con  T(\ga)^{4,p,\perp}$ such that
	\[
	\ga_\psi\circ L = \ga_\vp,
	\]
	where $\ga_\vp(\cdot)=\Phi(1,\cdot)$, $\ga_\psi(\cdot)=\Psi(1,\cdot)$, and $\Phi,\Psi$ are the variation of $\ga$ given by $\vp,\psi$ respectively.
\end{lemma}

\begin{proof}
	By compactness there exists $\tau>0$ such that $\ga|_{(x-\tau,x+\tau)}:(x-\tau,x+\tau)\to M$ is an embedding for any $x\in \S^1$. Fix $\ro\in(0,\ro_0)$. If we choose $\si'>0$ is sufficiently small, depending only on $\ga$, and any $\psi \in B_{\si'}(0)\con T(\ga)^{4,p}$, we have that
	\[
	\ga_\psi\left(x-\frac\tau2,x+\frac\tau2\right)\con U_x,
	\]
	where $U_x\con M$ is an open neighborhood of $\ga\left(\left[x-\tfrac\tau2,x+\tfrac\tau2\right]\right)$ parametrized by the exponential map restricted to the normal bundle of $\ga$. More precisely, there exists an open connected set
	\[
	\Om_x \con \bcup_{y\in\left(x-\tau,x+\tau\right)} T_{\ga(y)}\ga^\perp,
	\]
	containing the origin of $T_{\ga(y)}\ga^\perp$ for any $y\in\left(x-\tau,x+\tau\right)$ such that any $q\in U_x$ can be uniquely written as $q= \exp^\perp (v_q)$ for some $v_q\in \Om_x$, where $\exp^\perp$ is the restriction of the exponential map to the normal bundle of $\ga$.
	
	Hence for any $y\in \left(x-\tfrac\tau2,x+\tfrac\tau2\right)$ there exists a unique $G(y)\in \left(x-\tau,x+\tau\right)$ and a unique $\vp \in T_{\ga(G(y))}\ga^\perp$ such that
	\[
	\ga_\psi (y) = \exp^\perp (\vp(G(y))).
	\]
	By defining $L=G^{-1}:G\left(x-\tfrac\tau2,x+\tfrac\tau2\right) \to \left(x-\tfrac\tau2,x+\tfrac\tau2\right)$ we see that
	\[
	\ga_\psi \circ L \, (y) = \exp(\vp(y)),
	\]
	for any $y \in G\left(x-\tfrac\tau2,x+\tfrac\tau2\right)$. Moreover since for $y\in \left(x-\tfrac\tau2,x+\tfrac\tau2\right)$ we can write explicitly
	\[
	G(y) = \left(\ga|_{(x-\tau,x+\tau)} \right)^{-1} \circ \pi \circ (\exp^\perp)^{-1}\circ \ga_\psi \, (y),
	\]
	where $\pi$ is the projection of the normal bundle, we see that $G$ is of class $W^{4,p}$, and then so is $L$. Also
	\[
	\vp(y) = (\exp^\perp)^{-1}\circ \ga_\psi \circ L \, (y),
	\]
	and then $\vp$ is of class $W^{4,p}$.	
	By arbitrariness of $x$, one can then define a normal field $\vp$ along $\ga$ and a diffeomorphism $L$ of $\S^1$ satisfying $\ga_\psi \circ L = \ga_\vp$.
	
	Finally, it follows from the construction that if $\psi_n \in T(\ga)^{4,p}$ converges to $0$, then the corresponding $L_n$ converges to the identity in $W^{4,p}$, and then also $\vp_n\to0$ in $W^{4,p}$. This proves that for the chosen $\ro$, taking a suitable $0<\si\le \si'$, for any $\psi\in B_\si(0)\con T(\ga)^{4,p}$ the resulting $\vp\in T(\ga)^{4,p,\perp}$ has norm less than the desired $\ro$.
\end{proof}

For the convenience of the reader, let us recall here that for a fixed smooth curve $\ga:\S^1\to M$ and $\ro>0$ sufficiently small we have that
\begin{equation*}
\begin{split}
\de \bE(\vp)[\psi] &= 
\left\lgl
\cT^\star\left(
\na_{s_{\ga_\vp}}^2 k_{\ga_\vp} 
+
\frac{1}{2} |k_{\ga_\vp}|^2k_{\ga_\vp} - k_{\ga_\vp} 
+ R (k_{\ga_\vp},\tau_{\ga_\vp})\tau_{\ga_\vp}
\right)
, \psi
\right\rgl_{L^2(ds_{\ga_\vp}),L^2(ds_{\ga_\vp})}
\end{split}
\end{equation*}
for any $\vp,\psi \in B_\ro(0)\con T(\ga)^{4,2}$, where $\ga_\vp(\cdot)=\Phi(1,\cdot)$, $\Phi$ is the variation of $\ga$ given by $\vp$, and $\cT:T_{\ga(x)}M\to T_{\Phi(1,x)}M$ is the function $\cT(\psi)=d[\exp_{\ga(x)}]_{\vp}(\psi)$ and $\cT^\star$ is its adjoint.
We therefore write
\[
\na_{T(\ga)^2,T(\ga)^2} \bE(\vp) = |\pa_x \ga_\vp | \cT^\star\left(
\na_{s_{\ga_\vp}}^2 k_{\ga_\vp} 
+
\frac{1}{2} |k_{\ga_\vp}|^2k_{\ga_\vp} - k_{\ga_\vp} 
+ R (k_{\ga_\vp},\tau_{\ga_\vp})\tau_{\ga_\vp}
\right).
\]

Analogously if $\ga:\S^1\to M$ is a smooth critical point of $\sE$ for some $p>2$ and $|k(x)|\neq0$ for any $x$, for $\ro>0$ sufficiently small we have that
\begin{equation*}
\begin{split}
\de \bE(\vp)[\psi] &= 
\left\lgl
\cT^\star\left(
\na_{s_{\ga_\vp}}^2 |k_{\ga_\vp}|^{p-2} k_{\ga_\vp} 
+
\frac{1}{p'} |k_{\ga_\vp}|^p k_{\ga_\vp} - k_{\ga_\vp} 
+ R (|k_{\ga_\vp}|^{p-2} k_{\ga_\vp},\tau_{\ga_\vp})\tau_{\ga_\vp}
\right)
, \psi
\right\rgl_{L^{p'}(ds_{\ga_\vp}),L^p(ds_{\ga_\vp})}
\end{split}
\end{equation*}
for any $\vp,\psi \in B_\ro(0)\con T(\ga)^{4,p}$, where $\ga_\vp$ and $\cT$ are as above. Therefore
\[
\na_{T(\ga)^{p'},T(\ga)^p} \bE(\vp) = |\pa_x \ga_\vp | \cT^\star\left(
\na_{s_{\ga_\vp}}^2 |k_{\ga_\vp}|^{p-2} k_{\ga_\vp} 
+
\frac{1}{p'} |k_{\ga_\vp}|^p k_{\ga_\vp} - k_{\ga_\vp} 
+ R (|k_{\ga_\vp}|^{p-2} k_{\ga_\vp},\tau_{\ga_\vp})\tau_{\ga_\vp}
\right).
\]

As anticipated, using \Cref{lem:ReparametrizationManifold} we can now improve \eqref{eq:LojaNormalePGenerico} to fields in $T(\ga)^{4,p}$.

\begin{cor}\label{cor:FullLojasiewiczPGenerico}
	Suppose that $(M,g)$ is an analytic complete Riemannian manifold endowed with an analytic metric tensor $g$. Let $p\ge2$. Let $\ga:\S^1\to\R^n$ be a smooth critical point of $\sE$. In case $p>2$ assume that $|k(x)|\neq0$ for any $x$. There exist $C,\si>0$ and $\te\in(0,\frac12]$ such that
	\begin{equation}\label{eq:FullLojaPGenerico}
	|\bE(\psi)-\bE(0)|^{1-\te} \le C \|  \na_{T(\ga)^{p'},T(\ga)^p} \bE(\psi) \|_{L^{p'}(ds_{\ga_\psi})},
	\end{equation}
	for any $\psi\in B_\si(0)\con T(\ga)^{4,p}$.
\end{cor}

\begin{proof}
	Let $\ro$ be as in \Cref{cor:LojaNormalePGenerico}. Without loss of generality $\ro<\ro_0(\ga)$, where $\ro_0(\ga)$ is given by \Cref{lem:Variation}, and then let $\si>0$ be the corresponding radius given by \Cref{lem:ReparametrizationManifold}. Let $\psi\in B_\si(0)\con T(\ga)^{4,p}$ and let $L:\S^1\to \S^1$ and $\vp \in B_\ro(0)\con  T(\ga)^{4,p,\perp}$ such that $\ga_\psi\circ L = \ga_\vp$ in the notation of \Cref{lem:ReparametrizationManifold}. Then
	\[
	\bE(\psi)=\sE(\ga_\psi)=\sE((\ga_\psi)\circ L ) = \sE (\ga_\vp)= E(\vp).
	\]
	Moreover by compactness and continuity of $w\mapsto d[\exp_{\ga(x)}]_{w}$ we see that there exists $R=R(\ga)>0$ such that for any $x\in \S^1$ and any $w\in T_{\ga(x)}M$ with $|w|\le R$ it holds that $d[\exp_{\ga(x)}]_{w}$ is invertible with $\|d[\exp_{\ga(x)}]_{w}\| \le C(\ga)$, and then the same holds for $(d[\exp_{\ga(x)}]_{w})^\star$  and their inverse.
%
	Up to take smaller $\ro$ and $\si$, we can assume that $|\vp|\le R$ and $|\psi|\le R$.
	
	Understanding that if $p=2$ then $|v|^{p-2}\equiv1$ for any vector $v$, we deduce that
	\[
	\begin{split}
	\|&\na_{T(\ga)^{p'},T(\ga)^p} E(\vp) \|^{p'}_{L^{p'}(ds_{\ga_\vp})}\le\\
	&\le  \int_{\S^1} \left| (d[\exp_{\ga(x)}]_{\vp})^\star\left(
	\na_{{\ga_\vp}}^2 |k_{\ga_\vp}|^{p-2} k_{\ga_\vp} 
	+\frac{1}{p'} |k_{\ga_\vp}|^p k_{\ga_\vp} - k_{\ga_\vp} 
	+ R (|k_{\ga_\vp}|^{p-2} k_{\ga_\vp},\tau_{\ga_\vp})\tau_{\ga_\vp}	
	\right)  \right|^{p'} \,ds_{\ga_\vp}  \\
	&\le C \int_{\S^1} \left|
	\na_{{\ga_\vp}}^2 |k_{\ga_\vp}|^{p-2} k_{\ga_\vp} 
	+
	\frac{1}{p'} |k_{\ga_\vp}|^p k_{\ga_\vp} - k_{\ga_\vp} 
	+ R (|k_{\ga_\vp}|^{p-2} k_{\ga_\vp},\tau_{\ga_\vp})\tau_{\ga_\vp}
	\right|^{p'} \,ds_{\ga_\vp}  \\
	&=  C \int_{\S^1} \left|  
	\na_{{\ga_\psi\circ L}}^2 | k_{\ga_\psi\circ L}|^{p-2} k_{\ga_\psi\circ L} 
	+
	\frac{1}{p'} |k_{\ga_\psi \circ L}|^p k_{\ga_\psi \circ L} - k_{\ga_\psi \circ L} 
	+ R (|k_{\ga_\psi\circ L}|^{p-2} k_{\ga_\psi \circ L},\tau_{\ga_\psi \circ L})\tau_{\ga_\psi \circ L}
	\right|^{p'} \,ds_{\ga_\psi \circ L } \\
	& =  C\int_{\S^1} \left|  
	\na_{{\ga_\psi}}^2 |k_{\ga_\psi }|^{p-2} k_{\ga_\psi} 
	+
	\frac{1}{p'} |k_{\ga_\psi }|^p k_{\ga_\psi } - k_{\ga_\psi } 
	+ R (|k_{\ga_\psi }|^{p-2} k_{\ga_\psi },\tau_{\ga_\psi })\tau_{\ga_\psi }
	\right|^{p'} \,ds_{\ga_\psi  } \\
	&=  C \int_{\S^1} \bigg| 
	\left[ (d[\exp_{\ga(x)}]_{\psi})^\star \right]^{-1} \cdot \\
	&\qquad\qquad\qquad \cdot (d[\exp_{\ga(x)}]_{\psi})^\star
	\left( 
	\na_{{\ga_\psi}}^2 |k_{\ga_\psi }|^{p-2} k_{\ga_\psi} 
	+
	\frac{1}{p'} |k_{\ga_\psi }|^p k_{\ga_\psi } - k_{\ga_\psi } 
	+ R (|k_{\ga_\psi }|^{p-2} k_{\ga_\psi },\tau_{\ga_\psi })\tau_{\ga_\psi }
	\right)
	\bigg|^{p'} \,ds_{\ga_\psi  } \\
	& \le C\int_{\S^1} \left| 
	(d[\exp_{\ga(x)}]_{\psi})^\star
	\left( 
	\na_{{\ga_\psi}}^2 |k_{\ga_\psi }|^{p-2} k_{\ga_\psi} 
	+
	\frac{1}{p'} |k_{\ga_\psi }|^p k_{\ga_\psi } - k_{\ga_\psi } 
	+ R (|k_{\ga_\psi }|^{p-2} k_{\ga_\psi },\tau_{\ga_\psi })\tau_{\ga_\psi }
	\right)
	\right|^{p'} \,ds_{\ga_\psi  } \\
	&\le C  \|  \na_{T(\ga)^{p'},T(\ga)^p} \bE(\psi) \|^{p'}_{L^{p'}(ds_{\ga_\psi})}
	\end{split}
	\]
	and therefore \eqref{eq:FullLojaPGenerico} readily follows from \eqref{eq:LojaNormalePGenerico}.
\end{proof}


\bigskip

\section{Convergence of $p$-elastic flows into manifolds}\label{sec:ConvergenceFinal}

In this final part we apply \Cref{cor:FullLojasiewiczPGenerico} and the strategy presented in $\R^n$ in \Cref{sec:Rn} to prove full convergence of the gradient flow of $\sE$ out of its sub-convergence.

Let us start by recalling the definitions of the gradient flows we are considering. For a given smooth curve $\ga_0:\S^1\to M$, we say that $\ga:[0,T)\times \S^1\to M$ is the solution of the gradient flow of $\sE$ with datum $\ga_0$ if it classically satisfies the equation
\begin{equation}\label{eq:DefnFlowPGenerico}
\begin{cases}
\pa_t \ga = -\left( \na^2 |k|^{p-2} k + \frac{1}{p'} |k|^p k -k + R(|k|^{p-2} k,\tau)\tau \right) & \mbox{ on } [0,T)\times \S^1,\\
\ga(0,\cdot)=\ga_0(\cdot)   & \mbox{ on } \S^1,
\end{cases}
\end{equation}
where we understand that $|k|^{p-2}\equiv1$ in case $p=2$.

In order to prove the convergence of the flow we need a local existence and uniqueness result. This is contained in the next theorem, whose proof is based on rather classical arguments about parabolic equations, and then we will just comment on that.

\begin{thm}\label{thm:LocalExistencePGenerico}
	Let $p\ge2$ and let $\ga_0:\S^1\to M$ be a smooth curve. If $p>2$ assume also that $|k_{\ga_0}(x)|\neq0$ for any $x$. There exists $T>0$ and a unique $\ga:[0,T)\times \S^1\to M$ such that $\ga(t,x)$ is a smooth solution of \eqref{eq:DefnFlowPGenerico}.
\end{thm}

	The outline of the proof of \Cref{thm:LocalExistencePGenerico} goes as follows.
	We can fix finitely many local charts $\{(U_i,\Phi_i)\}_{i=1}^N$ on $M$ such that $\ga_0(\S^1)\con \cup_i U_i$. We can choose such charts so that for some $r_i>0,x_i\in \S^1$ we have $U_i= B_{r_i}(\ga_0(x_i))$, $\ga_0(\tfrac12(x_i+x_{i+1}))\in U_i\cap U_{i+1}$, and $B_{\frac{r_i}{2}}(x_i)\cap U_j=\emptyset$ for $i\neq j$, where we understand that $N+1=1$. Fix also points $a_i,b_i\in \S^1$ such that
	\[
	a_i<\tfrac12 (x_i+x_{i-1}) < b_{i-1} < x_i < a_{i+1},
	\]
	so that $\ga_0(a_i),\ga_0(b_{i-1}) \in U_{i-1}\cap U_{i}$. Next we consider the curves $\ga_{0,i}\coloneqq\Phi_i\circ \ga_0|_{a_i,b_i}:(a_i,b_i)\to \R^m$.
	
	Consider now $p=2$. In such local coordinates one checks that \eqref{eq:DefnFlowPGenerico} in terms of $\ga_i=\Phi_i\circ \ga$ becomes
	\begin{equation}\label{eq:FlowInChartP2}
	\begin{cases}
	\pa_t \ga_i^l = -\frac{1}{|\pa_x \ga_i|^4}\pa^4_x \ga_i^l + X^l(\ga_i,\pa_x\ga_i,\pa^2_x\ga_i,\pa_x^3\ga_i)  & l=1,...,m,\\
	\ga_i(0,\cdot)= \ga_{0,i},
	\end{cases}
	\end{equation}
	where $X^l:U_1\times U_2\times U_3\times U_4\to\R^m$ is smooth and $U_j\con\R^m$ is a suitable open bounded set for any $j=1,...,4$. It is possible to prove local existence and uniqueness with continuity with respect to the datum for \eqref{eq:FlowInChartP2}, thus getting a solution $\ga_i:[0,T_i)\times (a_i,b_i)\to\R^m$; indeed \eqref{eq:FlowInChartP2} is a parabolic quasi-linear system and one can replicate the very flexible strategy of \cite{MaMa12}, as also pointed out by the authors. Now if $T_i$ is sufficiently small, one has that $\Phi_i^{-1}\circ\ga_i$ makes sense and solves \eqref{eq:DefnFlowPGenerico} up to reparametrization on $[0,T_i)\times (a_i,b_i)$. Also, a solution of the flow in \eqref{eq:DefnFlowPGenerico} is independent of the parametrization of the curve at time $t$, and therefore $\Phi_i^{-1}\circ\ga_i$ and $\Phi_{i-1}^{-1}\circ\ga_{i-1}$ coincide on $(a_i,b_{i-1})$ up to a reparametrization on the interval $(x_{i-1},x_i)$.
	Hence we can glue together the solutions obtaining a flow solving \eqref{eq:DefnFlowPGenerico} as stated in \Cref{thm:LocalExistencePGenerico}.
	
	If instead $p>2$, assuming $k\neq0$, rewriting $k= |\pa_x\ga|^{-2}\ga^\perp(\pa_x^2\ga)$, where $\ga^\perp=M^\top-\ga^\top$, the evolution equation becomes
	\[
	\pa_t \ga = -\frac{1}{|\pa_x\ga|^{2p}} \pa_x^2\left( | \ga^\perp \pa_x^2 \ga |^{p-2} \pa_x^2 \ga  \right)
	+ X(\ga,\pa_x\ga,\pa_x^2\ga,\pa_x^3\ga),
	\]
	where $X:U_1\times U_2\times U_3\times U_4\to \R^n$ always denotes a smooth function and $U_j\con\R^m$ is a suitable open bounded set for any $j=1,...,4$. Exploiting the fact that by hypothesis the curvature of the datum does not vanish, the strong parabolicity of the system is preserved for short times and one proves \Cref{thm:LocalExistencePGenerico} by means of the same techniques.

\begin{remark}\label{rem:Riparametrizzazioni}
	Let us remark here a well known fact about the uniqueness up to reparametrizations in the theory of evolution equation of geometric nature. Let us say that $\ga:[0,T)\times\S^1\to M$ solves
	\begin{equation}\label{eq:CauchyCanonica}
	\begin{cases}
	\pa_t \ga(t,x) = V_\ga(t,x) ,\\
	\ga(0,\cdot)=\ga_0(\cdot),
	\end{cases}
	\end{equation}
	everything is smooth, and $V_\ga \in T\ga^\perp$ for any time is a normal velocity field computed in terms of the curve $\ga$ at any time, and $\ga$ is the unique solution of \eqref{eq:CauchyCanonica}. Suppose that the velocity is geometric in the sense that if $\chi:[0,T)\times\S^1\to\S^1$ satisfies that $\chi(t,\cdot)$ is a diffeomorphism for any $t$ and $\si(t,x)\coloneqq \ga(t,\chi(t,x))$, then $V_{\si}(t,x) = V_\ga(t,\chi(t,x))$. Observe that this is exactly the case of the family of flows we are considering.
	In such a case, it is immediate to check that $\si$ solves
	\[
	\begin{cases}
	\pa_t \si (t,x) = V_\si(t,x) + W(t,x)\tau_\si(t,x),\\
	\si(0,\cdot) = \ga_0(\chi(0,\cdot)),
	\end{cases}
	\]
	and $W$ can be computed explicitly in terms of $\chi$ and $\ga$.
	In complete analogy, if $\be:[0,T)\times \S^1 \to M$ is given and solves
	\[
	\begin{cases}
	\pa_t \be(t,x)= V_\be(t,x) + w(t,x)\tau_\be(t,x),\\
	\be(0,\cdot)= \ga_0 (\chi_0(\cdot)),
	\end{cases}
	\]
	where $\chi_0:\S^1\to\S^1$ is a diffeomorphism, then letting $\psi:[0,T)\times\S^1\to\S^1$ be the smooth solution of
	\[
	\begin{cases}
	\pa_t\psi(t,x) = - |(\pa_x\be)(t,\psi(t,x))|^{-1} w(t,\psi(t,x)),\\
	\psi(0,\cdot)=\chi_0^{-1}(\cdot),
	\end{cases}
	\]
	it immediately follows that $\hat\ga(t,x)\coloneqq \be(t,\psi(t,x))$ solves \eqref{eq:CauchyCanonica}, and then $\hat\ga=\ga$ by uniqueness.
	We shall use this sort of geometric uniqueness up to reparametrizations several times.
\end{remark}

\medskip
	
The next proposition contains a result about parabolic estimates we will need for the proof of the main theorem.
	
\begin{prop}\label{prop:StimaParabolica}
	Let $\ga:[0,\tau)\times \S^1 \to M\con \R^n$ be a smooth solution of \eqref{eq:DefnFlowPGenerico} and let $\hat \ga$ be its constant speed reparametrization. Let $\Ga:\S^1\to M\con \R^n$ be a fixed smooth curve parametrized with constant speed. If $p>2$ assume that $|k_{\Ga}(x)|\neq0$ for any $x$. Then there is $\bar{\si}=\bar\si(k_\Ga)>0$ depending only on $k_\Ga$ such that if
	\[
	\|\hat \ga(t,\cdot) - \Ga \|_{W^{4,p}} \le \bar{\si},
	\]
	for any $t\in[0,\tau)$, then
	\begin{equation*}
		\| k_{\hat \ga}(t,\cdot) \|_{W^{m,2}} \le   C(m,\|k_\Ga\|_{W^{2,2}},\bar{\si},\cU, \La) (1+\| k_{\hat \ga}(0,\cdot) \|_{W^{m,2}}),
	\end{equation*}
	for any $t\in[0,\tau)$ and any $m\in\N$, where $\hat \ga$ is the constant speed reparametrization of $\ga$. Also, $\cU=\cU(\bar{\si})$ is a bounded neighborhood of $\Ga$ in $M$ such that the flow $\ga$ is contained in $\cU$ for any $t\in[0,\tau)$, and $\La\coloneqq \sup_{x\in\cU} |B_x|$ is the maximal norm assumed by the second fundamental form of $M\hookrightarrow \R^n$ on $\cU$.
\end{prop}

In \Cref{prop:StimaParabolica}, writing that a constant depends on an open set $\cU\con M$ is a shortcut for saying that such constant depends on the metric $g$ of $M$ on $\cU$, and thus on all the intrinsic geometric quantities depending on $g$ on $\cU$.

The proof of \Cref{prop:StimaParabolica} is a bit technical but based on classical arguments in the theory of geometric parabolic equations, and it is postponed to \Cref{app:StimaParabolica}.

\begin{remark}\label{rem:Interpolation}
	We recall separately some interpolation inequalities we shall employ.	
	For $\alpha\in(0,1)$ and $k\in\N$ with $k\ge1$ such that $s=\alpha k$ is not an integer, \cite[Example 5.15]{Lunardi} and \cite[Corollary 1.7]{Lunardi} imply that
	\[
	\|f\|_{W^{s,p}(\R)} \le c(\alpha,p,k) \|f\|^{1-\alpha}_{L^p(\R)}\|f\|^\alpha_{W^{k,p}(\R)},
	\]
	for any $f\in W^{k,p}(\R)$. By taking $\alpha\in(0,1)$ such that $\alpha k > k-1$ and it is not an integer, we have the inequality
	\[
	\|f\|_{W^{k-1,p}(\R)} \le c(\alpha,p,k) \|f\|^{1-\alpha}_{L^p(\R)}\|f\|^\alpha_{W^{k,p}(\R)}.
	\]
	Hence for a function $F\in W^{k,p}(\S^1;\R^n)$, using a continuous extension operator $T: W^{k,p}(\S^1;\R^n)\to  W^{k,p}(\R;\R^n)$ we deduce the inequality
	\begin{equation}\label{eq:Interpolation}
	\|F\|_{W^{k-1,p}(\S^1;\R^n)} \le c(\alpha,p,k) \|F\|^{1-\alpha}_{L^p(\S^1;\R^n)}\|F\|^\alpha_{W^{k,p}(\S^1;\R^n)}.
	\end{equation}
	Similarly, setting $k=1$ and thus $\alpha=s$ it holds that
	\begin{equation}\label{eq:Interpolation3}
	\|F\|_{W^{s,p}(\S^1;\R^n)} \le c(s,p) \|F\|^{1-s}_{L^p(\S^1;\R^n)}\|F\|^s_{W^{1,p}(\S^1;\R^n)},
	\end{equation}
	for any $F \in W^{1,p}(\S^1;\R^n)$.
	
	The same references also imply the following interpolation inequality. Let $k\in \N$ with $k\ge2$ and $0<\de<\be<1$; then there exists $\te'\in(0,1)$ such that
	\[
	\|f\|_{C^{k,\de}(\R)} \le c(k,\be,\ep,\te') \|f\|^{1-\te'}_{C^{k-2}(\R)} \|f\|^{\te'}_{C^{k,\be}(\R)},
	\]
	for any $f \in C^{k,\be}(\R)$. More precisely, we can choose $\te'=\frac{2+\de}{2+\be}$, so that $k+\de = (1-\te')(k-2) +\te'(k+\be)$. By suitable extension of a function $F\in C^{k,\be}(\S^1;\R^n)$, we have the inequality
	\begin{equation}\label{eq:Interpolation2}
	\|F\|_{C^{k,\de}(\S^1;\R^n)} \le c(k,\be,\ep,\te') \|F\|^{1-\te'}_{C^{k-2}(\S^1;\R^n)} \|F\|^{\te'}_{C^{k,\be}(\S^1;\R^n)}.
	\end{equation}
\end{remark}

\medskip

Let us adopt the following notation. If $\ga(t,\cdot)$ is some curve, we will denote by
\[
\cS(t,x)= \frac{2\pi}{L(\ga(t,\cdot))}\int_0^x |\pa_x\ga(t,y)|\,dy,
\]
so that the reparametrization
\[
\ga(t,\cS^{-1}(t,\cdot))
\]
is a constant speed curve.

\medskip

We are finally ready to prove the following theorem, which promotes sub-convergence to full convergence of the flow.

\begin{thm}\label{thm:ConvergenceManifoldsMAIN}
	Suppose that $(M,g)$ is an analytic complete Riemannian manifold endowed with an analytic metric tensor $g$. Let $p\ge2$ and suppose that $\ga:[0,+\infty)\times \S^1\to M$ is a smooth solution of \eqref{eq:DefnFlowPGenerico}. Suppose that there exist a sequence of isometries $I_n:M\to M$, a sequence of times $t_n\nearrow+\infty$, and a smooth critical point $\ga_\infty:\S^1\to M$ of $\sE$ such that
	\[
	I_n \circ \ga(t_n,\cS^{-1}(t_n,\cdot)) -  \ga_\infty (\cdot) \xrightarrow[n\to\infty]{} 0 \qquad \mbox{ in } C^{m}(\S^1),
	\]
	for any $m\in\N$. If $p>2$ assume also that  $|k_{\ga_\infty}(x)|\neq0$ for any $x$.
	
	Then the flow $\ga(t,\cdot)$ converges in $C^m(\S^1)$ to a critical point as $t\to+\infty$, for any $m$ and up to reparametrization.
\end{thm}

\begin{proof}
	In the following the constants may change from line to line and their dependence on universal parameters will be omitted. Let $m\ge 8$ be fixed. Let $\ep\in(0,1)$ that will be fixed later on. By hypothesis there exists $n_\ep\in \N$ such that
	\[
	\|I_{l}\circ \ga(t_{l},\cS^{-1}(t_{l},\cdot)) - \ga_\infty (\cdot)\|_{C^{m}(\S^1)} \le \ep \qquad\forall\,l\ge n_\ep.
	\]
	Let us rename $\ga_\ep(\cdot)\coloneqq I_{l}\circ  \ga(t_{l},\cS^{-1}(t_{l},\cdot))$ for a chosen $l \ge n_\ep$. If $p>2$, for $\ep$ small we can assume that $|k_{\ga_\ep}(x)|\neq0$ for any $x$. By \Cref{thm:LocalExistencePGenerico} there exists a solution $\tilde\ga:[0,T)\times \S^1\to M$ of
	\[
	\begin{cases}
	\pa_t \tilde\ga = -\left( \na_{s}^2 | k_{\tilde\ga}|^{p-2} k_{\tilde\ga} + \frac{1}{p'} |k_{\tilde\ga}|^p k_{\tilde\ga} -k_{\tilde\ga} + R(| k_{\tilde\ga}|^{p-2} k_{\tilde\ga},\tau_{\tilde\ga})\tau_{\tilde\ga} \right) & \mbox{ on } [0,T)\times \S^1,\\
	\tilde \ga(0,\cdot)=\ga_\ep(\cdot),
	\end{cases}
	\]
	for some $T>0$. Since $I_{n_\ep}$ is an isometry, we have that $\tilde \ga (t,x) = I_{n_\ep} \circ \ga(t_{n_\ep}+t,x)$ up to reparametrization. Hence, recalling \Cref{rem:Riparametrizzazioni}, as $\ga$ exists for any time by hypothesis, we get that $T=+\infty$.
	
	We denote by $\hat \ga$ the constant speed reparametrization of $\tilde \ga$.	
	For any $\ep$ sufficiently small, we can write $\hat\ga$ as a variation of $\ga_\infty$, at least for small times. More precisely, there is some $T'\in(0,+\infty]$ such that for any $t\in[0,T')$ there exists $\psi_t\in T(\ga_\infty)^{4,p}$ such that $\hat\ga(t,\cdot)= \exp\psi_t(x)$ and $\|\psi_t\|_{W^{4,p}}<\si$, with $\sigma$ is as in \Cref{cor:FullLojasiewiczPGenerico} applied to $\ga_\infty$; the existence of $\psi_t$ follows as in \Cref{lem:ReparametrizationManifold} as the exponential map restricted to tangent vectors with small norm along $\gamma_\infty$ parametrizes a neighborhood of $\gamma_\infty$.
	We assume that $T'$ is the maximal time such that $\hat\ga$ can be written in such a way with fields $\psi_t$ with $\|\psi_t\|_{W^{4,p}}<\si$.
	
	Suppose by contradiction that $T'<+\infty$.	Up to choosing a smaller $\si=\si(\ga_\infty)$, we can apply \Cref{prop:StimaParabolica} with $\Ga=\ga_\infty$ and $\tau=T'$ on the flow $\tilde \ga$. In the notation on \Cref{prop:StimaParabolica}, we obtain
	\[
	\begin{split}
	 \sup_{[0,T')} \|\hat \ga(t,\cdot) - \ga_\ep\|_{C^{m-3}(\S^1)}
	 &\le \sup_{[0,T')} c \|\hat \ga(t,\cdot) - \ga_\ep(\cdot)\|_{W^{m-2,2}(\S^1)} \\
	 & \le \sup_{[0,T')} c\|\hat \ga(t,\cdot) - \ga_\ep\|_{C^1(\S^1)} + C(\si(\ga_\infty)) \|k_{\hat\ga}(t,\cdot) - k_{\ga_\ep}\|_{W^{m-4,2}(\S^1)}
	 \\
	 & \le \sup_{[0,T')} c\|\hat{\ga}(t,\cdot) - \ga_\infty\|_{C^1(\S^1)} + c\|\ga_\ep - \ga_\infty\|_{C^1(\S^1)} +\\
	 &\qquad\qquad + C(\si(\ga_\infty))( \|k_{\hat\ga}(t,\cdot)\|_{W^{m-4,2}(\S^1)} + \|k_{\ga_\ep}\|_{W^{m-4,2}(\S^1)} )
	 \\
	 &\le C( \si(\ga_\infty)) + c\ep +C(m,\|k_{\ga_\infty}\|_{W^{2,2}}, \si (\ga_\infty) , \cU, \La ) (1 + \|k_{\ga_\ep}\|_{W^{m-2,2}})
	 \\
	 &\le
	 C(m,\ga_\infty, \cU, \La ) (1 + \|k_{\ga_\ep}\|_{W^{m-2,2}}) \\
	 &\le C(m,\ga_\infty, \cU, \La) ( 1 + \ep + \|k_{\ga_\infty}\|_{C^{m-2}(\S^1)} ) \\
	 &\le C(m,\ga_\infty,  \cU, \La).
	\end{split}
	\]
	Also, $\cU$ and $\La$ only depend on $\ga_\infty$ and $\si=\si(\ga_\infty)$, then the above estimate becomes
	\[
	\sup_{[0,T')} \|\hat \ga(t,\cdot) - \ga_\ep(\cdot)\|_{C^{m-3}(\S^1)} \le C(m,\ga_\infty),
	\]
	and we observe that the constant on the right hand side is independent of $\ep$.
	By triangular inequality we also have
	\begin{equation}\label{eq:mStima}
	\sup_{[0,T')} \|\hat \ga(t,\cdot) - \ga_\infty(\cdot)\|_{C^{m-3}(\S^1)} \le  C(m,\ga_\infty).
	\end{equation}
	Now define
	\[
	H(t)= (\sE(\hat\ga(t,\cdot)) - \sE(\ga_\infty) )^\te,
	\]
	where $\te$ is the \L ojasiewicz-Simon exponent of \Cref{cor:FullLojasiewiczPGenerico}. Observe that since $\sE(\ga_0)\ge \sE(\ga_\infty)$, by uniqueness of the flow we also have that $\sE(\hat\ga(t,\cdot))\ge \sE(\ga_\infty)$ for any $t$, and we can also assume that $0<\sE(\hat\ga(t,\cdot)) - \sE(\ga_\infty)<1$ without loss of generality. In particular $H$ is well defined and positive.
	
	By \Cref{rem:Riparametrizzazioni} we have that $\pa^\perp_t \hat\ga = -|\pa_x\hat\ga|^{-1}\nabla_{T(\hat\ga)^{p'},T(\hat\ga)^p}E(0)$. Using \Cref{cor:FullLojasiewiczPGenerico} get that
	\begin{equation}\label{eq:DifferentialInequalityPGenerico'}
	\begin{split}
	-\frac{d}{dt}H(t) &= - \te H^{\frac{\te-1}{\te}}(t) \scal{ \nabla_{T(\hat\ga)^{p'},T(\hat\ga)^p}E(0) , \pa^\perp_t\hat\ga   }_{L^{p'}(dx),L^{p}(dx)}\\
	&= \te H^{\frac{\te-1}{\te}}(t) \|\pa^\perp_t\hat\ga\|^2_{L^2(ds_{\hat\ga})} \\
	&\ge C(L(\hat\ga)) \te H^{\frac{\te-1}{\te}}(t) \|\pa^\perp_t\hat\ga\|_{L^{p'}(ds_{\hat\ga})}  \|\pa^\perp_t\hat\ga\|_{L^2(ds_{\hat\ga})}\\
	&\ge C(L(\hat\ga),\ga_\infty)\te H^{\frac{\te-1}{\te}}(t) \|\nabla_{T( \ga_\infty)^{p'},T(\ga_\infty)^p} \bE (\psi_t)\|_{L^{p'}(ds_{\hat\ga})} \|\pa^\perp_t\hat\ga\|_{L^{2}(ds_{\hat\ga})}  \\
	& \ge C(L(\hat\ga),\ga_\infty) \te H^{\frac{\te-1}{\te}}(t) |\bE (\psi_t) -\bE(0)|^{1-\te} \|\pa^\perp_t\hat\ga\|_{L^{2}(ds_{\hat\ga})} \\
	&= C(L(\hat\ga),\ga_\infty) \te H^{\frac{\te-1}{\te}}(t) |\sE(\hat\ga(t,\cdot))-\sE(\ga_\infty)|^{1-\te} \|\pa^\perp_t\hat\ga\|_{L^{2}(ds_{\hat\ga})} \\
	&= C(L(\hat\ga),\ga_\infty) \te\|\pa^\perp_t\hat\ga\|_{L^{2}(ds_{\hat\ga})} ,
	\end{split}
	\end{equation}
	for $t\in[0,T')$.
	
	Now let us write more explicitly $\hat\ga$ as $\hat \ga (t,x)= \tilde \ga(t,\chi(t,x))$, where $\chi(t,\cdot)$ is the inverse of
	\[
	x\mapsto \frac{2\pi}{L(\tilde \ga(t,\cdot))}\int_0^xds_{\tilde \ga},
	\]
	and $|\pa_x\tilde\ga|(0,x)=|\pa_x\hat\ga|(0,x)=\frac{L(\ga_\ep)}{2\pi}$, as the initial datum $\ga_\ep$ is parametrized with constant speed.
	
	Using \eqref{eq:DifferentialInequalityPGenerico'} we have that
	\[
	\begin{split}
		\left| \int_0^x ds_{\tilde\ga} - \frac{L(\ga_\ep)}{2\pi}x \right|
		&= \left| \int_0^t \pa_t \int_0^x  ds_{\tilde\ga} \,dt \right| 
		\le \int_0^t \left| \int_0^x \scal{k_{\tilde \ga}, \pa_t\tilde{ \ga}}\,ds_{\tilde\ga} \right|\,dt
		\le  C(\ga_\infty) \int_0^t \|\pa_t\tilde\ga\|_{L^2(ds_{\tilde\ga})} \, dt \\
		&=  C(\ga_\infty) \int_0^t \|\pa^\perp_t\hat\ga\|_{L^2(ds_{\hat\ga})} \, dt \le C(\ga_\infty) H(0) \le C(\ga_\infty) \|\ga_\ep-\ga_\infty\|_{C^2}^\theta\le C(\ga_\infty)\ep^\te,
	\end{split}
	\]
	for any $t<T'$ and any $x\in \S^1$. It follows that $\Big| \tfrac{2\pi}{L(\tilde \ga)}\int_0^xds_{\tilde\ga}-x \Big|\le C(\ga_\infty)\ep^\te$ as well. If $\ep$ is sufficiently small, we deduce that
	\begin{equation}\label{eq:BoundParametrizzazione}
		|\chi(t,x)-x|\le C(\ga_\infty)\ep^\te,
	\end{equation}
	for any $t<T'$ and any $x\in \S^1$.
	
	Therefore, writing $\tilde\ga(t,y)=\hat\ga(t,\vp(t,y))$ and letting $\tilde\ga_\infty\coloneqq\ga_\infty(t,\vp(t,y))$, we have
	\[
	\begin{split}
		\|\hat\ga(t,\cdot)- \ga_\infty\|^2_{L^2(dx)}  
		&\le C(\ga_\infty) \|\hat\ga(t,\cdot)- \ga_\infty\|^2_{L^2(ds_{\hat\ga})} = \int |\tilde\ga(t,y)-\ga_\infty(\vp(t,y))|^2\,ds_{\tilde\ga}\\
		&\le C(\ga_\infty) \int |\tilde{\ga}(t,y)-\tilde{ \ga}(0,y)|^2\,ds_{\tilde{ \ga}} 
		+ |\tilde{ \ga}(0,y)-\tilde\ga_\infty(t,y)|^2 \,ds_{\tilde{ \ga}} .
	\end{split}
	\]
	We estimate the two terms above as
	\[
	\begin{split}
		\left(\int |\tilde{\ga}(t,y)-\tilde{ \ga}(0,y)|^2\,ds_{\tilde{ \ga}} \right)^{\frac12} 
		&\le \int_0^t \| \pa_t \tilde{ \ga}\|_{L^2(ds_{\tilde\ga})} \, dt = \int_0^t \|\pa_t^\perp \hat \ga \|_{L^2(ds_{\hat\ga})} \, dt 
		\le C(\ga_\infty)H(0) \le C(\ga_\infty)\ep^\te,
	\end{split}
	\]
	and
	\[
	\begin{split}
		\int |\tilde{ \ga}(0,y)-\tilde\ga_\infty(t,y)|^2 \,ds_{\tilde{ \ga}} 
		&= \int |\ga_\ep(y)-\tilde\ga_\infty(t,y)|^2 \,ds_{\tilde{ \ga}} 
		= \int |\ga_\ep(\chi(t,x)) - \ga_\infty (x) |^2 \, ds_{\hat\ga}\\
		& \le 2 \int |\ga_\ep(\chi(t,x)) - \ga_\ep (x) |^2 
		+ |\ga_\ep(x)- \ga_\infty (x) |^2 \, ds_{\hat\ga} \\
		&\le C(\ga_\infty) \left(
		(\Lip(\ga_\ep))^2 \int_{\S^1} |\chi(t,x)-x|^2\,dx + \|\ga_\ep-\ga_\infty\|_{\infty}^2
		\right)	\\
		&\le C(\ga_\infty)\ep^{2\te},
	\end{split}
	\]
	where we used \eqref{eq:BoundParametrizzazione}. It follows that
	\begin{equation}\label{eq:Stima}
	\|\hat\ga(t,\cdot)- \ga_\infty\|_{L^2(dx)} \le C(\ga_\infty)\ep^\te,	
	\end{equation}
	for any $t<T'$.
	%
	%
	%
	%
	By \eqref{eq:Interpolation} and \eqref{eq:Interpolation3}, using that for $s\in(\tfrac12,1)$ we have the continuous embeddings $W^{s,2}\hookrightarrow C^{0,\alpha'}\hookrightarrow L^p$, for some $\alpha\in(0,1)$ we can write
	\begin{equation*}
	\begin{split}	
	\|\hat\ga(t,\cdot)-\ga_\infty\|_{W^{4,p}}
	&\le C  \|\hat\ga(t,\cdot)- \ga_\infty\|_{C^5}^{\alpha} \|\hat\ga(t,\cdot)- \ga_\infty\|_{L^p(dx)}^{1-\alpha} \\
	&\le C  \|\hat\ga(t,\cdot)- \ga_\infty\|_{C^5}^{\alpha}
	\|\hat\ga(t,\cdot)- \ga_\infty\|_{W^{s,2}(dx)}^{1-\alpha} \\
	&\le C  \|\hat\ga(t,\cdot)- \ga_\infty\|_{C^5}^{\alpha}
	\left(
	C\|\hat\ga(t,\cdot)- \ga_\infty\|_{L^2(dx)}^{1-s} \|\hat\ga(t,\cdot)- \ga_\infty\|_{W^{1,2}(dx)}^s	
	\right)^{1-\alpha}\\
	& \le  C(\ga_\infty)  \ep^{\te(1-s)(1-\alpha)}
	\end{split}
	\end{equation*}
	for $t\in[0,T')$, where in the last inequality we used 	\eqref{eq:mStima} and \eqref{eq:Stima}. This means that
	\[
	\|\exp(\psi_t)- \ga_\infty\|_{W^{4,p}} \le C(\ga_\infty)  \ep^{\te(1-s)(1-\alpha)},
	\]
	for $t\in[0,T')$, and if $\ep$ is sufficiently small this implies that $\|\psi_t\|_{W^{4,p}}\le \tfrac12\si$ for any $t\in[0,T')$, contradicting the maximality of $T'$.
	
	\medskip
	
	Hence we proved that for some now fixed $\ep_0>0$, for any $l \ge n_{\ep_0}$ the constant speed parametrized flow $\hat\ga_l$ starting at $\ga_l\coloneqq I_l\circ \ga(t_l,\cS^{-1}(t_l,\cdot)))$ exists for any time and it can be written as a variation of $\ga_\infty$ with uniformly bounded fields. In particular the evolution $\hat\ga_l$ stays in the compact set $\cU$ for any $t$ for any $l\ge n_{\ep_0}$. Similarly, it follows that the original flow $\ga$ definitely remains in $I_{n_{\ep_0}}^{-1}(\cU)\eqqcolon\cV$.
	
	For $l\ge n_{\ep_0}$ consider the sequence of isometries $I_l:\cV\to\cU$. By Ascoli-Arzelà Theorem, up to subsequence we have that $I_l$ uniformly converges to a map $I:\cV\to \cU$, which is still an isometry, and thus it is smooth by Myers-Steenrod Theorem. Moreover, observe that by compactness there exists a constant $\De>0$ such that
	\[
	|x-y| \le d_{(M,g)}(x,y) \le \De |x-y|,
	\]
	for any couple $x,y\in \cU$ or $x,y\in \cV$, where $d_{(M,g)}$ is the geodesic distance on $M$.
	
	Let us call $\hat\ga_0\coloneqq \hat\ga_{n_{\ep_0}}$.	
	By uniqueness and \Cref{rem:Riparametrizzazioni} we have that
	\[
	\hat \ga_0 = I_{n_{\ep_0}} \circ I_l^{-1} \circ \hat \ga_l,
	\]
	up to a translation in time depending on $l$, for any $l>n_{\ep_0}$.
	
	We know that for any $\ep\in(0,\ep_0)$ and any $l\ge n_\ep$ the flow $\hat\ga_l$ verifies that
	\[
	\|\hat\ga_l(t,\cdot) - \ga_\infty \|_{L^2(dx)} \le C(\ga_\infty)\ep^\te,
	\]
	for any $t>0$. Indeed, this follows from \eqref{eq:Stima} applied on the flow $\hat\ga_l$.
	Therefore we have
	\[
	\| \hat\ga_0(t+T_l,\cdot) - I_{n_{\ep_0}} \circ I_l^{-1} \circ \ga_\infty \|_{L^2(dx)} 
	= \| I_{n_{\ep_0}} \circ I_l^{-1} \circ \hat \ga_l(t,\cdot)  - I_{n_{\ep_0}} \circ I_l^{-1} \circ \ga_\infty \|_{L^2(dx)}
	\le \bar{C} \ep^\te, 
	\]
	for any $\ep$, $l\ge n_\ep$, and any $t>0$, where $\bar{C}=\bar{C}(\ga_\infty,\De)$ and $T_l$ depends on $l$. Also, we have that
	\[
	I_{n_{\ep_0}} \circ I_l^{-1} \circ \ga_\infty \xrightarrow[l]{}
	I_{n_{\ep_0}} \circ I^{-1} \circ \ga_\infty,
	\]
	uniformly on $\S^1$, and then in $L^2(dx)$. Hence, finally, for any given $\bar\ep$ we can set $\ep=(\bar{C}\bar\ep)^{\frac1\te}$ and take $l\ge n_\ep$ such that $\|I_{n_{\ep_0}} \circ I_l^{-1} \circ \ga_\infty 
	- I_{n_{\ep_0}} \circ I^{-1} \circ \ga_\infty \|_{L^2(dx)} \le \bar \ep$, and we obtain
	\[
	\begin{split}
		\| \hat\ga_0 (t,\cdot) &- I_{n_{\ep_0}} \circ I^{-1} \circ \ga_\infty \|_{L^2(dx)}  \le\\
		& \le \| \hat\ga_0 (t,\cdot) - I_{n_{\ep_0}} \circ I_l^{-1} \circ \ga_\infty  \|_{L^2(dx)} 
		+
		\|I_{n_{\ep_0}} \circ I_l^{-1} \circ \ga_\infty 
		 - I_{n_{\ep_0}} \circ I^{-1} \circ \ga_\infty \|_{L^2(dx)} \\
		 & \le 2\bar \ep ,
	\end{split}
	\]
	for any $t> T_l$. This implies that
	\[
	\exists\, \lim_{t\to+\infty} \hat\ga_0(t,\cdot) = I_{n_{\ep_0}} \circ I^{-1} \circ \ga_\infty \eqqcolon \bar \ga \qquad\quad \mbox{ in } L^2(dx).
	\]
	Now we can use \eqref{eq:Interpolation} to get that
	\[
	\|\hat\ga_0(t,\cdot)-\bar\ga\|_{W^{m-4,2}}
	\le C \|\hat\ga_0(t,\cdot)-\bar\ga\|_{C^{m-3}}^{\alpha} \|\hat\ga_0(t,\cdot)-\bar\ga\|_{L^2(dx)}^{1-\alpha},
	\]
	and using \eqref{eq:mStima} we see that $\hat\ga_0 \to \bar \ga$ in $W^{m-4,2}$. Taking higher $m$ and interpolating using \eqref{eq:Interpolation2}, one gets that $\hat\ga_0$ converge smoothly.
	Finally, since the original flow $\ga$ is a reparametrization of $I_{n_{\ep_0}}^{-1}\circ \hat\ga_0$, it smoothly converges as desired, up to reparametrization.	
\end{proof}

\bigskip

\appendix

\section{Proof of \Cref{prop:SecondVariation}}

Here we prove the general formula of the second variation of the $p$-elastic energy.
We remark that we need to keep track of every term in the final formulas \eqref{eq:SecVarTotale} and \eqref{eq:SecondVarTotale} as we need to recall the precise expression of the second variation a couple of times.

\begin{proof}[Proof of \Cref{prop:SecondVariation}]
	Denoting by $\ga_\ep(\cdot)=\Phi(\ep,\cdot)$, by $\Phi$ the variation of $\ga$ with variation field $\vp$, and by $\cT:T_{\ga(x)}M\to T_{\ga_\ep(x)}M$ the map $\cT\psi=d[\exp_{\ga(x)}]_{\ep\vp}(\psi)$, we have that
	\[
	\cL(\vp)[\psi]=  \frac{d}{d\ep}\bigg|_0 \int \left\lgl 
	|k_{\ga_\ep}|^{p-2}k_{\ga_\ep}
	, \na_{{\ga_\ep}}^2( \ga_\ep^\perp (\cT\psi)) \right\rgl
	+ \left\lgl
	\frac{1}{p'}|k_{\ga_\ep}|^pk_{\ga_\ep} - k_{\ga_\ep} +R(|k_{\ga_\ep}|^{p-2}k_{\ga_\ep},\tau_{\ga_\ep})\tau_{\ga_\ep} 
	, \cT\psi 
	\right\rgl \, ds_{\ga_\ep}.
	\]
	Now we calculate term by term the above identity.
	Using \eqref{eq:V2} twice we have that
	\[
	\begin{split}
	(M^\top-\ga^\top)\frac{d}{d\ep}\bigg|_0 &\left( \na_{{\ga_\ep}}^2( \ga_\ep^\perp (\cT\psi))  \right)
	= \na \left( (M^\top-\ga^\top)\frac{d}{d\ep}\bigg|_0 \na_{{\ga_\ep}} (\ga_\ep^\perp (\cT\psi)) \right) +\\
	&\indent -\scal{\na\psi,\na\vp}k + \scal{k,\vp}\na^2\psi + \scal{\na\psi,k}\na\vp +\ga^\perp R(\vp,\tau)\na\psi\\
	&= \na\left( \na (M^\top-\ga^\top)\frac{d}{d\ep}\bigg|_0 (\ga_\ep^\perp (\cT\psi)) \right) +\\
	&\indent + \na\left[ -\scal{\psi,\na\vp}k + \scal{k,\vp}\na\psi + \scal{\psi,k}\na\vp +\ga^\perp R(\vp,\tau)\psi  \right] +\\
	&\indent -\scal{\na\psi,\na\vp}k + \scal{k,\vp}\na^2\psi + \scal{\na\psi,k}\na\vp +\ga^\perp R(\vp,\tau)\na\psi.
	\end{split}
	\]
	We can compute
	\begin{equation}\label{eq:J1}
	\begin{split}
	(M^\top-\ga^\top)\frac{d}{d\ep}\bigg|_0 (\ga_\ep^\perp (\cT\psi))
	&=  (M^\top-\ga^\top)\frac{d}{d\ep}\bigg|_0 \left( d[\exp_{\ga(x)}]_{\ep\vp}(\psi) - \scal{d[\exp_{\ga(x)}]_{\ep\vp}(\psi),\tau_{\ga_\ep}}\tau_{\ga_\ep }  \right)\\
	&=  (M^\top-\ga^\top)\frac{d}{d\ep}\bigg|_0  d[\exp_{\ga(x)}]_{\ep\vp}(\psi),
	\end{split}
	\end{equation}
	where we used that $\scal{d[\exp_{\ga(x)}]_{\ep\vp}(\psi),\tau_{\ga_\ep}}\big|_0= \scal{d[\exp_{\ga(x)}]_{0}(\psi),\tau_{\ga}}=\scal{\psi,\tau}=0$. Now the field
	\[
	\ep \mapsto J(\ep)= d[\exp_{\ga(x)}]_{\ep\vp}(\ep\psi)
	\]
	is a Jacobi field along the geodesic $\si_\vp$ such that $\si_\vp(0)=\ga(x)$, $\si_\vp'(0)=\vp$, $J(0)=0$, and $J'(0)=\psi$ (see \cite[Chapter 5, Corollary 2.5]{DoCarmo}). Then
	\begin{equation}\label{eq:J2}
	\frac{d}{d \ep}\bigg|_0
	d[\exp_{\ga(x)}]_{\ep\vp}(\psi) =
	\frac{d}{d \ep}\bigg|_0\left(
	\frac1\ep d[\exp_{\ga(x)}]_{\ep\vp}(\ep\psi)
	\right) 
	= \frac{d}{d \ep}\bigg|_0\left(
	\frac1\ep J(\ep)
	\right).
	\end{equation}
	We claim that $\tfrac{d}{d \ep}\bigg|_0\left(\frac1\ep J(\ep)\right)=0$ for any $x$. In fact let $\{E_i(\cdot)\}$ be an orthonormal parallel frame along $\si_\vp$, and write $J(\ep)=J^i(\ep)E_i(\ep)$. The Jacobi Equation (see \cite[Chapter 5, Definition 2.1]{DoCarmo}) for $J$ then reads
	\[
	(J^i)''(\ep) E_i(\ep) + R(J^i(\ep)E_i(\ep),\si_\vp'(\ep))\si_\vp'(\ep)=0.
	\]
	Therefore since $J^i(\cdot)$ is of class $C^2$ with $J^i(0)=0$ we conclude that
	\begin{equation}\label{eq:J3}
	\frac{d}{d \ep}\bigg|_0\left(\frac1\ep J(\ep)\right) =  \frac{d}{d \ep}\bigg|_0\left( \frac{J^i(\ep)}{\ep} \right) E_i(0) = \frac12 (J^i)''(0) E_i(0) = - R(J(0),\vp(x))\vp(x) =0,
	\end{equation}
	and thus
	\begin{equation}\label{eq:Jacobi0}
	(M^\top-\ga^\top)\frac{d}{d\ep}\bigg|_0 (\ga_\ep^\perp (\cT\psi))=0.
	\end{equation}
	
	Eventually we deduce that
	\begin{equation}\label{eq:SecondVar1}
	\begin{split}
	(M^\top-\ga^\top)\frac{d}{d\ep}\bigg|_0 &\left( \na_{{\ga_\ep}}^2( \ga_\ep^\perp (\cT\psi))  \right)
	= \\
	&= \na\left[ -\scal{\psi,\na\vp}k + \scal{k,\vp}\na\psi + \scal{\psi,k}\na\vp +\ga^\perp R(\vp,\tau)\psi  \right] +\\
	&\indent -\scal{\na\psi,\na\vp}k + \scal{k,\vp}\na^2\psi + \scal{\na\psi,k}\na\vp +\ga^\perp R(\vp,\tau)\na\psi.
	\end{split}
	\end{equation}
	
	On the other hand we have that
	\begin{equation}\label{eq:SecondVar2}
	\begin{split}
	(M^\top-\ga^\top)&\frac{d}{d\ep}\bigg|_0 \left(|k_{\ga_\ep}|^{p-2}k_{\ga_\ep}\right)
	=\\
	&= (M^\top-\ga^\top) \bigg( (p-2)|k|^{p-4}\scal{k
		, \na^2\vp + \scal{\vp,k}k +\\
		&\indent + R(\vp,\tau)\tau} k 
	+ |k|^{p-2}  ( \na^2\vp + \scal{\vp,k}k + R(\vp,\tau)\tau )  \bigg) \\
	&=  |k|^{p-2}\na^2\vp + |k|^{p-2}\scal{\vp,k}k +\\
	&\indent + |k|^{p-2}\ga^\perp R(\vp,\tau)\tau  
	+(p-2)\Big( |k|^{p-4}\scal{k,\na^2\vp}k + |k|^{p-2}\scal{\vp,k} k +\\
	&\indent + |k|^{p-4}\scal{k, R(\vp,\tau)\tau}k   \Big) \\
	&=  |k|^{p-2}\na^2\vp 
	+(p-2) |k|^{p-4}\scal{k,\na^2\vp}k +\\
	&\indent + |k|^{p-2} R(\vp,\tau)\tau 
	+ (p-2)|k|^{p-4}\scal{k, R(\vp,\tau)\tau}k +\\
	&\indent +(p-1) |k|^{p-2}\scal{\vp,k} k,
	\end{split}
	\end{equation}
	where we used that $R(\vp,\tau)\tau \in T\ga^\perp$.
	Similarly
	\begin{equation}\label{eq:SecondVar3}
	\begin{split}
	(M^\top-\ga^\top)\frac{d}{d\ep}\bigg|_0 \left(|k_{\ga_\ep}|^{p}k_{\ga_\ep}\right) 
	&= p|k|^{p-2}\scal{\na^2\vp,k}k + p|k|^p\scal{\vp,k}k +\\
	&\indent +p|k|^{p-2} \scal{R(\vp,\tau)\tau,k} k +\\
	&\indent + |k|^p\left( \na^2\vp + \scal{\vp,k}k + \ga^\perp R(\vp,\tau)\tau  \right) \\
	&= |k|^p \na^2\vp  +  p|k|^{p-2}\scal{\na^2\vp,k}k +\\
	&\indent +p|k|^{p-2} \scal{R(\vp,\tau)\tau,k} k 
	+ |k|^p R(\vp,\tau)\tau +\\
	&\indent +(p+1)|k|^p\scal{\vp,k}k.
	\end{split}
	\end{equation}
	
	Also, we already know that
	\begin{equation}\label{eq:SecondVar4}
	(M^\top-\ga^\top)\frac{d}{d\ep}\bigg|_0 (-k_{\ga_\ep}) = - \left( \na^2\vp + \scal{\vp,k}k + R(\vp,\tau)\tau  \right).
	\end{equation}
	
	Finally, since
	\[
	\scal{R(|k_{\ga_\ep}|^{p-2}k_{\ga_\ep},\tau_{\ga_\ep })\tau_{\ga_\ep },\ga_\ep^\perp \cT\psi } =\scal{ R(\ga_\ep^\perp \cT\psi, \tau_{\ga_\ep}  )\tau_{\ga_\ep} , |k_{\ga_\ep}|^{p-2}k_{\ga_\ep} },
	\]
	calling $\cR$ the tensor $\cR(X,Y,Z)=R(X,Y)Z$ we have that
	\begin{equation*}
	\begin{split}
	\frac{d}{d\ep}\bigg|_0  \scal{R(|k_{\ga_\ep}|^{p-2}k_{\ga_\ep},\tau_{\ga_\ep })\tau_{\ga_\ep }&,\ga_\ep^\perp \cT\psi } 
	=\\
	&= \scal{ (D_\vp \cR)(\psi,\tau,\tau) + \cR(\pa_\ep|_0 (\ga_\ep^\perp\cT\psi),\tau,\tau) +\\
		&\indent+\cR(\psi,\na\vp,\tau) 
		+ \cR(\psi,\tau,\na\vp), |k|^{p-2}k } +\\
	&\indent +\scal{R(\psi,\tau)\tau, (M^\top-\ga^\top) \pa_\ep|_0 (|k_{\ga_\ep}|^{p-2}k_{\ga_\ep})}.
	\end{split}
	\end{equation*}
	
	Using \eqref{eq:J1}, \eqref{eq:J2}, and \eqref{eq:J3} we see that
	\[
	\frac{d}{d\ep}\bigg|_0 \ga_\ep^\perp \cT\psi= -\scal{\psi,\na\vp}\tau.
	\]
	
	Therefore
	\begin{equation}\label{eq:SecondVar5}
	\begin{split}
	\frac{d}{d\ep}\bigg|_0  \scal{R(|k_{\ga_\ep}|^{p-2}k_{\ga_\ep},\tau_{\ga_\ep })\tau_{\ga_\ep },\ga_\ep^\perp \cT\psi } 
	&= \scal{  (D_\vp \cR)(\psi,\tau,\tau) + R(\psi,\na\vp)\tau 
		+ R(\psi,\tau)\na\vp, |k|^{p-2}k  } +\\
	&\indent +\big\lgl
	R(\psi,\tau)\tau \,,\, 
	|k|^{p-2}\na^2\vp  +(p-2) |k|^{p-4}\scal{k,\na^2\vp}k +\\
	&\indent + |k|^{p-2} R(\vp,\tau)\tau 
	+ (p-2)|k|^{p-4}\scal{k, R(\vp,\tau)\tau}k +\\
	&\indent +(p-1) |k|^{p-2}\scal{\vp,k} k
	\big\rgl .
	\end{split}
	\end{equation}

	Putting together \eqref{eq:SecondVar1}, \eqref{eq:SecondVar2}, \eqref{eq:SecondVar3}, \eqref{eq:SecondVar4}, and \eqref{eq:SecondVar5} we conclude that
	\[
	\cL(\vp)[\psi] = I_1+I_2+I_3+I_4+I_5,
	\]
	where
	\[
	\begin{split}
	I_1 &= \int \bigg\lgl |k|^{p-2}k, 
	\na\left[ -\scal{\psi,\na\vp}k + \scal{k,\vp}\na\psi + \scal{\psi,k}\na\vp +\ga^\perp R(\vp,\tau)\psi  \right] +\\
	&\indent -\scal{\na\psi,\na\vp}k + \scal{k,\vp}\na^2\psi +\\
	&\indent + \scal{\na\psi,k}\na\vp +\ga^\perp R(\vp,\tau)\na\psi
	\bigg\rgl\,ds_\ga,
	\end{split}
	\]
	\[
	\begin{split}
	I_2 &= \int 
	\bigg\lgl 
	|k|^{p-2}\na^2\vp 
	+(p-2) |k|^{p-4}\scal{k,\na^2\vp}k +\\
	&\indent + |k|^{p-2} R(\vp,\tau)\tau 
	+ (p-2)|k|^{p-4}\scal{k, R(\vp,\tau)\tau}k +\\
	&\indent +(p-1) |k|^{p-2}\scal{\vp,k} k
	, \na^2\psi \bigg\rgl \,ds_\ga,
	\end{split}
	\]
	\[
	\begin{split}
	I_3 &= \int \bigg\lgl
	|k|^p \na^2\vp  +  p|k|^{p-2}\scal{\na^2\vp,k}k +\\
	&\indent +p|k|^{p-2} \scal{R(\vp,\tau)\tau,k} k 
	+ |k|^p R(\vp,\tau)\tau +\\
	&\indent +(p+1)|k|^p\scal{\vp,k}k
	- \left( \na^2\vp + \scal{\vp,k}k + R(\vp,\tau)\tau  \right)
	,\psi
	\bigg\rgl\,ds_\ga,
	\end{split}
	\]
	\[
	\begin{split}
	I_4 &= \int  \left\lgl   (D_\vp \cR)(\psi,\tau,\tau) + R(\psi,\na\vp)\tau + R(\psi,\tau)\na\vp, |k|^{p-2}k  \right\rgl +\\
	&\indent +\bigg\lgl
	R(\psi,\tau)\tau, 
	|k|^{p-2}\na^2\vp 
	+(p-2) |k|^{p-4}\scal{k,\na^2\vp}k +\\
	&\indent + |k|^{p-2} R(\vp,\tau)\tau 
	+ (p-2)|k|^{p-4}\scal{k, R(\vp,\tau)\tau}k +\\
	&\indent +(p-1) |k|^{p-2}\scal{\vp,k} k
	\bigg\rgl \,ds_\ga ,
	\end{split}
	\]
	\[
	\begin{split}
	I_5 &= - \int \left(\left\lgl |k|^{p-2}k, \na^2 \psi \right\rgl 
	+ \left\lgl \frac{1}{p'}|k|^p k -k + R(|k|^{p-2}k,\tau)\tau , \psi   \right\rgl\right) \scal{k,\vp}\,ds_\ga.
	\end{split}
	\]
	Integrating by parts and rearranging the terms we end up with
	\begin{equation}\label{eq:SecVarTotale}
	\begin{split}
	\cL(\vp)[\psi]
	&=
	\int 
	\scal{k,\vp}\scal{|k|^{p-2}k,\na^2\psi}
	+ \bigg\lgl
	|k|^{p-2}\na^2\vp
	+(p-2) |k|^{p-4}\scal{k,\na^2\vp}k +\\
	&\indent + |k|^{p-2} R(\vp,\tau)\tau 
	+ (p-2)|k|^{p-4}\scal{k, R(\vp,\tau)\tau}k
	,\na^2\psi
	\bigg\rgl\,ds_\ga +\\
	&+\int \big\lgl 
	-\na|k|^{p-2}k,
	\scal{k,\vp}\na\psi 
	\big\rgl
	-(p-1)\big\lgl
	\na(\scal{\vp,k}|k|^{p-2}k)
	,\na\psi
	\big\rgl
	\,ds_\ga +\\
	&+\int \big\lgl
	-\na|k|^{p-2}k,
	R(\vp,\tau)\psi
	\big\rgl 
	+ \scal{ |k|^{p-2}k,R(\vp,\tau)\na\psi }
	+ \\
	&\indent +  \left\lgl   (D_\vp \cR)(\psi,\tau,\tau) + R(\psi,\na\vp)\tau + R(\psi,\tau)\na\vp, |k|^{p-2}k  \right\rgl 
	+\bigg\lgl
	R(\psi,\tau)\tau, 
	|k|^{p-2}\na^2\vp +\\
	&\indent +(p-2) |k|^{p-4}\scal{k,\na^2\vp}k
	+ |k|^{p-2} R(\vp,\tau)\tau 
	+ (p-2)|k|^{p-4}\scal{k, R(\vp,\tau)\tau}k +\\
	&\indent +(p-1) |k|^{p-2}\scal{\vp,k} k
	\bigg\rgl
	\,ds_\ga +\\
	&+\int \bigg\lgl
	-\na|k|^{p-2}k, 
	-\scal{\psi,\na\vp}k 
	+\scal{\psi,k}\na\vp
	\bigg\rgl
	+ \bigg\lgl \na(|k|^p\na\vp) +\\
	&\indent -\na(\scal{|k|^{p-2}k,\na\vp}k)
	+
	|k|^p \na^2\vp  +  p|k|^{p-2}\scal{\na^2\vp,k}k +\\
	&\indent +p|k|^{p-2} \scal{R(\vp,\tau)\tau,k} k 
	+ |k|^p R(\vp,\tau)\tau 
	+(p+1)|k|^p\scal{\vp,k}k +\\
	&\indent - \left( \na^2\vp + \scal{\vp,k}k + R(\vp,\tau)\tau  \right)
	,\psi
	\bigg\rgl
	\,ds_\ga +\\
	&-  \int \left(\left\lgl |k|^{p-2}k, \na^2 \psi \right\rgl 
	+ \left\lgl \frac{1}{p'}|k|^p k -k + R(|k|^{p-2}k,\tau)\tau , \psi   \right\rgl\right) \scal{k,\vp}\,ds_\ga,
	\end{split}
	\end{equation}
	that is
	\begin{equation}\label{eq:SecondVarTotale}
	\begin{split}
	\cL(\vp)[\psi]
	&=
	\int 
	\scal{k,\vp}\scal{|k|^{p-2}k,\na^2\psi}
	+ \bigg\lgl
	|k|^{p-2}\na^2\vp
	+(p-2) |k|^{p-4}\scal{k,\na^2\vp}k +\\
	&\indent + |k|^{p-2} R(\vp,\tau)\tau 
	+ (p-2)|k|^{p-4}\scal{k, R(\vp,\tau)\tau}k
	,\na^2\psi
	\bigg\rgl\,ds_\ga +\\
	&+\int \big\lgl 
	-\na|k|^{p-2}k,
	\scal{k,\vp}\na\psi 
	\big\rgl
	-(p-1)\big\lgl
	\na(\scal{\vp,k}|k|^{p-2}k)
	,\na\psi
	\big\rgl
	\,ds_\ga +\\
	&+\int \big\lgl
	-\na|k|^{p-2}k,
	R(\vp,\tau)\psi
	\big\rgl 
	+ \scal{ |k|^{p-2}k,R(\vp,\tau)\na\psi }
	+ \\
	&\indent +  \left\lgl   (D_\vp \cR)(\psi,\tau,\tau) + R(\psi,\na\vp)\tau + R(\psi,\tau)\na\vp, |k|^{p-2}k  \right\rgl 
	+\bigg\lgl
	R(\psi,\tau)\tau, 
	|k|^{p-2}\na^2\vp +\\
	&\indent +(p-2) |k|^{p-4}\scal{k,\na^2\vp}k
	+ |k|^{p-2} R(\vp,\tau)\tau 
	+ (p-2)|k|^{p-4}\scal{k, R(\vp,\tau)\tau}k +\\
	&\indent +(p-1) |k|^{p-2}\scal{\vp,k} k
	\bigg\rgl
	\,ds_\ga +\\
	&+\int \scal{\Om(\vp),\psi}\,ds_\ga +\\
	&-  \int \left(\left\lgl |k|^{p-2}k, \na^2 \psi \right\rgl 
	+ \left\lgl \frac{1}{p'}|k|^p k -k + R(|k|^{p-2}k,\tau)\tau , \psi   \right\rgl\right) \scal{k,\vp}\,ds_\ga,
	\end{split}
	\end{equation}
	where $\Om:T(\ga)^{4,p,\perp}\to T(\ga)^{p,\perp}$ is a compact operator, indeed, comparing \eqref{eq:SecVarTotale} and \eqref{eq:SecondVarTotale}, one sees that $\Omega(\varphi)$ contains at most second order derivatives of $\varphi \in T(\ga)^{4,p,\perp}$. Hence by Sobolev embeddings, $\Omega$ is compact as a functional from $T(\ga)^{4,p,\perp}$ to $T(\ga)^{p,\perp}$. Therefore the statement follows.
\end{proof}

\section{Proof of \Cref{prop:StimaParabolica}}\label{app:StimaParabolica}

Throughout this section we assume the hypotheses of \Cref{prop:StimaParabolica}. More precisely we consider $\ga:[0,\tau)\times \S^1 \to M$ to be a smooth solution of \eqref{eq:DefnFlowPGenerico}, and $\Ga:\S^1\to M$ is a fixed smooth curve parametrized with constant speed. We let $\hat \ga$ be the constant speed reparametrization of $\ga$. We assume that $\bar{\si}>0$ is such that
\[
\|\hat \ga(t,\cdot) - \Ga \|_{W^{4,p}} \le \bar{\si},
\]
for any $t\in[0,\tau)$. From now on we denote by $k=k(t,x)$ and $\hat k = \hat k(t,x)$ the curvature of $\ga(t,\cdot)$ and $\hat\ga(t,\cdot)$ at the point $x$ respectively. We are going to prove that if $\bar{\si}$ is small enough, we can bound the Sobolev norm $\|\hat k\|_{W^{m,2}}$ uniformly in time in terms of the initial datum and a suitable constant.
We will deal here only with the case $p>2$. The very same argument can be replicated for the case $p=2$, and the calculations become even simpler. Therefore we also assume that $|k_{\Ga}(x)|\neq0$ for any $x$.

\medskip

We assume that $\bar{\si}>0$ is small enough so that $\|\hat \ga(t,\cdot) -\Ga\|_{C^{3}}$ is so small that
\begin{equation}\label{eq:StimeIpotesi}
|k(t,x)|\ge c>0,
\end{equation}
for any $t\in[0,\tau)$ and $x$, where $c=c(\bar{\si},k_\Ga)$. Observe that therefore the choice of $\bar\si$ only depends on the curvature $k_\Ga$ of $\Ga$. Moreover, by assumptions there exists a bounded neighborhood $\cU=\cU(\bar{\si})$ of $\Ga$ in $M$ such that the flow $\ga$ is contained in $\cU$ for every time. We then let $\La\coloneqq \sup_{x\in\cU} |B_x|$, where $|B|$ is the norm of the second fundamental form of $M\hookrightarrow\R^n$. In the forthcoming constants denoted by the capital letter $C$ we will usually omit the dependence on $\bar{\si}$, $\sE(\Ga)$, $\cU$, $\La$, and a chosen index $m\in\N$.

\medskip

Let us introduce the so called scale invariant norms on $k$. We let
\[
\|k\|_{n,q} \coloneqq \sum_{j=0}^n \|\na^j k \|_q,
\]
where
\[
\|\na^j k \|_q \coloneqq L(\ga(t,\cdot))^{j+1-\frac1q} \left( \int |\na^j k|^q\,ds \right)^{\frac1q}.
\]
As the integrand is a geometric quantity integrated with respect to arclegth, one can verify that
\[
\|k\|_{n,q}= \|\hat k \|_{n,q}.
\]

First we need to recall a few facts about these norms and some interpolation inequalities.

\begin{lemma}\label{lem:StimeKuwert}
	Let $\ga(t,x)$ be as above.
	\begin{enumerate}
		\item\label{it:SobolevEquiv} For any $n\in\N$ and $q\in[1,\infty)$, it holds that
		\begin{equation}\label{eq:SobolevEquiv}
			\frac1C \|k\|_{n,q} \le \|\hat k\|_{W^{n,q}} \le C \|k\|_{n,q}.
		\end{equation}
		\item\label{it:2.17} For any $n\in\N$ it holds that
		\begin{equation}\label{eq:2.17}
		\|k\|_{n,2}^2 \le c(n) (\|\na^nk\|_2^2 + \|k\|_2^2).
		\end{equation}
		\item\label{it:2.14} For any $n\in\N$, $q\ge2$, and $0\le i < n$, it holds that
		\begin{equation}\label{eq:2.14}
			\|\na^i k \|_q \le C(q,n) \|k\|_2^{1-\alpha} \|k\|_{n,2}^\alpha,
		\end{equation}
		with $\alpha=\tfrac1n\left(i+\tfrac12-\tfrac1q\right)$.
		\item\label{it:2.16} Suppose $\nu\ge2$ and $1\le i_1\le ... \le i_\nu\le n-1$ are integers. Let $\mu= \sum_{j=1}^\nu i_j$. If $\mu+\tfrac\nu2 < 2n+1$, then for any $\ep>0$ it holds that
		\begin{equation}\label{eq:2.16}
			\int \Pi_{j=1}^\nu |\na^{i_j} k | \le \ep \int |\na^n k|^2\,ds + C(\ep).
		\end{equation}
	\end{enumerate}
\end{lemma}

\begin{proof}
	We prove the items separately.
		\begin{enumerate}
			\item[(\ref{it:SobolevEquiv})] If $\phi$ is a smooth normal field along $\ga$ on $M$, it holds that
			\begin{equation}\label{eq:DerivateNormali}
			\na^n\phi = \pa_s^n\phi + \sum_{j=0}^{n-1} L_j^n(\pa_s^j \phi),
			\end{equation}
			for any $n$, where $L_j^n$ is a smooth tensor defined on $\cU$, and $\|L_j^n\| \le C(n,\bar{\si},k_\Ga,\cU,\La)$. Indeed the equality immediately follows for $n=1$, and then
			\[
			\na^{n+1}\phi = \pa_s \na^n\phi - \scal{ \pa_s \na^n\phi,\tau}\tau - \scal{ \pa_s \na^n\phi, N_j} N_j  
			=\pa_s \na^n\phi + \scal{  \na^n\phi,\pa_s\tau}\tau - \scal{  \na^n\phi, \pa_sN_j} N_j ,
			\]
			where $\{N_j\}$ is a locally defined orthonormal frame of $M$, and then \eqref{eq:DerivateNormali} follows by induction on $n$.
			
			Now we consider \eqref{eq:DerivateNormali} with $\phi=k$. Observe that for any $n\in\N$ we have
			\[
			\int |\pa_{s_\ga}^n k|^q\,ds_\ga = \int |\pa_{s_{\hat \ga}}^n \hat k|^q\,ds_{\hat \ga}.
			\]
			Since $0<L(\Ga)-\eta\le L(\ga) \le L(\Ga)+\eta$ for $\bar \si$ small, we get that \eqref{eq:SobolevEquiv} follows from \eqref{eq:DerivateNormali} by induction on $n$ using that $\hat\ga$ is parametrized with constant speed.	
			\item[(\ref{it:2.17})] By scaling invariance, we can assume $L(\ga(t,\cdot))=1$. For $n=2$ the inequality follows from $\|\na k\|_2^2 = -\int  \scal{k,\na^2k}\,ds$; for higher $n$ the inequality follows by induction.
			\item[(\ref{it:2.14})] By scaling invariance, we can assume that $L(\ga(t,\cdot))=1$ and also that $\ga(t,\cdot)$ is arclength parametrized. Then since $|\pa_s|\phi||\le |\na \phi|$ for any normal field $\phi$ along $\ga$ on $M$, the standard proof of \cite[Theorem 3.70]{Aubin} applies.
			\item[(\ref{it:2.16})] \Cref{eq:2.16} follows by the very same arguments leading to \cite[(2.16)]{DzKuSc02}, observing that the proof only relies on H\"{o}lder's inequality, on \eqref{eq:2.17}, and on \eqref{eq:2.14}. The constant $C(\ep)$ appearing on the right hand side of \eqref{eq:2.16} depends on $\ep$ and we omitted dependence on $\int |k|^2\,ds$ as it is estimated in terms of $\sE(\Ga)$ and $\bar{\si}$.
		\end{enumerate}
\end{proof}

By \Cref{it:SobolevEquiv} in \Cref{lem:StimeKuwert}, our aim is then to bound the norms $\|\na^m k\|_2$ uniformly in time for any $m\in \N$, as this will imply \Cref{prop:StimaParabolica}.

Let us denote by $V$ the velocity of the flow $\ga$, that is
\[
V = -\na^2(|k|^{p-2}k) - \frac{1}{p'} |k|^pk + k - R(|k|^{p-2}k,\tau)\tau.
\]
We denote $\pa_t^\perp \coloneqq \ga^\perp \pa_t $, where recall that $\ga^\perp = M^\top-\ga^\top$. First of all, we need to derive the evolution equations for the derivatives of the curvature.

\begin{lemma}\label{lem:Evolution}
	Let $\ga(t,x)$ be as above.
	\begin{enumerate}
		\item\label{it:EvolutionK} It holds that
		\begin{equation}\label{eq:EvolutionK}
		\pa_t^\perp k = \na^2 V +\scal{V,k}k +R(V,\tau)\tau.
		\end{equation}
		\item\label{it:Commutatore} For any smooth normal field $\phi$ along $\ga$ on $M$ it holds that
		\begin{equation}\label{eq:Commutatore}
			\pa_t^\perp \na \phi  -  \na \pa_t^\perp \phi = \scal{\phi,k}\na V + \scal{k,V} \na\phi -\scal{\phi,\na V} k + R(V,\tau)\phi.
		\end{equation}
		\item\label{it:EvolutionDerivate} For any $m\in\N$, denoting by $\phi_m=\na^mk$, we have that
		\begin{equation}\label{eq:EvolutionDerivate}
			\begin{split}
				\pa_t^\perp \phi_m + \na^{m+4}(|k|^{p-2}k) 
				&=
				\na^{m+2} \left( - \frac{1}{p'} |k|^pk + k - R(|k|^{p-2}k,\tau)\tau \right) +\\
				&\indent + \na^m \left( \scal{V,k}k + R(V,\tau)\tau \right) 
				+ \sum_{j=0}^{m-1} \na^{m-1-j} (X(\phi_j)),
			\end{split}
		\end{equation}
		where
		\[
		X(\phi) = \scal{\phi,k}\na V + \scal{k,V} \na\phi -\scal{\phi,\na V} k + R(V,\tau)\phi,
		\]
		for any smooth normal field $\phi$ along $\ga$ on $M$.
	\end{enumerate}
\end{lemma}

\begin{proof}
	We prove the items separately.
	\begin{enumerate}
		\item[(\ref{it:EvolutionK})] \Cref{eq:EvolutionK} follows from \eqref{eq:VarKPerp}.
		\item[(\ref{it:Commutatore})] For $\phi,\psi$ normal along $\ga$ on $M$ we compute
		\[
		\begin{split}
			\scal{\pa_t^\perp \na\phi, \psi} 
			&= 
			\scal{\pa_t ( \pa_s \phi + \scal{\phi,k}\tau + \scal{\phi,\pa_s N_j}N_j), \psi},
		\end{split}
		\]
		where $\{N_j\}$ is a locally defined orthonormal frame of $M$. Using that $\pa_t\pa_s = \pa_s\pa_t +\scal{k,V}\pa_s$, writing
		\[
		\pa_t\phi = \pa_t^\perp \phi - \scal{\phi,\pa_t^\perp \tau}\tau - \scal{\phi,\pa_t^\perp N_j}N_j,
		\]
		using that $\pa_t^\perp \tau = \na V$ by \eqref{eq:Vartauperp} and using Gauss equation, one obtains \eqref{eq:Commutatore}.		
		\item[(\ref{it:EvolutionDerivate})] We know that \eqref{eq:EvolutionDerivate} holds for $m=0$ by \eqref{eq:EvolutionK}. The cases $m\ge1$ then easily follow by induction on $m$ using \eqref{eq:Commutatore}.
	\end{enumerate}
\end{proof}

We introduce the following notation. If $\phi_1,...,\phi_r$ are vector fields along $\ga$, we denote by
\[
\phi_1*...*\phi_r,
\]
a generic contraction of the given fields by some tensor whose norm is locally bounded on $M$. The outcome of the contraction may be a vector or a scalar function. In particular we have that $|\phi_1*...*\phi_r| \le C |\phi_1|...|\phi_r|$ on $\cU$ for some constant $C$ clearly depending also on the specific tensor.

We need a few last inequalities and then we will be able to prove the desired bounds using the evolution equations \eqref{eq:EvolutionDerivate}.

\begin{lemma}\label{lem:AltreStime}
	Let $\ga(t,x)$ be as above.
	\begin{enumerate}
		\item\label{it:Bound1}
		Let $n\in \N$ with $n\ge1$ and $j\in\{1,...,n\}$. For any $\ep>0$ it holds that
		\begin{equation}\label{eq:Bound1}
			\int |\na^{n-j} k |^2\,ds \le \ep\|\na^nk\|_2^2 + C(\ep).
		\end{equation}
		\item\label{it:Bound2}
		Let $n\in \N$ with $n\ge1$ and $j\in\{1,...,n\}$. For any $\ep>0$ it holds that
		\begin{equation}\label{eq:Bound2}
		\int |\na^{n-j} k ||\na^n k|\,ds \le \ep\|\na^nk\|_2^2 + C(\ep).
		\end{equation}
		\item\label{it:DerivatePotenzaK}
		For any $q\in\R$ and $L\in \N$ with $L\ge1$ it holds that
		\begin{equation}
		\begin{split}
			\pa_s^L |k|^q &= 
			C_{q,L} |k|^{q-2L}\scal{k,\na k}^L + \sum_{\stackrel{l_1+...+l_\lambda=L}{1\le l_1\le...\le l_\lambda<L}}
			 C_{\bar l}|k|^{q_{\bar l}} k * ... * k * \na^{l_1} k * ... * \na^{l_\lambda}k +\\
				&\indent + q|k|^{q-2}\scal{k,\na^Lk},
		\end{split}
		\end{equation}
		where $C_{q,L}\in\R$ and we denoted by $\bar l$ the array $(l_1,...,l_\lambda)$, and $C_{\bar l}$, $q_{\bar l}$ are some numbers depending on $\bar l$.
	\end{enumerate}
\end{lemma}

\begin{proof}
	We prove the items separately.
	\begin{enumerate}
	\item[(\ref{it:Bound1})] Applying \eqref{eq:2.14} and then \eqref{eq:2.17} we find
	\[
	\begin{split}
		\int|\na^{n-j}k|^2\,ds &\le C \|\na^{n-j}k\|_2^2 \le C \|k\|_2^{2(1-\alpha)}\|k\|_{n,2}^{2\alpha}
		\le C \|k\|_2^{2(1-\alpha)}  ( \|\na^nk\|_2^2 + \|k\|_2^2 )^\alpha \\
		&\le C \|\na^nk\|_2^{2\alpha} + C \le \ep \|\na^nk\|_2^{2} + C(\ep),
	\end{split}
	\]
	where we used that $\alpha=\tfrac{n-j}{n}=1-\tfrac{j}{n}<1$, and then $2\alpha<2$ and we used Young inequality.
	\item[(\ref{it:Bound2})] Arguing as in \ref{it:Bound1} we obtain
	\[
	\begin{split}
		\int |\na^{n-j} k ||\na^n k|\,ds 
		&\le C \|\na^{n-j}k\|_2 \|\na^nk\|_2
		\le C \|k\|_2^{1-\alpha}\|k\|_{n,2}^\alpha   \|\na^nk\|_2 \\
		&\le C (\|\na^nk\|_2^2 + \|k\|_2^2)^{\frac\alpha2} \|\na^nk\|_2 \le C\|\na^nk\|_2 + C\|\na^nk\|_2^{1+\alpha},
	\end{split}
	\]
	where $\alpha=1-\tfrac{j}{n}$, then $1+\alpha<2$, and thus the conclusion follows again by Young inequality.
	\item[(\ref{it:DerivatePotenzaK})] For $L=1$ we have $\pa_s|k|^q = q|k|^{q-2}\scal{k,\na k}$, and then for $L\ge2$ the claim follows by induction.
	\end{enumerate}
\end{proof}

We are ready for estimating the norms $\|\na^mk\|_2$ uniformly in time. Recall that $\|\na k\|_{L^\infty} \le C=C(\bar{\si},\| k_\Ga\|_{W^{2,2}})$ and also $\|\na^2k\|_p \le C=C(\bar{\si},\|\na^2k_\Ga\|_{L^2})$ uniformly in time by assumptions. From now on let
\[
m\ge 3,
\]
and assume by induction that
\[
\|\hat k\|_{W^{m-1,2}} (t)\le C(1+ \|\hat k\|_{W^{m-1,2}} (0)),
\]
for any $t\in[0,\tau)$, where always $C=C(m-1,\bar{\si},\cU,\|k_\Ga\|_{W^{2,2}},\La)$. In particular $|\na^rk|$ is uniformly bounded in $L^\infty$ for any $r=0,...,m-2$, and $\na^{m-1}k \in L^2(ds_\ga)$.

\medskip

Multiplying \eqref{eq:EvolutionDerivate} by $\na^mk$ and integrating we get the evolutions
\begin{equation}\label{eq:EvolutionIntegrals}
\begin{split}
	\pa_t\left(\frac12 \int |\na^mk|^2\,ds\right) &+ \int \scal{\na^{m+2}k, \na^{m+2}(|k|^{p-2}k)} \,ds 
	=\\
	&= \int \left\lgl \na^m k , \na^{m+2} \left( - \frac{1}{p'} |k|^pk + k - R(|k|^{p-2}k,\tau)\tau \right) \right\rgl \,ds
	- \frac12 \int |\nabla^m k|^2 \scal{V,k} \, ds +\\
	&\indent + \int \left\lgl \na^mk
	,
	\na^m \left( \scal{V,k}k + R(V,\tau)\tau \right) 
	+ \sum_{j=0}^{m-1} \na^{m-1-j} (X(\na^jk))
	\right\rgl \,ds,
\end{split}
\end{equation}
in the notation of \eqref{eq:EvolutionDerivate}. We now estimate each term in \eqref{eq:EvolutionIntegrals}. We are going to see that, once the second summand on the left hand side is correctly estimated, by the same arguments also the remaining terms will be controlled in the right way.

\medskip

We have
\begin{equation}\label{eq:TermineSinistraCompleto}
\begin{split}
	\int \scal{\na^{m+2}k &, \na^{m+2}(|k|^{p-2}k)}
	=
	\sum_{j=0}^{m+2} \begin{pmatrix}m+2 \\ j\end{pmatrix}
	 \int  \scal{\na^{m+2}k, \pa_s^j(|k|^{p-2}) \na^{m+2-j}k} \\
	 &= \int |k|^{p-2} |\na^{m+2}k|^2 +\sum_{j=1}^{m+2} \begin{pmatrix}m+2 \\ j\end{pmatrix}
	 \int  \scal{\na^{m+2}k, \pa_s^j(|k|^{p-2}) \na^{m+2-j}k}\\
	 &\ge C \int |\na^{m+2}k|^2 
	 +\sum_{j=1}^{m+2} \begin{pmatrix}m+2 \\ j\end{pmatrix}
	 \int \scal{\na^{m+2}k, \pa_s^j(|k|^{p-2}) \na^{m+2-j}k}.
\end{split}
\end{equation}
Let $j\in\{1,...,m+2\}$. By \Cref{lem:AltreStime} we find
\begin{equation}\label{eq:TermineSinistraJ}
\begin{split}
	\int \scal{\na^{m+2}k&, \pa_s^j(|k|^{p-2}) \na^{m+2-j}k} 
	= C_{p-2,j} \int |k|^{p-2-2j}\scal{k,\na k}^j\scal{\na^{m+2}k, \na^{m+2-j}k} +\\
	&\indent + \sum_{\stackrel{l_1+...+l_\lambda=j}{1\le l_1\le...\le l_\lambda<j}}
	C_{\bar l} \int |k|^{q_{\bar l}}k*...*k \na^{l_1}k*...*\na^{l_\lambda}k
	 \scal{\na^{m+2}k, \na^{m+2-j}k} +\\
	&\indent + (p-2)\int |k|^{p-4}\scal{k,\na^jk}\scal{\na^{m+2}k, \na^{m+2-j}k}.
\end{split}
\end{equation}
We study separately the case $j=m+2$. So let $j=m+2$. In this case
\[
\begin{split}
\int \scal{\na^{m+2}k&, \pa_s^{m+2}(|k|^{p-2}) k} 
= C_{p-2,m+2} \int |k|^{p-2-2(m+2)}\scal{k,\na k}^{m+2}\scal{\na^{m+2}k, k} +\\
&\indent + \sum_{\stackrel{l_1+...+l_\lambda=m+2}{1\le l_1\le...\le l_\lambda<m+2}}
C_{\bar l} \int |k|^{q_{\bar l}}k*...*k \na^{l_1}k*...*\na^{l_\lambda}k
\scal{\na^{m+2}k, k} + \\
&\indent + (p-2)\int |k|^{p-4}\scal{\na^{m+2}k, k}^2 \\
&\ge 
C_{p-2,m+2} \int |k|^{p-2-2(m+2)}\scal{k,\na k}^{m+2}\scal{\na^{m+2}k, k} +\\
&\indent + C_{(1,m+1)} \int |k|^{q_{(1,m+1)}} k*...*k \na k* \na^{m+1}k \scal{\na^{m+2}k, k} + \\
&\indent  + \sum_{\stackrel{l_1+...+l_\lambda=m+2}{1\le l_1\le...\le l_\lambda<m+1}}
C_{\bar l} \int |k|^{q_{\bar l}}k*...*k \na^{l_1}k*...*\na^{l_\lambda}k
\scal{\na^{m+2}k, k}.
\end{split}
\]
Clearly
\[
\left| C_{p-2,m+2} \int |k|^{p-2-2(m+2)}\scal{k,\na k}^{m+2}\scal{\na^{m+2}k, k} \right| \le \ep\int|\na^{m+2}k|^2 + C(\ep),
\]
\[
\left|C_{(1,m+1)} \int |k|^{q_{(1,m+1)}} k*...*k \na k* \na^{m+1}k \scal{\na^{m+2}k, k}\right| \le \ep\int|\na^{m+2}k|^2 + C(\ep).
\]
We want to prove the very same estimate for the generic term
\begin{equation}\label{eq:GenericTerm}
	C_{\bar l} \int |k|^{q_{\bar l}}k*...*k \na^{l_1}k*...*\na^{l_\lambda}k
	*\na^{m+2}k,
\end{equation}
with $1\le l_1\le...\le l_\lambda<m+1$ and $l_1+...+l_\lambda=m+2$. We divide two cases.

\begin{itemize}
	\item Suppose that writing $l_1\le...\le l_{i-1} < m-1 \le l_i \le...\le l_\lambda\le m$ it occurs that $\sum_{r=i}^\lambda l_r = m+2$, or suppose that there is some $s$ such that $l_s=m-2$.
	
	In the first situation, since $\sharp\{l_r\ge m-1\} \le 1 + \frac{3}{m-1}$, if $m>4$ then $\sharp\{l_r\ge m-1\} \le 1$, and it is impossible to satisfy that $\sum_{r=i}^\lambda l_r = m+2$. If instead $m=3$, then it must be that $l_\lambda=3,l_{\lambda-1}=2$, while if $m=4$, then it must occur that $l_\lambda=l_{\lambda-1}=3$. We can handle these two cases individually:
	\begin{equation}\label{eq:Conto}
	\begin{split}
		\int |\na^2k||\na^3k||\na^5k| &\le C\|\na^2k\|_p \|\na^3k\|_q\|\na^5k\|_2 \le C\|\na^5 k\|_2(\|\na^5k\|_2^\alpha + C) \\
		&\le \ep\int|\na^5k|^2 + C(\ep),
	\end{split}
	\end{equation}
	where $\tfrac12 + \tfrac1p +\tfrac1q =1$, and then $q>2$, and $\alpha=\tfrac15(3+\tfrac12 -\tfrac1q)=\tfrac35+\tfrac{1}{5p}$ by \Cref{lem:StimeKuwert}, so that $1+\alpha<2$;
	\[
	\begin{split}
		\int |\na^3k|^2|\na^6k| 
		&\le C(\|\na^3k\|_4^4)^{\frac12}\|\na^6 k\|_2 
		\le \left( \|\na^3k\|_2^{2(1-s)} \|\na^3k\|_q^{sq} \right)^{\frac12} \|\na^6k\|_2 \\
		&\le  \|\na^3k\|_2^{(1-s)} ( C+ \|\na^6k\|_2^\alpha )^{\frac{sq}{2}} \|\na^6k\|_2 \le C\|\na^6k\|_2 + \|\na^6k\|_2^{1+\frac{\alpha s q}{2}},
	\end{split}
	\]
	where we used that $\|\na^3k\|_2$ is uniformly bounded by induction, $4=2(1-s)+qs$, and $\alpha=\tfrac16(3+\tfrac12-\tfrac1q)$ by \Cref{lem:StimeKuwert}. Since $sq=2+2s$, for $s$ small we get $\tfrac{\alpha s q}{2} = \alpha(1+s) < (\tfrac12 +\tfrac{1}{12})(1+s) <1 $, and we conclude the desired estimate by Young inequality.
	
	In the other situation, when there is some $l_s=m-2$, we see that if $\sharp\{l_r\ge m-1\}<2$ then $\{l_r\ge m-1\}$ is either empty or only contains $l_\lambda$, and thus by induction the integral is estimated by
	\[
	C \int |\na^{l_\lambda}k||\na^{m+2}k|, \qquad\mbox{ or } \qquad C\int |\na^{m+2}k|,
	\]
	and we can apply \Cref{lem:AltreStime}. If instead $\sharp\{l_r\ge m-1\}\ge 2$, then, from $\sum_1^\lambda l_r = m+2$, we are in the case $m=3$, and one estimates a term like
	\[
	C \int |\na^3k||\na^5k|,
	\]
	by \Cref{lem:AltreStime}, or like
	\[
	C \int |\na^2k|^2|\na^5k|,
	\]
	in the same way we treated \eqref{eq:Conto}.
	\item In this second case we have that writing $l_1\le...\le l_{i-1} < m-1 \le l_i \le...\le l_\lambda\le m$ it occurs that $\sum_{r=i}^\lambda l_r <m+2$, and $l_r\neq m-2$ for any $r$.
	
	In this case we first integrate by parts in \eqref{eq:GenericTerm} once, and thus we get something of the form
	\[
	\begin{split}
		\left|C_{\bar l} \int |k|^{q_{\bar l}}k*...*k \na^{l_1}k*...*\na^{l_\lambda}k
		*\na^{m+2}k \right| 
		&\le
		C \int \Pi_{r=1}^\lambda |\na^{j_r}k| \,\, |\na^{m+1}k| \\
		& \le C \int \Pi_{r=i}^\lambda |\na^{j_r}k| \,\, |\na^{m+1}k|,
	\end{split}
	\]
	where we wrote $1\le j_1\le  ... j_{i-1}< m-1 \le j_i \le ... \le j_\lambda \le m+1$, where we used that since $l_r\neq m-2$ for any $r$, then $\sharp\{l_r\ge m-1\}=\sharp \{j_r\ge m-1\}$. It follows that $\sum_{r=i}^\lambda j_r \le m+2$, and then $\sharp\{j_r \ge m-1\} \le \tfrac{m+2}{m-1}$. Therefore the desired estimate follows by applying \eqref{eq:2.16} with
	\[
	\mu = m+1 +\sum_{r=i}^\lambda j_r \le 2m+3, \qquad
	\nu = 1 + 	\sharp\{j_r \ge m-1\} \le \frac{2m+1}{m-1}, \qquad
	n= m+2,
	\]
	observing that $\mu+\tfrac\nu2 \le 2m+3 + \tfrac12 \frac{2m+1}{m-1} < 2(m+2) +1$ for any $m\ge 3$.
\end{itemize}

\medskip

Consider now the remaining cases of $j\in\{1,...,m+1\}$ in \eqref{eq:TermineSinistraJ}. For $j=1$, by \Cref{lem:AltreStime}, we get
\[
\left|\int \scal{\na^{m+2}k, \pa_s(|k|^{p-2}) \na^{m+1}k} \right| 
\le C \int |\na^{m+2}k||\na^{m+1}k| \le \ep \int |\na^{m+2}k|^2 + C(\ep).
\]
For $j\in\{2,...,m+1\}$ we see that \eqref{eq:TermineSinistraJ} can be rewritten
\[
\begin{split}
\int \scal{\na^{m+2}k&, \pa_s^j(|k|^{p-2}) \na^{m+2-j}k} 
= C_{p-2,j} \int |k|^{p-2-2j}\scal{k,\na k}^j\scal{\na^{m+2}k, \na^{m+2-j}k} +\\
&\indent + \sum_{\stackrel{l_1+...+l_\lambda=m+2}{1\le l_1\le...\le l_\lambda<m+1}}
C_{\bar l} \int |k|^{q_{\bar l}}k*...*k \na^{l_1}k*...*\na^{l_\lambda}k *
\na^{m+2}k +\\
&\indent + (p-2)\int |k|^{p-4}\scal{k,\na^jk}\scal{\na^{m+2}k, \na^{m+2-j}k},
\end{split}
\]
and we have that
$\left| C_{p-2,j} \int |k|^{p-2-2j}\scal{k,\na k}^j\scal{\na^{m+2}k, \na^{m+2-j}k}  \right| \le \ep\int |\na^{m+2}k|^2 + C(\ep)$ by \Cref{lem:AltreStime}, the terms in the sum have exactly the form of the generic term \eqref{eq:GenericTerm} we studied above, and thus they can be estimated as desired, and finally the last summand
\[
(p-2)\int |k|^{p-4}\scal{k,\na^jk}\scal{\na^{m+2}k, \na^{m+2-j}k},
\]
is estimated by \Cref{lem:AltreStime} if $j=m+1$, otherwise it is again of the form \eqref{eq:GenericTerm}.

Putting together the estimates we found, coming back to \eqref{eq:TermineSinistraCompleto}, we get that the estimate becomes
\begin{equation*}
\begin{split}
C(\ep) + \ep \int |\na^{m+2}k|^2 +\int \scal{\na^{m+2}k &, \na^{m+2}(|k|^{p-2}k)}
\ge  C \int |\na^{m+2}k|^2,
\end{split}
\end{equation*}
and choosing $\ep>0$ sufficiently small this becomes
\begin{equation*}
	C_1  \int |\na^{m+2}k|^2 \le \int \scal{\na^{m+2}k , \na^{m+2}(|k|^{p-2}k)} + C_2.
\end{equation*}

\medskip

Finally it is easy to check that the terms on the right hand side of \eqref{eq:EvolutionIntegrals} can be estimated by analogous quantities. More precisely, after integration by parts, we need to estimate
\[
\int \left\lgl \na^{m+2} k , \na^m \left( - \frac{1}{p'} |k|^pk + k - R(|k|^{p-2}k,\tau)\tau \right) \right\rgl ,
\]
and
\[
\int \left\lgl \na^mk , \na^m \left( \scal{V,k}k + R(V,\tau)\tau \right) 
+ \sum_{j=0}^{m-1} \na^{m-1-j} (X(\na^jk))
\right\rgl .
\]
Noticing that in a term of the form $ \na^{m-1-j} (X(\na^jk))$ the sum of the orders of the derivatives of $k$ is $m+2$, one can check that the above two integrals can be estimated by exactly the same arguments employed before. Moreover, also the remained integral on the right hand side of \eqref{eq:EvolutionIntegrals}, namely $\int |\nabla^m k|^2 \scal{V,k} \, ds$, can be easily estimated using \Cref{lem:StimeKuwert}.

In the end \eqref{eq:EvolutionIntegrals} becomes
\[
\begin{split}
\pa_t\left(\frac12 \int |\na^mk|^2\,ds\right) &+ C_1  \int |\na^{m+2}k|^2 \le\\
&\le \pa_t\left(\frac12 \int |\na^mk|^2\,ds\right)
+\int \scal{\na^{m+2}k , \na^{m+2}(|k|^{p-2}k)} + C_2 \\
& \le \ep\int |\na^{m+2}k|^2   + C(\ep) + C_2,
\end{split}
\]
and taking $\ep>0$ sufficiently small we get the estimate
\[
\pa_t\left(\frac12 \int |\na^mk|^2\,ds\right) + C_3  \int |\na^{m+2}k|^2 \le  C_4.
\]
By \eqref{eq:2.17} one finally obtains
\[
\pa_t\left( \int |\na^mk|^2\,ds\right) + C_5  \int |\na^{m}k|^2 \le  C_6,
\]
and by comparison we get the desired bound
\[
\|\na^mk\|_2^2(t) \le \|\na^mk\|_2^2(0) + C,
\]
for any $t\in[0,\tau)$.

\section{An example of a non-converging flow}

In this appendix we construct an example of an analytic complete $2$-dimensional submanifold $M$ of $\R^3$ and of a solution to the gradient flow of the elastic energy with exponent $p=2$ such that it does not converge. Indeed, the resulting flow will escape from any compact set of $M$ for times sufficiently large. By \Cref{thm:MainIntro} such a flow cannot even subconverge in the sense of the hypotheses of \Cref{thm:MainIntro}.

\medskip

In $\R^3=\{(x,y,z)\,:\,x,y,z\in\R\}$, consider the curve
\[
\si:(0,+\infty)\to\R^3, \qquad \si(t)=(f(t),0,t),
\]
that parametrizes the graph of a function $f\colon (0,+\infty)\to\R$ on the plane $\{y=0\}$, and assume that $f$ is analytic and $f(t)>0$ for any $t>0$. Consider the surface of revolution about the $z$-axis of the curve $\si$, that is, the analytic surface parametrized by the immersion
\[
F(\te,t) = \left( f(t)\cos \te, f(t) \sin \te, t \right),
\]
for $t>0$ and $\te\in[0,2\pi]$.
We denote by $M$ such a complete analytic submanifold of $\R^3$. Consider $t_0>0$ and the closed curve $\ga_0$ given by the intersection $M\cap \{z=t_0\}$. Letting
\[
f(t)=1+\frac1t,
\]
we want to show that if $t_0$ is sufficiently big, the resulting solution $\ga$ of the gradient flow of the elastic energy with exponent $p=2$ starting from $\ga_0$ does not converge, and, in fact, it escapes from any compact set of $M$ for times sufficiently large.

\medskip

As the manifold $M$ lacks of a ``good'' family of isometries that could satisfy the hypotheses of \Cref{thm:ConvergenceManifoldsMAIN}, it may be not too surprising that the flow does not converge. However, we notice that $(M,g)$ is a manifold of bounded geometry, that is, the Riemann curvature tensor, that we will compute later, is uniformly pointwise bounded together with all its derivatives and the injectivity radius of $(M,g)$ is strictly positive. Therefore, this example shows that not even such hypotheses are sufficient for the convergence of the flow. 

\medskip

Since $M$ is a surface of revolution, its Gaussian curvature $K$ can be computed in terms of $f$ and it equals
\[
K(\te,t)= \frac{-f''(t)}{f(t)(1+(f'(t))^2)^2}.
\]
Hence, as $M$ is a $2$-dimensional, the Riemann tensor is given by (see \cite{DoCarmo})
\[
R(X,Y)Z = K ( \scal{Y,Z}X - \scal{X,Z}Y  ),
\]
for any tangent vectors $X,Y,Z$. Moreover, a unit normal field along $M$ is given by
\[
\nu (\te,t) = \frac{1}{\sqrt{1+(f'(t))^2}} \left( -\cos\te, - \sin \te, f'(t) \right).
\]

By the rotational symmetry of $M$ and the choice of $\ga_0$, the flow $\ga$ is of the form $\ga(t,\S^1)= M \cap \{z=G(t)\}$ for some function $G$ (this also follows from the explicit computation of the driving velocity of the flow computed below, which is a multiple of the curvature of the evolving curve). Hence the desired conclusion follows once we prove the following two facts.

\begin{enumerate}
	\item \label{it:Example1}
	There exists $\tau_0>0$ such that for any $\tau\ge\tau_0$, the curve $\alpha(\te)=(f(\tau)\cos\te,f(\tau)\sin\te,\tau)$ is not a critical point of the energy.
	
	\item \label{it:Example2}
	There exists $\tau_0>0$ such that for any $\tau\ge\tau_0$, letting $\alpha(\te)=(f(\tau)\cos\te,f(\tau)\sin\te,\tau)$, we have that $-\na_{T(\alpha)^2,T(\al)^2} E(0)$ is a positive multiple of the curvature $k_\al$ of $\alpha$, and the third component of $k_\alpha$ is positive.
\end{enumerate}

Both the above items will follow from the direct calculation of $\na_{T(\alpha)^2,T(\al)^2} E(0)$, that is, from the computation of the first variation at the curve $\alpha$.

We have that $|\alpha'|= f(\tau)$, and then
\[
\begin{split}
	k_\alpha &= M^\top(\pa_{s_\alpha}^2 \al)
	= -\frac{1}{f(\tau)}\left[ \left( \cos \te, \sin \te, 0 \right) - \left\langle \left( \cos \te, \sin \te, 0 \right), \nu \right\rangle \nu
	\right] \\
	&= -\frac{1}{f(\tau)}\left[ \frac{(f'(\tau))^2}{1+(f'(\tau))^2}\left( \cos \te, \sin \te, 0 \right)
	+ \left(  0,0, \frac{f'(\tau)}{1+(f'(\tau))^2} \right)
	\right],
\end{split}
\]
and we see that the third component of $k_\alpha$ is always strictly positive, indeed $f'(\tau)=-\tfrac{1}{\tau^2}$.
Denoting as usual $\al^\perp\coloneqq M^\top - \alpha^\top$, we compute
\[
\begin{split}
	\na k_\al &= \al^\perp(\pa_{s_\alpha}k_\al) = \frac{1}{f(\tau)} \al^\perp(\pa_\te k_\al) 
	= \frac{1}{f(\tau)}   \frac{(f'(\tau))^2}{1+(f'(\tau))^2} \al^\perp\left( -\sin \te, \cos \te, 0 \right) 
	=0,
\end{split}
\]
and then also $\na^2 k_\al=0$. Therefore
\[
\begin{split}
	\frac{1}{|\alpha'|}\na_{T(\alpha)^2,T(\al)^2} E(0) 
	&= \na^2 k_\al + \frac12 |k_\alpha|^2 k_\al - k_\al + R(k_\al,\tau_\al)\tau_\al  
	= \frac12 |k_\alpha|^2 k_\al - k_\al + K k_\al \\
	&= \left( \frac12 |k_\alpha|^2 + K -1 \right) k_\al 
	= \left(
	\frac12 \frac{(f'(\tau))^2}{f^2(\tau)(1+(f'(\tau))^2)} 
	- \frac{f''(\tau)}{f(\tau)(1+(f'(\tau))^2)^2} -1	
	\right) k_\al \\
	&= 
	\left(  \frac12 (f'(\tau))^2(1+(f'(\tau))^2) - f(\tau)f''(\tau) - f^2(\tau)(1+(f'(\tau))^2)^2  \right)
	\frac{k_\al}{f^2(\tau)(1+(f'(\tau))^2)^2} \\
	&= \big( -1 + O(\tau^{-1}) \big) \frac{k_\al}{f^2(\tau)(1+(f'(\tau))^2)^2}.
\end{split}
\]
From the above computation we see that for $\tau_0$ sufficiently large and $\tau\ge\tau_0$, we have that $\na_{T(\alpha)^2,T(\al)^2} E(0) \neq 0$, and then $\al$ is not a critical point, and $-\na_{T(\alpha)^2,T(\al)^2} E(0) $ is a positive multiple of $k_\al$.

We deduce that if $t_0\ge\tau_0$, then the flow $\ga$ does not remain in a bounded subset of $M$, and, in fact, it sweeps the set $M\cap \{ z\ge t_0 \}$.

\bigskip


\end{document}